\documentclass[12pt,a4paper,psamsfonts]{amsart}
\usepackage{amssymb,amscd,amsxtra,calc}
\usepackage{cmmib57}
\usepackage[dvips]{graphicx}
\usepackage[dvips,a4paper,bookmarks,bookmarksopen,%
bookmarksnumbered,%
colorlinks=false]{hyperref}

\theoremstyle{plain}
    \newtheorem{thm}{Theorem}[section]
    \newtheorem*{mainthm}{Main Theorem}
    \newtheorem{lem}[thm]{Lemma}%
    \newtheorem{cor}[thm]{Corollary}
    \newtheorem{prop}[thm]{Proposition}
    \newtheorem{sub}{Sub}[thm]
    \newtheorem{lemsub}[sub]{Lemma}
    \newtheorem{corsub}[sub]{Corollary}

\theoremstyle{definition}
    \newtheorem{dfn}[thm]{Definition}
    \newtheorem*{notan}{Notation}

\theoremstyle{remark}
    \newtheorem{rem}[thm]{Remark}
    \newtheorem{remsub}[sub]{Remark}
    \newtheorem*{remn}{Remark}
    \newtheorem{exam}[thm]{Example}

    \newtheorem{clsub}[sub]{Claim}
    \newtheorem{step}{Step}%

\numberwithin{equation}{section}

\DeclareSymbolFont{stmry}{U}{stmry}{m}{n}
\DeclareMathDelimiter\llbracket{\mathopen}{stmry}{"4A}%
					  {stmry}{"71}
\DeclareMathDelimiter\rrbracket{\mathclose}{stmry}{"4B}%
					   {stmry}{"79}

\newcommand{\BAA}{\mathbb{A}}
\newcommand{\BCC}{\mathbb{C}}

\newcommand{\BKK}{\mathbb{K}}
\newcommand{\BLL}{\mathbb{L}}
\newcommand{\BPP}{\mathbb{P}}
\newcommand{\BQQ}{\mathbb{Q}}
\newcommand{\BRR}{\mathbb{R}}
\newcommand{\BSS}{\mathbb{S}}
\newcommand{\BTT}{\mathbb{T}}
\newcommand{\BZZ}{\mathbb{Z}}

\newcommand{\SA}{\mathcal{A}}
\newcommand{\SC}{\mathcal{C}}
\newcommand{\SF}{\mathcal{F}}
\newcommand{\SG}{\mathcal{G}}
\newcommand{\SH}{\mathcal{H}}
\newcommand{\SI}{\mathcal{I}}
\newcommand{\SL}{\mathcal{L}}
\newcommand{\SO}{\mathcal{O}}
\newcommand{\ST}{\mathcal{T}}
\newcommand{\SU}{\mathcal{U}}
\newcommand{\SV}{\mathcal{V}}
\newcommand{\SX}{\mathcal{X}}
\newcommand{\SY}{\mathcal{Y}}

\newcommand{\ep}{\varepsilon}

\newcommand{\GM}{\mathfrak{m}}
\newcommand{\GR}{\mathfrak{R}}
\newcommand{\GX}{\mathfrak{X}}

\newcommand{\stt}{\mathtt{s}}
\newcommand{\ttt}{\mathtt{t}}
\newcommand{\utt}{\mathtt{u}}
\newcommand{\vtt}{\mathtt{v}}
\newcommand{\xtt}{\mathtt{x}}
\newcommand{\ytt}{\mathtt{y}}
\newcommand{\ztt}{\mathtt{z}}

\newcommand{\SAA}{\mathsf{A}}
\newcommand{\SKK}{\mathsf{K}}
\newcommand{\MM}{\mathsf{M}}
\newcommand{\NN}{\mathsf{N}}

\newcommand{\Bmu}{\boldsymbol{\mu}}
\newcommand{\Bsigma}{\boldsymbol{\sigma}}
\newcommand{\BSigma}{\boldsymbol{\Sigma}}
\newcommand{\Btau}{\boldsymbol{\tau}}

\newcommand{\Be}{\boldsymbol{e}}
\newcommand{\Bt}{\boldsymbol{t}}
\newcommand{\bzero}{\boldsymbol{0}}

\newcommand{\chara}{\operatorname{char}}
\newcommand{\Codim}{\operatorname{codim}}
\newcommand{\Cone}{\operatorname{Cone}}
\newcommand{\dd}{\mathrm{d}}
\newcommand{\Def}{\operatorname{Def}}

\newcommand{\Div}{\operatorname{div}}
\newcommand{\DNA}{\operatorname{DNA}}
\newcommand{\depth}{\operatorname{depth}}
\newcommand{\ev}{\operatorname{ev}}
\newcommand{\SExt}{\operatorname{\mathcal{E}\mathit{xt}}}
\newcommand{\BExt}{\operatorname{\mathbb{E}xt}}

\newcommand{\Hom}{\operatorname{Hom}}
\newcommand{\RHom}{\operatorname{RHom}}
\newcommand{\SHom}{\operatorname{\mathcal{H}\mathit{om}}}
\newcommand{\SRHom}{\operatorname{\mathcal{RH}\mathit{om}}}
\newcommand{\id}{\operatorname{id}}
\newcommand{\Ker}{\operatorname{Ker}}
\newcommand{\LC}{\operatorname{LC}}
\newcommand{\mult}{\operatorname{mult}}
\newcommand{\ob}{\operatorname{ob}}
\newcommand{\OH}{\operatorname{H}}
\newcommand{\OR}{\operatorname{R}}
\newcommand{\OT}{\operatorname{T}}
\newcommand{\Pic}{\operatorname{Pic}}
\newcommand{\Proj}{\operatorname{Proj}}
\newcommand{\Res}{\operatorname{Res}}
\newcommand{\Sets}{\operatorname{Sets}}
\newcommand{\Sing}{\operatorname{Sing}}
\newcommand{\Spec}{\operatorname{Spec}}
\newcommand{\Spf}{\operatorname{Spf}}
\newcommand{\Supp}{\operatorname{Supp}}

\newcommand{\isom}{\simeq}
\newcommand{\injmap}{\hookrightarrow}
\newcommand{\surjmap}{\twoheadrightarrow}

\newcommand{\alg}{\mathrm{alg}}

\newcommand{\loc}{\mathrm{loc}}

\begin{document}

\title[]{Simply connected surfaces of general type
in positive characteristic
via deformation theory}
\author{Yongnam Lee}
\address[Yongnam Lee]{%
\textsc{Department of Mathematics} \endgraf
\textsc{Sogang University, Sinsu-dong, Mapo-gu, Seoul 121-742 Korea}}
\email{ynlee@sogang.ac.kr}
\author{Noboru Nakayama}
\address[Noboru Nakayama]{%
\textsc{Research Institute for Mathematical Sciences} \endgraf
\textsc{Kyoto University, Kyoto 606-8502 Japan}}
\email{nakayama@kurims.kyoto-u.ac.jp}

\begin{abstract}
Algebraically simply connected surfaces of general type with \(
p_g=q=0\) and \( 1\le K^2\le 4\) in positive characteristic (with
one exception in \( K^2=4\)) are presented by using a \(
\BQQ\)-Gorenstein smoothing of two-dimensional toric singularities,
a generalization of Lee--Park's construction \cite{LeePark}
to the positive characteristic case,
and Grothendieck's specialization theorem for the
fundamental group.
\end{abstract}


\thanks{
The first named author is partly supported by the WCU Grant funded
by the Korean Government (R33-2008-000-10101-0), and by the National Research
Foundation of Korea(NRF) grant funded by the Korea government(MEST)
(No. 2010-0008752). The second named
author is partly supported by the Grant-in-Aid for Scientific
Research (C), Japan Society for the Promotion of Science.}


\maketitle

\thispagestyle{empty}

\setcounter{tocdepth}{1}
\tableofcontents

\section{Introduction}
\label{sect:intro}

One of the interesting problems in the classification of algebraic
surfaces is to find a new family of simply connected surfaces of
general type with geometric genus \( p_g=0\). Surfaces with \( p_g=0\) are
interesting in view of Castelnuovo's criterion:
\emph{An irrational
surface with irregularity \( q=0\) must have bigenus \( P_2\ge 1\)}.
Simply connected
surfaces of general type with \( p_g=0\) are little known and the
classification is still open.

When a surface is defined over the field \( \BCC\) of complex
numbers, the only known simply connected minimal surfaces of general
type with \( p_g=0\) were Barlow surfaces \cite{Barlow} until 2006.
The canonical divisor of Barlow surfaces satisfies \( K^2=1\).
Recently, the first named author and J.~Park \cite{LeePark} have
constructed a simply connected minimal surface of general type with
\( p_g=0\) and \( K^2=2\)  by using a \( \BQQ\)-Gorenstein smoothing
and Milnor fiber of a smoothing (or rational blow-down surgery).
When a surface is defined over a field of positive characteristic,
the existence of algebraically simply connected minimal surface of
general type with \( p_g=0\) is known only for some special
characteristics. Lang \cite{Lang81} has constructed surfaces of
general type with \( p_g=0\) and \( K^2=1\) in characteristic \( 5
\). Ekedahl \cite{Ek} has given examples of surfaces of general type
with \( p_g=0\) and with \( 1\le K^2\le 9\) in characteristic two.
The inequality
\( 1 \leq K^2_{S} \leq 9\) holds by results of Shepherd-Barron \cite{SB1},
\cite{SB2} when \( S\) is a minimal surface of general type with \(
p_g=0\). This is shown in Liedtke's lecture notes \cite{Lied} on algebraic surfaces in positive characteristic.

We shall construct such a surface of general type
defined over an algebraically closed field of any characteristic
applying the Lee--Park construction given in \cite{LeePark}.
The following is our main result:

\begin{mainthm}
For any algebraically closed field \( \Bbbk \) and for any integer
\( 1 \leq K^{2} \leq 4 \), there exists an algebraically simply
connected minimal surface \( \BSS \)
of general type over \( \Bbbk \) with
\(p_{g}(\BSS) = q(\BSS) = \dim \OH^{2}(\BSS, \Theta_{\BSS/\Bbbk}) = 0 \)
and \( K_{\BSS}^{2} = K^{2} \) except \( (\chara(\Bbbk), K^{2}) = (2, 4) \),
where \( \Theta_{\BSS/\Bbbk} \) denotes the tangent sheaf.
Moreover, one can find such a surface with ample
canonical divisor when \( 1 \leq K^{2} \leq 4 \), except
\((\chara(\Bbbk), K^{2}) = (2, 1)\), \((2, 2)\), and \( (2, 4)\).
\end{mainthm}

\begin{remn}
\begin{enumerate}
\item  The surface \( \BSS \) in Main Theorem is liftable to characteristic zero
since \( \OH^{2}(\BSS, \Theta_{\BSS/\Bbbk}) = \OH^{2}(\BSS, \SO_{\BSS}) = 0 \).

\item
In our method of constructing \( \BSS \),
the bound \( K^{2} \leq 4 \) is necessary (cf.\ Remark~\ref{rem:K2bound} below).

\item  The ampleness property in Main Theorem
is known in characteristic zero by \cite[Section~2]{RS} and
the proof of \cite[Theorem~4.1]{PPS2} in
the case of \( K^{2} = 3 \) and \( 4 \).

\item We have the exceptional cases of \( (\chara(\Bbbk), K^{2}) \) by
the lack of certain examples of rational surfaces \( X \) in Section~\ref{sect:Global}.
This is a technical reason, and
the existence of such a surface \( \BSS \) in the exceptional cases is an open problem.
\end{enumerate}
\end{remn}

The Lee--Park construction is as follows in the case of
\( K^{2} = 2 \) (cf.\ \cite[Section~3]{LeePark}):
First, we consider a special pencil of
cubics in \( \BPP^2\) and blow up many times to get a projective
surface \( M \) (\( \tilde Z\) in \cite{LeePark}) which contains a
disjoint union of five linear chains of smooth rational curves
representing the resolution graphs of special quotient singularities
called of class T (cf.\ Definition~\ref{dfn:classT} below). Then, we
contract these linear chains of rational curves from the surface
\(M\) to produce a projective surface \( X\) with five special
quotient singularities of class T and with \( K_{X}^{2} = 2 \).
We can prove the existence of a
global \(\BQQ\)-Gorenstein smoothing of the singular surface \( X\)
(cf.\ \cite[Theorem~4]{LeePark}), in which a general fiber \( X_t\)
of the \( \BQQ\)-Gorenstein smoothing is a simply connected minimal
surface of general type with \( p_g=0\) and \( K^2 =2\) (cf.\
\cite[Proposition~8]{LeePark}). The method of Lee--Park construction
works to other types of rational elliptic surfaces, which are used to
construct a simply connected minimal surface of general type with \(
p_g=0\) and with \( K^2=1\), \(3\), or \(4\)
(\cite[Section~7]{LeePark}, \cite{PPS1}, \cite{PPS2}). We shall show
that the Lee--Park construction of singular surfaces also
works in positive characteristic, but several key parts in the proof
to show the existence of a global \(\BQQ\)-Gorenstein smoothing
should be modified.

Over the field \( \BCC \) of complex numbers, the existence of a local
\(\BQQ\)-Gorenstein smoothing of a singularity of class T
is given by the index-one cover (cf.\ \cite[Proposition 5.9]{LW},
\cite[Proposition 3.10]{KSh}). The key idea in \cite{LeePark} to
show the vanishing of the obstruction space for a global \(
\BQQ\)-Gorenstein smoothing is to use the lifting property of
derivations of normal surface to its minimal resolutions, the
tautness of the quotient singularities, and the special
configurations of resolution graphs of singular points.

In characteristic \( 0 \), the tautness holds for quotient singularities
(cf.\ \cite[Satz~2.10]{Brieskorn}), i.e, the minimal resolution graph of a
quotient singularity determines the type of singularity.
It is known in characteristic \( 0 \) that the tautness
is equivalent to \( H^1(\Theta_D) = 0\) for any ``sufficiently large''
effective divisor \( D\) supported on
the exceptional divisors on its minimal resolution
(cf.\ \cite[Theorem~3.10]{LauferDef}, \cite[Section~2]{Laufer}).
However, tautness does not hold in characteristic \( p > 0 \) in general.
Indeed, we have some examples of rational double points
when \( p = 2 \), \( 3 \), and \( 5 \), by
Artin's classification in \cite{ArtinRDP}.
The lifting property of derivations from a normal surface to its
minimal resolution exists in characteristic \( 0 \)
(\cite[Proposition 1.2]{BW}),
but this is not true in characteristic \( p > 0 \)
(cf.\ \cite[Example, pp.~345]{Artin74} and
Proposition~\ref{prop:tangentsheaves}.\eqref{prop:tangentsheaves:reflexive}
below). In \cite{Wahl75}, the proof of Theorem C illustrates why the
lifting property is guaranteed to hold only in characteristic \( 0 \),
and also provides sufficient conditions for its truth
in characteristic \( p > 0 \).

However, in our constructions, we have only
two-dimensional toric singularities by contracting linear chains of
smooth rational curves.
In Section~\ref{sect:dualgraph} below, we prove that
the singularity obtained by
contracting a linear chain of smooth rational curves is
a two-dimensional toric singularity and it is taut
(cf.\ Theorem~\ref{thm:graph2toric} below).
Moreover, it turns out that the lifting property of derivations mentioned above
is not so important for proving the vanishing of the obstruction space
(cf.\ Corollary~\ref{cor:compareTangent} below).

In Section~\ref{sect:DefCT}, we introduce the notion of toric surface
singularity of class T (cf.\ Definition~\ref{dfn:classT}),
explain some properties with Tables~\ref{table:M} and
\ref{table:M2} of related numbers for some special cases,
and construct a so-called \( \BQQ \)-Gorenstein smoothing explicitly
by using toric description in Theorem~\ref{thm:localsmoothing}.

We recall several basics on the deformation theory of schemes
in Section~\ref{sect:Deformation} including Schlessinger's theory of
functors on Artinian rings.
As an exercise, using cotangent complexes and obstruction theory,
we shall prove in Theorem~\ref{thm:H2T=0}
that \emph{if \( \OH^{2}(X, \Theta_{X/\Bbbk}) = 0 \)
for an algebraic \( \Bbbk \)-variety \( X \) with only isolated singularities,
then the morphism \( \Def_{X} \to \Def_{X}^{(\loc)} \)
between the global and local the deformation functors of \( X \)
is smooth in the sense of Schlessinger}
(cf.\ \cite[Proposition~6.4]{WahlSmoothings}).
An algebraization result (Theorem~\ref{thm:algn})
is added, which plays an important role
when we construct an algebraic deformation.
The proof uses Artin's theory of algebraization (cf.\ \cite{ArtinFormal}).

In Section~\ref{sect:DeformationClassT},
we shall construct a
deformation of a normal projective surface \( X \) with toric singularities
of class T assuming some extra conditions.
As a consequence, we have a so-called \( \BQQ \)-Gorenstein smoothing
not only over the base field \( \Bbbk \) but
also over a complete discrete valuation ring
with the residue field \( \Bbbk \)
(cf.\ Theorems~\ref{thm:globalsmoothing} and \ref{thm:DVRsmoothing}).
As a byproduct, in Corollary~\ref{cor:log del Pezzo index two},
we can give a correct version of
the existence of \( \BQQ \)-Gorenstein smoothing of
any log del Pezzo surfaces of index two in positive characteristic
(cf.\ \cite[Theorem~5.16]{logdelPezzo}).
By the smoothing over the discrete valuation ring and the Grothendieck
specialization theorem
(cf.\ \cite[Exp.~X, Corollaire~2.4, Th\'eor\`eme~3.8]{SGA1}),
the algebraic simply connectedness of the smooth fiber is
reduced to that of a smooth fiber of a \( \BQQ \)-Gorenstein smoothing
of a reduction of \( X \) of our construction
to the complex number field \( \BCC \).
Note that the simply connectedness in case \( \Bbbk = \BCC \)
for the known Lee--Park constructions has been proved in
\cite{LeePark}, \cite{PPS1}, and \cite{PPS2} by using
Milnor fiber (or rational blow-down) and by applying
van-Kampen's theorem on the minimal resolution of \( X\).

Our plan of the proof of the Main Theorem is as follows:
We first construct a normal rational surface \( X \)
with only toric singularities
of class T satisfying extra conditions by Lee--Park's method.
This \( X \) is constructed from a suitable cubic pencil \( \Phi \)
on \( \BPP^{2} \),
a special construction of a birational morphism \( M \to \BPP^{2} \),
and a blowdown \( M \to X \) of a union of linear chains of rational curves
defining toric singularities of class T.
Second, we apply the results in Section~\ref{sect:DeformationClassT}
to \( X \) and obtain a so-called \( \BQQ \)-Gorenstein smoothing
of \( X \), in which a general smooth fiber \( X_{t} \) is
an expected surface satisfying
the required conditions in Main Theorem.

In Section~\ref{sect:Global}, we shall explain the outline of our proof
giving sufficient conditions for the data needed in constructing \( X \).
During the discussion, from the conditions, we shall prove some key results
such as the vanishing \( \OH^{2}(X, \Theta_{X/\Bbbk}) = 0\), and
the nef and bigness of \( K_{X} \).

In Section~\ref{sect:proof}, we shall give eight examples
working in our proof. Most examples are taken from
\cite{LeePark}, \cite{PPS1}, and \cite{PPS2}, but
Examples~\ref{exam:char2,K2=1}--\ref{exam:char2,K2=2} are new
which are presented by Heesang Park.
These new examples are needed because of some problems
in small characteristics.
For example,
rational surfaces admitting a minimal elliptic fibration with
a configuration type
\( (\text{I}_8, \text{I}_1, \text{I}_1, \text{I}_2)\) of singular fibers
are used in \cite{PPS1} and \cite{PPS2}, but in characteristic \( 2 \),
such a configuration does not exist \cite{Lang}.
Moreover, the vanishing \( \OH^{2}(X, \Theta_{X/\Bbbk}) = 0\)
does not hold in small characteristic in some cases.
By checking the conditions of Section~\ref{sect:Global}
for all the examples, we finally complete the proof of the Main Theorem.

\subsection*{Acknowledgements}
The first named author would like to thank Research Institute for
Mathematical Sciences (RIMS) in Kyoto University and Korea Institute for
Advanced Study for their hospitality during his visit. Main part of
the paper was worked out during his stay at RIMS in 2010.
The authors express their gratitude to
Professors Jonathan Wahl and Christian Liedtke
for the invaluable comments.
The authors are grateful to Dr.~Takuzo Okada and Dr.~Hisanori Ohashi
for their careful reading of the draft versions,
to Dr.~Heesang Park for her providing
Examples~\ref{exam:char2,K2=1}--\ref{exam:char2,K2=2}, and to
Prof.~Dongsoo Shin for his drawing figures. The authors thank to the
referee for suggesting many improvements.


\subsection*{Notation and conventions}
In this article, we fix an algebraically closed field \( \Bbbk \)
of characteristic \( p \geq 0 \).

\( \bullet \) An \emph{algebraic scheme} over a field \( \BKK \) means
a \( \BKK \)-scheme of finite type.
This is called also an \emph{algebraic \( \BKK \)-scheme}.
An \emph{algebraic variety} over a field \( \BKK \)
(or an \emph{algebraic \( \BKK \)-variety})
is an integral separated algebraic scheme over \( \BKK \).

\( \bullet \) For an algebraic \(\Bbbk\)-scheme \( X \),
\(\Omega^{1}_{X/\Bbbk}\) denotes the sheaf of
one-forms and \( \Theta_{X/\Bbbk}\) denotes the
tangent sheaf, i.e., \( \Theta_{X/\Bbbk} =
\SHom_{\SO_{X}}(\Omega^{1}_{X/\Bbbk}, \SO_{X}) \).
For a non-singular algebraic \( \Bbbk \)-variety \( X \) and
a normal crossing divisor \( B \),
\( \Omega^{1}_{X/\Bbbk}(\log B) \) denotes the sheaf of logarithmic
one-forms with poles along only \( B \). Its dual
\( \SHom_{\SO_{X}}(\Omega^{1}_{X/\Bbbk}(\log B), \SO_{X}) \) is denoted by
\( \Theta_{X/\Bbbk}(-\log B) \), which is identified with the sheaf of
derivations \( \delta \in \Theta_{X/\Bbbk} \) such that
\( \delta(\SO_{X}(-B)) \subset \SO_{X}(-B) \).

\( \bullet \) For a normal integral separated scheme \( X \)
and for a (Weil) divisor \( D \), the
reflexive sheaf \( \SO_{X}(D) \) of rank one is by definition the
subsheaf of the sheaf of rational functions
determined by the following property: A non-zero rational function
\( \varphi \) is contained in \( \OH^{0}(\SU, \SO_{X}(D)) \) for a
non-empty open subset \( \SU \) if and only if \(
\Div(\varphi)|_{\SU} + D|_{\SU}\geq 0 \) for the principal divisor \(
\Div(\varphi)|_{\SU} \) on \( \SU \).

\( \bullet \) The maximal ideal of a local ring \(A\) is denoted by \( \GM_{A} \).

\( \bullet \) A \emph{geometric point} \( t \) of a scheme \( X \)
is by definition a morphism
\( t \colon \Spec \Bbbk(t) \to X\) for
an algebraically closed field \( \Bbbk(t) \).
For any scheme \( Z \) over \( X \), the geometric fiber \( Z_{t} \) over
the geometric point \( t \) is defined to be the fiber product
\( Z \times_{X, \, t} \Spec \Bbbk(t) \).

\( \bullet \) An \emph{\'etale neighborhood} of a pair \( (X, x) \) of a scheme \( X \) and
a point \( x \) is by definition a pair \( (X', x') \)
of a scheme \( X' \) \'etale over \( X \) and a point \( x' \) lying over \( x \)
such that the induced homomorphism between the residue fields at \( x \) and \( x' \)
is isomorphic (cf.\ \cite[Section~2]{ArtinApprox}).

\( \bullet \) In this article,
the (\'etale) fundamental group of a scheme \( X \) is called
the \emph{algebraic fundamental group} and
is denoted by \( \pi_{1}^{\alg}(X) \) to avoid confusion with
the topological fundamental group of a complex algebraic scheme
(cf.\ \cite[Exp.~XII, Corollaire~5.2]{SGA1}).


\section{Linear chains of rational curves as exceptional loci}
\label{sect:dualgraph}

The main purpose of this section is to prove
that any two-dimensional singularity having
a linear chain of smooth rational curves
as the exceptional locus of the minimal resolution is ``toric.''
In other words, we prove the tautness (cf.\ \cite{Laufer})
of such singularities.
In the case of characteristic zero, the singularity is just
the cyclic quotient singularity of type \( \frac{1}{n}(1, q) \)
for some integers \( n > q > 0 \) with \( \gcd(n, q) = 1 \), and
the tautness is known by \cite[Satz~2.10]{Brieskorn}.

In this section, let us fix a Noetherian ring \( \Lambda \) and
positive integers \( n \), \( q \) such that \( n > q \) and \( \gcd(n, q) = 1 \).
We define integers \( b_{1} \), \ldots, \( b_{l} \) greater than one
by the continued fraction:
\[ n/q = [b_{1}, \ldots, b_{l}]
:= b_{1} - \cfrac{1}{b_{2} - \cfrac{1}{{}^{\ddots}- \cfrac{1}{b_{l}}}}
\]
and introduce the following two conditions:

\begin{dfn}[Conditions $C(n, q)$ and $C(n, q)'$]\label{dfn:cond_b}
Let \( Y \) be an affine algebraic flat \( \Lambda \)-scheme
and let \( \Sigma \) be a closed subscheme such that
\begin{itemize}
\item  every geometric fiber of \( Y \to \Spec \Lambda \) is a normal surface,
\item \( \Sigma \) is a section of \( Y \to \Spec \Lambda \), and
\item  \( Y \setminus \Sigma \) is smooth over \( \Spec \Lambda \).
\end{itemize}
The pair \( (Y, \Sigma) \) is said to satisfy {\it condition} \( C(n, q) \)
with a proper surjective morphism \( \mu \colon M \to Y \)
if the following conditions are satisfied:
\begin{enumerate}
\item  \( M \to Y \to \Spec \Lambda \) is smooth and
\( \mu \) is an isomorphism over \( Y \setminus \Sigma \).

\item  \( \mu^{-1}(\Sigma) \) is a divisor \( \sum\nolimits_{i = 1}^{l} E_{i} \)
such that:
\begin{enumerate}
\item  Any \( E_{i} \) is a Cartier divisor with
\( E_{i} \isom \BPP^{1}_{\Lambda} \) and
\( \SO_{E_{i}}(-E_{i}) \isom \SO(b_{i}) \).

\item  \( E_{i} \cap E_{j} = \emptyset \) if \( |i - j| > 1 \).

\item  The scheme-theoretic intersection
\( \Sigma_{i} := E_{i-1} \cap E_{i} \) is a section over
\( \Spec \Lambda \) for all \( 2 \leq i \leq l \).
\end{enumerate}
\end{enumerate}
If, in addition, there exist divisors \( B_{1} \), \( B_{2} \) on \( Y \)
satisfying the conditions below, then \( (Y, B_{1}, B_{2}) \) is
said to satisfy {\it condition} \( C(n, q)' \):
\begin{enumerate}
    \addtocounter{enumi}{2}
\item  The set-theoretic intersection \( B_{1} \cap B_{2} \) is \( \Sigma \).

\item  There is a relative Cartier divisor \( E_{0} \) (resp.\ \( E_{l+1} \))
with respect to \( M \to \Spec \Lambda \) such that:
\begin{enumerate}
\item  \( E_{0} \cap E_{i} = \emptyset \) and \( E_{j} \cap E_{l+1} = \emptyset\)
for all \( i > 1 \) and \( j < l\).

\item  The scheme-theoretic intersections
\( \Sigma_{1} := E_{0} \cap E_{1} \) and \( \Sigma_{l+1} := E_{l} \cap E_{l+1} \)
are sections of \( M \to \Spec \Lambda \).

\item \( B_{1} \) (resp.\ \( B_{2} \)) is the image of \( E_{l+1} \) (resp.\
\( E_{0} \)) by \( \mu \colon M \to Y \).
\end{enumerate}
\end{enumerate}
\end{dfn}

\begin{remsub}
In the situation above,
the ``dual graph'' of \( \sum\nolimits_{i = 0}^{l+1}E_{i} \) is written as
\begin{equation}\label{eq:dualgraph4}
\mbox{
\begin{picture}(200, 30)(-5, 0)
    \put(0, 5){\circle*{10}}\put(-10, 20){$E_{0}$}
    \put(5, 5){\line(1, 0){25}}
    \put(35, 5){\circle{10}}\put(30, 20){$E_{1}$}
    \put(40, 5){\line(1, 0){25}}
    \put(70, 5){\circle{10}}\put(65, 20){$E_{2}$}
    \put(75, 5){\line(1, 0){10}}
    \put(90, 5){\line(1, 0){10}}
    \put(105, 5){\line(1, 0){10}}
    \put(120, 5){\line(1, 0){10}}
    \put(135, 5){\line(1, 0){10}}
    \put(150, 5){\circle{10}}\put(145, 20){$E_{l}$}
    \put(155, 5){\line(1, 0){25}}
    \put(185, 5){\circle*{10}}\put(175, 20){$E_{l+1}$}
\end{picture}
}
\end{equation}
where the ends are distinguished, since these are not \( \mu \)-exceptional.
\end{remsub}

\begin{remsub}
If \( \Lambda \) is an algebraically closed field, then the condition
\( C(n, q) \) means that \( Y \) is a normal affine surface with
a singular point \( \Sigma \)
in which the exceptional locus of the minimal resolution is
a linear chain of smooth rational curves.
Moreover, in this case the condition \( C(n, q)' \) means that
\( B_{1} \) and \( B_{2} \) are prime divisors on \( Y \)
with \( B_{1} \cap B_{2} = \Sigma \) in which the dual graph of
\( \mu^{-1}(B_{1} \cup B_{2}) \) is also a linear chain
and the end components correspond to the proper transforms
of \( B_{1} \) and \( B_{2} \).
\end{remsub}

In this section we will do the following:
\begin{itemize}
\item  Giving an explicit construction of a toric \( \Lambda \)-scheme \( V \)
with two boundary divisors \( D_{1} \), \( D_{2} \) such that
\( (V, D_{1}, D_{2}) \) satisfies \( C(n, q)' \).

\item  Showing that
any \( (Y, \Sigma) \) satisfying \( C(n, q) \)
is ``\'etale equivalent to'' \( V \)
along \( \Sigma \) when \( \Lambda \) is a field
(cf.\ Theorem~\ref{thm:graph2toric000}).

\item  Showing that any \( (Y, B_{1}, B_{2}) \) satisfying \( C(n, q)' \)
is ``\'etale over'' \( (V, D_{1}, D_{2}) \)
when \( \Lambda \) is a local ring
(cf.\ Theorem~\ref{thm:graph2toric}).

\item As applications of Theorems~\ref{thm:graph2toric000} and
\ref{thm:graph2toric},
giving some local properties of
the surface singularity which has a linear chain of smooth rational curves
as the exceptional locus of the minimal resolution.
\end{itemize}

\begin{remn}
When \( \Lambda \) is an algebraically closed field,
Theorem~\ref{thm:graph2toric000}
corresponds to the ``tautness'' (cf.\ \cite{Laufer}) for two-dimensional
normal singularities having
linear chains of smooth rational curves as the exceptional locus
of the minimal resolution.
In positive characteristic, the tautness for linear chains of
smooth rational curves seems to be well-known,
but the authors could not find any reference.
Recently, Hara has found another proof of
Theorem~\ref{thm:graph2toric000}
(cf.\ Case (1) in the proof of \cite[Theorem~2.1]{Hara})
using an argument in \cite{BPV}.
\end{remn}

We begin with constructing \( V \) by applying
the theory of toric varieties or of torus embeddings.
We refers the reader to
\cite{Demazure}, \cite{TE}, \cite{Danilov}, \cite{Oda}, etc.\
for more details of the theory.
Let \( \NN_{0} \) be a free abelian group of rank two with a basis
\( (e_{1}, e_{2}) \), i.e., \( \NN_{0} = \BZZ e_{1} + \BZZ e_{2} \).
For the fixed integers \( n \) and \( q \) above,
we set
\[ v := (1/n)(e_{1} + qe_{2}) \in \NN_{0, \BQQ} = \NN_{0} \otimes \BQQ
\quad \text{ and } \quad \NN := \NN_{0} + \BZZ v \subset \NN_{0, \BQQ}. \]
Now, we define \( V \) to be the affine toric \( \Lambda \)-scheme
\( \BTT_{\NN}(\Bsigma) \)
associated with \( \NN \) and with
the cone \( \Bsigma =
\BRR_{\geq 0}e_{1} + \BRR_{\geq 0}e_{2} \subset \NN_{\BRR} \).
More precisely, \( V := \Spec \Lambda[\Bsigma^{\vee} \cap \MM] \)
for the semi-group ring \( \Lambda[\Bsigma^{\vee} \cap \MM]  \), where
\( \MM = \Hom(\NN, \BZZ) \) and
\( \Bsigma^{\vee} = \{m \in \MM_{\BRR} \mid m \geq 0 \text{ on } \Bsigma\} \).
Note that \( \Lambda[\Bsigma^{\vee}\cap \MM_{0}] \)
is a polynomial algebra \( \Lambda[\xtt_{1}, \xtt_{2}] \) of two variables,
where \( \MM_{0} = \Hom(\NN_{0}, \BZZ) \) and that
the subalgebra \( \Lambda[\Bsigma^{\vee} \cap \MM]  \) is
generated by monomials \( \xtt_{1}^{m_{1}}\xtt_{2}^{m_{2}} \)
for integers \( m_{1} \), \( m_{2} \) such that
\( m_{1} \geq 0 \), \( m_{2} \geq 0 \), and \( m_{1} + qm_{2} \equiv 0 \mod n \).
For the group ring \( \Lambda[\MM] \),
the affine \( \Lambda \)-scheme \( \BTT_{\NN} = \Spec \Lambda[\MM] \) is a group scheme
isomorphic to \( \mathbb{G}^{2}_{\textrm{m}, \Lambda} \),
the algebraic torus of relative dimension two.
Since \( \Bsigma^{\vee} \cap \MM \subset \MM \),
\( V \) contains \( \BTT_{\NN}(\{0\}) = \BTT_{\NN} \) as an open subset.
Moreover, the standard action of \( \BTT_{\NN} \) on \( \BTT_{\NN} \) extends to
\( V \), compatibly with the open immersion \( \BTT_{\NN} \subset V \).

\begin{dfn}\label{dfn:toricnq}
The \( \Lambda \)-scheme \( V \) is said to be the {\it toric \( \Lambda \)-scheme of type}
\( (n, q) \).
\end{dfn}

For an element \( m \in \MM \), let \( \Be(m) \) be the element in
the group algebra \( \Lambda[\MM] \) corresponding to \( m \).
Then, \( \Be(m) \) can be regarded as a rational function on \( V \), and
it is regular when \( m \in \Bsigma^{\vee} \cap \MM \).
The cone \( \Bsigma \) has four faces: \( \Bsigma \) itself,
two rays \( \BRR_{\geq 0}e_{1} \),
\( \BRR_{\geq 0}e_{2} \), and the zero-dimensional cone \( \{0\} \).
For such a face \( \Btau \), let \( \Btau^{\perp} \) be the vector subspace
\( \{m \in \MM_{\BRR} \mid m(\Btau) = 0\} \) and consider the ring homomorphism
\( \Lambda[\Bsigma^{\vee} \cap \MM] \to \Lambda[\Btau^{\perp} \cap \MM]\)
defined by
\[ \Be(m) \mapsto \begin{cases}
\Be(m), & \text{ if } m \in \Btau^{\perp}; \\
0, & \text{ otherwise.}
\end{cases}\]
The kernel of the ring homomorphism defines a closed subscheme \( Z(\Btau) \), which is
the closure of an orbit of \( \BTT_{\NN} \) on \( V \).
We set \( D_{i} := Z(\BRR_{\geq 0}e_{i}) \) for \( i = 1 \), \( 2 \), and
set \( \BSigma := Z(\Bsigma) \).
Then, the following properties are easily shown:
\begin{itemize}
\item \( D_{1} \) and \( D_{2} \) are prime divisors on \( V \)
when \( \Lambda \) is integral.

\item \( \BSigma \) is a section of \( V \) over \( \Spec \Lambda \).

\item  \( \BSigma \) is the scheme-theoretic
intersection \( D_{1} \cap D_{2} \)
(cf.\ Remark~\ref{remsub:dfn:dfn:toricnq:D1D2} below).

\item \( V \setminus (D_{1} \cup D_{2}) \) is the ``open torus'' \( \BTT_{\NN}(\{0\}) \).

\item For \( i = 1 \), \( 2 \),
\( D_{i} \isom \BAA^{1}_{\Lambda}\) and
\( D_{i} \setminus \BSigma \isom \mathbb{G}_{\textrm{m}, \Lambda} \).

\item \( \BTT_{\NN}(\{0\}) \), \( D_{1} \setminus \BSigma \),
\( D_{2} \setminus \BSigma\), and \( \BSigma \) are
the orbits of \( \BTT_{\NN} \) in \( V \).
\end{itemize}

Note also that, for \( i = 1 \), \( 2 \), the order
of zeros (or the minus of the order of poles) of
the rational function \( \Be(m) \) for \( m \in \MM \) along
the prime divisor \( D_{i} \) equals \( m(e_{i}) \)
(cf.\ \cite[Chapter~I, Theorem~9]{TE},\cite[\S 5]{Danilov},
\cite[Section~2.1]{Oda}). Hence,
the principal divisor \( \Div(\Be(m)) \)
has the following expression:
\begin{equation}\label{eq:principaldiv000}
\Div(\Be(m)) = m(e_{1})D_{1} + m(e_{2})D_{2}.
\end{equation}

\begin{remsub}\label{remsub:dfn:dfn:toricnq:D1D2}
Using \eqref{eq:principaldiv000}, we can prove that
\( \BSigma\) is just the scheme-theoretic intersection
of \( D_{1} \) and \( D_{2} \) as follows.
For \( i = 1 \), \( 2 \),
the defining ideal \( I_{i} \) of \( D_{i} \) is generated by \( \Be(m) \)
for \( m \in \Bsigma^{\vee} \cap \MM \) with \( m(e_{i}) > 0 \)
by \eqref{eq:principaldiv000}.
On the other hand, the defining ideal \( I \) at \( \BSigma \) is generated by
\( \Be(m) \) for \( m \in \Bsigma^{\vee} \cap \MM \setminus \{0\}\).
Thus, \( I= I_{1} + I_{2} \), i.e., \( \BSigma \) is the
scheme-theoretic intersection \( D_{1} \cap D_{2} \).
\end{remsub}

\begin{remsub}
If \( \Lambda' \) is a \( \Lambda \)-algebra, then
\( V \times_{\Spec \Lambda} \Spec \Lambda' \)
is also the toric \( \Lambda' \)-scheme of type \( (n, q) \).
If \( \Lambda \) is an algebraically closed field,
then \( V \) is an affine normal surface
with unique singular point \( \BSigma \).
\end{remsub}

In order to show that \( (V, D_{1}, D_{2}) \) satisfies the condition \( C(n, q)' \),
we shall construct
the toric ``minimal resolution of singularities'' \( \nu \colon U \to V \)
which is essentially the same as the well-known Jung--Hirzebruch resolution
(cf.\ \cite{Jung}, \cite{Hirz}).
The following lemma follows from the property \( n/q = [b_{1}, \ldots, b_{l}] \):
the proof is left to the reader.

\begin{lem}\label{lem:v_i}
There exist vectors \( v_{0} \), \( v_{1} \), \ldots, \( v_{l+1} \in \NN \cap \Bsigma \)
satisfying the following conditions\emph{:}
\begin{enumerate}
\item  Let \( p_{i} \), \( q_{i} \) be integers determined by
\( v_{i} = (1/n)(p_{i}e_{1} + q_{i}e_{2}) \).  Then,
\[ p_{0} = 0 < p_{1} = 1 < \cdots < p_{l+1} = n \quad \text{ and }
\quad q_{0} = n > q_{1} = q> \cdots > q_{l} = 1 > q_{l+1} = 0. \]
\item  \( v_{i-1} + v_{i+1} = b_{i}v_{i} \) for \( 1 \leq i \leq l \).
\item \( \NN = \BZZ v_{i-1} + \BZZ v_{i}  \) for \( 1 \leq i \leq l \).
\item\label{lem:v_i:4}
Let \( (h_{1}, h_{2}) \) be the basis of \( \MM_{0} = \Hom(\NN_{0}, \BZZ) \) dual to
\( (e_{1}, e_{2}) \), i.e., \( h_{i}(e_{j}) = \delta_{i, j} \).
Then, for all \( 1 \leq i \leq l \),
\[ (-q_{i}h_{1} + p_{i}h_{2}, \, q_{i-1}h_{1} - p_{i-1}h_{2}) \]
is the basis of \( \MM \) dual to \( (v_{i-1}, v_{i}) \).
\item \label{lem:v_i:5}
\( p_{i}q_{j} \equiv p_{j}q_{i} \mod n\) for all \( 1 \leq i, j \leq l\).
\end{enumerate}
\end{lem}

Note that \( v_{0} = e_{2} \), \( v_{1} = v \), and \( v_{l+1} = e_{1} \).
Note also that the last assertion \eqref{lem:v_i:5} is derived from
\eqref{lem:v_i:4} since it is equivalent to:
\( (-q_{i}h_{1} + p_{i}h_{2})(v_{j}) \in \BZZ  \) for all \( 1 \leq i, j \leq l\).

The cones \( \Bsigma_{i} := \BRR_{\geq 0} v_{i-1} + \BRR_{\geq 0} v_{i} \) for
\( 1 \leq i \leq l \) and the rays \( \BRR_{\geq 0} v_{j} \)
for \( 0 \leq j \leq l +1 \) together with the zero cone \( \{0\} \)
form a non-singular \emph{fan} \( \triangle \) of \( \NN \)
such that the \emph{support} \( |\triangle| \)
coincides with \( \Bsigma \)
(cf.\ \cite[Section~4.2]{Demazure}, \cite[Chapter~I, \S2]{TE},
\cite[\S 5]{Danilov}, \cite[Chapter~1]{Oda}).
Let \( U \) be
the toric \( \Lambda \)-scheme \( \BTT_{\NN}(\triangle)  \)
associated with the fan \( \triangle \).
Then, we have a canonical proper surjective morphism \( \nu \colon U \to V \)
which is an isomorphism at least on the open torus \( \BTT_{\NN}(\{0\}) \).
Now, \( U \) is a union of the smooth affine toric \( \Lambda \)-schemes
\( U_{i} = \Spec \Lambda[\Bsigma_{i}^{\vee} \cap \MM] \isom \BAA^{2}_{\Lambda}\).
In fact, by Lemma~\ref{lem:v_i}, we see that
\( \Lambda[\Bsigma_{i}^{\vee} \cap \MM] \) is a polynomial \( \Lambda \)-algebra
of two variables generated by
\[ \xi_{i} := \Be(-q_{i}h_{1} + p_{i}h_{2}) = \xtt_{1}^{-q_{i}}\xtt_{2}^{p_{i}}
\quad \text{ and } \quad
\eta_{i} := \Be(q_{i-1}h_{1} - p_{i-1}h_{2}) = \xtt_{1}^{q_{i-1}}\xtt_{2}^{-p_{i-1}}. \]
Here, \( U_{i} \cap U_{j} \) equals the open torus
\( \BTT_{\NN}(\{0\}) \isom \BTT_{\NN}\) of \( U \) if \( |i - j| > 1 \), since
\( \Bsigma_{i} \cap \Bsigma_{j} = \{0\} \).
For \( 1 \leq i \leq l\), the intersection \( U_{i}^{\star} := U_{i} \cap U_{i+1} \)
is isomorphic to the toric \( \Lambda \)-scheme
\( \BTT_{\NN}(\BRR_{\geq 0}v_{i}) \isom
\BAA^{1}_{\Lambda} \times_{\Spec \Lambda} \mathbb{G}_{\mathrm{m}, \Lambda}\),
since \( \Bsigma_{i} \cap \Bsigma_{i+1} = \BRR_{\geq 0}v_{i} \).
For \( 0 \leq i \leq l+1 \), let \( G_{i} \) be the \( \BTT_{\NN} \)-invariant
prime divisor on \( U = \BTT_{\NN}(\triangle) \) corresponding to
the ray \( \BRR_{\geq 0} v_{i} \).
In particular, \( G_{0} \) (resp.\ \( G_{l+1} \)) is the proper transform of
\( D_{2} \) (resp.\ \( D_{1} \)).

\begin{lem}\label{lem:VDDC}
The triplet \( (V, D_{1}, D_{2}) \) satisfies the condition
\( C(n, q)' \) with the proper morphism \( \nu \colon U \to V \).
\end{lem}

\begin{proof}
For any \( m \in \MM \), we have
\begin{equation}\label{eq:principaldiv222}
\Div(\Be(m)) = \sum\nolimits_{i = 0}^{l+1} m(v_{i}) G_{i}
\end{equation}
as in \eqref{eq:principaldiv000},
where \( \Be(m) \) is regarded as a rational function on \( U \).
Now, \( G_{i}|_{U_{j}} = 0 \) for \( j < i \) and \( j > i+1 \),
since \( v_{i} \not\in \Bsigma_{j}  \).
Furthermore, we have:
\begin{equation}\label{eq:GG_i}
G_{i}|_{U_{i}} = \Div(\eta_{i})|_{U_{i}} \quad \text{ and }
\quad
G_{i}|_{U_{i+1}} = \Div(\xi_{i+1})|_{U_{i+1}},
\end{equation}
by the calculation
\begin{align*}
\Div(\xi_{i})|_{U_{i}} &=
\sum\nolimits_{j = i-1}^{i}(-q_{i}h_{1} + p_{i}h_{2})(v_{j})\, G_{j}|_{U_{i}}
= G_{i-1}|_{U_{i}},  \\
\Div(\eta_{i})|_{U_{i}} &=
\sum\nolimits_{j = i-1}^{i}(q_{i-1}h_{1} - p_{i-1}h_{2})(v_{j})\, G_{j}|_{U_{i}}
= G_{i}|_{U_{i}},
\end{align*}
using \eqref{eq:principaldiv222} and Lemma~\ref{lem:v_i}.
As a consequence,
\[ U_{i}^{\star} = U_{i} \setminus G_{i-1} = U_{i} \cap \{ \xi_{i} \ne 0 \}
\isom \BAA^{1}_{\Lambda} \times \mathbb{G}_{\textrm{m}, \Lambda}. \]
The toric \( \Lambda \)-scheme \( U \) is obtained by
gluing the affine \( \Lambda \)-planes
\( U_{i} \isom \BAA^{2}_{\Lambda}\) for \( 1 \leq i \leq l+1 \) with
the coordinate systems \( (\xi_{i}, \eta_{i}). \)
The transition relation is
\begin{equation}\label{eq:transition222}
\xi_{i+1}|_{U_{i}^{\star}} = \xi_{i}^{b_{i}}\eta_{i} |_{U_{i}^{\star}},
\quad \text{ and } \quad
\eta_{i+1}|_{U_{i}^{\star}} = \xi_{i}^{-1}|_{U_{i}^{\star}},
\end{equation}
on \( U_{i}^{\star} = U_{i} \cap U_{i+1} \).
This is possible, since \( (-q_{i+1}, p_{i+1}) - (q_{i-1}, -p_{i-1}) = b_{i}(-q_{i}, p_{i})\)
(cf.\ Lemma~\ref{lem:v_i}).
Furthermore, \( G_{i} \isom \BPP^{1}_{\Lambda} \) with
\( \SO_{G_{i}}(G_{i}) \isom \SO(-b_{i}) \) for all \( 1 \leq i \leq l \)
by \eqref{eq:GG_i} and \eqref{eq:transition222}.
On the other hand, \( G_{0} \) and \( G_{l+1} \)
are isomorphic to \( \BAA^{1}_{\Lambda} \).
The scheme-theoretic intersection \( \BSigma_{i} = G_{i-1} \cap G_{i} \)
for \( 1 \leq i \leq l + 1 \) is just the orbit
corresponding to the two-dimensional cone \( \Bsigma_{i} \).
Therefore, \(  \sum\nolimits_{i = 0}^{l+1} G_{i}  \) is
a (relative) simple normal crossing divisor (over \( \Spec \Lambda \))
with the dual graph similar to
\eqref{eq:dualgraph4}, and hence, \( (V, D_{1}, D_{2}) \) satisfies
the condition \( C(n, q)' \).
\end{proof}

\begin{remsub}\label{remsub:lem:VDDC}
Let \( p_{i} \) and \( q_{i} \) for \( 0 \leq i \leq l+1 \)
be the integers defined by \( (n, q) \) in Lemma~\ref{lem:v_i}.
If \( a_{1} \) and \( a_{2} \) are integers with
\( a_{1} + qa_{2} \equiv 0 \bmod n \),
then, \( p_{i}a_{1} + q_{i}a_{2} \equiv 0 \bmod n\) for all
\( 0 \leq i \leq l+1 \), and
\begin{equation}\label{eq:muaD000}
\mu^{*}(a_{1}D_{1} + a_{2}D_{2}) =
\sum\nolimits_{i = 0}^{l+1} \frac{p_{i}a_{1} + q_{i}a_{2}}{n} G_{i}.
\end{equation}
In fact, this is derived from
\eqref{eq:principaldiv000} and \eqref{eq:principaldiv222} applied to
\(m = a_1h_1 + a_2h_2\).
As special cases, we have
\begin{equation}\label{eq:muDD}
\mu^{*}(nD_{1}) =
\sum\nolimits_{i = 1}^{l+1} p_{i} G_{i} \quad \text{ and } \quad
\mu^{*}(nD_{2}) =
\sum\nolimits_{i = 0}^{l} q_{i} G_{i}.
\end{equation}
We have another proof of \eqref{eq:muaD000} in
Corollary~\ref{corsub:lem:Lipman0002} below.
\end{remsub}

Next, we consider an arbitrary affine \( \Lambda \)-scheme \( Y \) and
a closed subscheme \( \Sigma \) satisfying \( C(n, q) \).
First we shall show:

\begin{lem}\label{lem:relnormal}
Let \( Y \) be an affine \( \Lambda \)-scheme and
\(\Sigma \) a closed subscheme such that \( (Y, \Sigma) \) satisfies
\( C(n, q) \) for a proper morphism \( \mu \colon M \to Y \). Then,
\( \SO_{Y} \isom \mu_{*}\SO_{M} \isom j_{*}\SO_{Y \setminus
\Sigma},\) where \( j \colon Y \setminus \Sigma
\injmap Y \) denotes the open immersion.
\end{lem}

\begin{proof}
We take an arbitrary point \( y \in \Sigma \) and a point \( x \in M \)
lying over \( y \). Let \( t \in \Spec \Lambda\) be the image of \( y \),
and let \( Y_{t} \) and \( M_{t} \) be the fiber over \( t \) of
\( Y \to \Spec \Lambda \) and \( M \to \Spec \Lambda \), respectively.
Then,
\begin{align*}
\depth \SO_{Y, y} &= \depth \SO_{Y_{t}, y} + \depth \SO_{\Spec \Lambda, t}
\geq \depth \SO_{Y_{t}, y} = 2, \\
\depth \SO_{M, x} &= \depth \SO_{M_{t}, x} + \depth \SO_{\Spec \Lambda, t}
\geq \depth \SO_{M_{t}, x} \geq 1,
\end{align*}
since \( Y \) and \( M \) are flat over \( \Lambda \), and since
all the fibers of \( Y \to \Spec \Lambda \) and \( M \to \Spec \Lambda \)
are normal surfaces.
Hence, \( \SO_{Y} \to j_{*}\SO_{Y \setminus \Sigma} \) is an isomorphism and
\( \SO_{M} \to j'_{*}\SO_{M \setminus \mu^{-1}(\Sigma)} \) is injective
for the open immersion \( j' \colon M \setminus \mu^{-1}(\Sigma) \injmap M \)
(cf.\ \cite[IV, Th\'eor\`eme~(5.10.5), Proposition~(5.10.2)]{EGA} or
\cite[Exp.~III, Proposition~3.3, Corollaire~3.5]{SGA2}).
Since \( M \setminus \mu^{-1}(\Sigma) \isom Y \setminus \Sigma \),
we have isomorphisms
\( \SO_{Y} \isom \mu_{*}\SO_{M} \isom j_{*}\SO_{Y \setminus \Sigma}\).
\end{proof}

The condition \( C(n, q) \) characterizes
the toric \( \Lambda \)-scheme of type \( (n, q) \) up to \'etale morphism when
\( \Lambda \) is a field. Namely, we have:

\begin{thm}[tautness]\label{thm:graph2toric000}
Assume that \( \Lambda \) is a field.
Let \( (Y, \Sigma) \) be a pair
satisfying the condition \( C(n, q) \)
with a proper morphism \( \mu \colon M \to Y\).
Let \( V \) be the toric \( \Lambda \)-scheme \( V \) of type \( (n, q) \)
and \( \nu \colon U \to V \)
the minimal resolution constructed as above. Then, there exists
an \'etale neighborhood \( Y^{\circ} \)
of \( \Sigma \) in \( Y \) with
\'etale morphisms \( \tau \colon Y^{\circ} \to V \) and
\( \Phi \colon M^{\circ} := M \times_{Y} Y^{\circ} \to U\)
which form a commutative Cartesian diagram\emph{:}
\[ \begin{CD}
M^{\circ} @>{\mu \times_{Y} Y^{\circ}}>> Y^{\circ} \\
@V{\Phi}VV @V{\tau}VV \\
U @>{\nu}>> \phantom{.}V.
\end{CD}\]
\end{thm}

\begin{remsub}
An \'etale neighborhood \( Y^{\circ} \) of \( \Sigma \) in \( Y \) means
an \'etale morphism \( Y^{\circ} \to Y \) such that the image contains \( \Sigma \)
and that \( \Sigma \times_{Y} Y^{\circ} \to \Sigma \) is an isomorphism
(cf.\ Notation and conventions).
\end{remsub}

Similarly, the condition \( C(n, q)' \) characterizes
the toric \( \Lambda \)-scheme of type \( (n, q) \) up to \'etale morphism:

\begin{thm}\label{thm:graph2toric}
Assume that \( \Lambda \) is a Noetherian local ring.
Let \( (Y, B_{1}, B_{2})\)  be a collection
satisfying the condition \( C(n, q)' \)
with a proper morphism \( \mu \colon M \to Y\).
Let \( V \) be the toric \( \Lambda \)-scheme \( V \) of type \( (n, q) \)
and \( \nu \colon U \to V \)
the minimal resolution constructed as above. Then, there exists
an open neighborhood \( Y^{\circ} \)
of \( \Sigma \) in \( Y \) with
\'etale morphisms \( \tau \colon Y^{\circ} \to V \) and
\( \Phi \colon M^{\circ} := M \times_{Y} Y^{\circ} \to U\)
which form the same commutative Cartesian diagram as in
Theorem~\emph{\ref{thm:graph2toric000}}.
\end{thm}

Before proving Theorems~\ref{thm:graph2toric000} and \ref{thm:graph2toric},
we need some results on
the ``singularity'' of \( Y \) along \( \Sigma \).

\begin{lem}\label{lem:Lipman000}
In the situation of Theorem~\emph{\ref{thm:graph2toric000}} or
\emph{\ref{thm:graph2toric}},
if an invertible sheaf \( \SL \) on \( M \) is \( \mu \)-nef, i.e.,
\( \deg (\SL|_{C}) \geq 0\) for any fiber \( C \) of
\( E_{i} \isom \BPP^{1}_{\Lambda} \to \Sigma \isom \Spec \Lambda\) for all
\( 1 \leq i \leq l \),
then \( \OH^{j}(M, \SL) = 0 \) for all \( j > 0 \) and
\( \SL \) is generated by global sections. If \( \SL \) is
\( \mu \)-numerically trivial,
i.e., \( \deg (\SL|_{C}) = 0\) for any \( C \) above, then
\( \mu_{*}\SL \) is an invertible sheaf on \( Y \)
and \( \mu^{*}(\mu_{*}\SL) \isom \SL \).
\end{lem}

\begin{proof}
For the vanishing \( \OH^{j}(M, \SL) = 0  \), it is enough to check
only when \( j = 1 \),
since the dimension of the fibers of \( \mu \) is at most one.
Let \( Z \) be the divisor \( \sum\nolimits_{i = 1}^{l} E_{i} \)
which is a relative Cartier divisor over \( \Spec \Lambda \).
Then, \( \SO_{M}(-Z) \) is \( \mu \)-nef by \( b_{i} \geq 2 \).
It is enough to prove the following two assertions:
\begin{enumerate}
\item  \label{lem:Lipman000:cond1} \( \OH^{1}(Z, \SL|_{Z}) = 0 \).

\item  \label{lem:Lipman000:cond2} \( \SL|_{Z} \) is generated by global sections.
\end{enumerate}
In fact, \( \OH^{1}(Z, \SL|_{Z} \otimes \SO_{Z}(-mZ)) = 0 \)
for any \( m \geq 0 \)
by \eqref{lem:Lipman000:cond1},
and it implies the vanishing
\( \OH^{1}(M, \SO_{mZ} \otimes \SL) = 0 \) for all \( m \geq 0 \)
by induction using the exact sequence
\[ 0 \to \SO_{Z}(-(m-1)Z) \otimes_{\SO_{M}} \SL \to
\SO_{mZ} \otimes_{\SO_{M}} \SL
\to \SO_{(m-1)Z} \otimes_{\SO_{M}} \SL \to 0. \]
Hence, \( \OH^{1}(M, \SL) = 0  \), since the module \( \OH^{1}(M, \SL) \)
is supported on \( \Sigma \) and
the formal completion \( \OH^{1}(M, \SL)^{\wedge} \)
along \( \Sigma \) is isomorphic to the projective limit
\( \varprojlim_{m} \OH^{1}(\SO_{mZ} \otimes \SL) \)
(cf.\ \cite[III, Th\'eor\`eme~(4.1.5)]{EGA}).
We also have \( \OH^{1}(M, \SO_{M}(-Z) \otimes \SL) = 0 \),
since \( \SO_{M}(-Z) \) is nef, and as a consequence,
the restriction map
\( \OH^{0}(M, \SL) \to \OH^{0}(Z, \SL|_{Z}) \) is surjective.
Thus, \( \SL \) is generated by global sections by \eqref{lem:Lipman000:cond2}.

For integers \( 1 \leq a \leq b \leq l \), we set
\( Z_{a, b} := \sum\nolimits_{i = a}^{b} E_{i} \).
We shall show \eqref{lem:Lipman000:cond1} and \eqref{lem:Lipman000:cond2}
by proving the following two assertions for any such pair \( (a, b) \) of integers:
\begin{enumerate}
    \addtocounter{enumi}{2}
\item \label{lem:Lipman000:cond3}
\( \OH^{1}(Z_{a, b}, \SL|_{Z_{a, b}}) = 0 \).

\item \label{lem:Lipman000:cond4}
\( \SL|_{Z_{a, b}} \) is generated by global sections.
\end{enumerate}
We shall prove them by induction on \( b - a \).
These are true when \( a = b \). In fact, \( Z_{a, a} = E_{a} \),
\( E_{a} \isom \BPP^{1}_{\Lambda} \), and
\( \SL|_{E_{a}} \isom \SO_{\BPP^{1}_{\Lambda}}(d_{a})\) for some \( d_{a} \geq 0 \).
Assume that \( b > a \). Then, we have two exact sequences
\begin{align*}
&0 \to \SO_{E_{a}}(-Z_{a+1, b}) \isom \SO_{\BPP^{1}_{\Lambda}}(-1)
\to \SO_{Z_{a, b}} \to \SO_{Z_{a+1, b}} \to 0 \quad \text{and}\\
&0 \to \SO_{E_{b}}(-Z_{a, b-1}) \isom \SO_{\BPP^{1}_{\Lambda}}(-1)
\to \SO_{Z_{a, b}} \to \SO_{Z_{a, b-1}} \to 0.
\end{align*}
Since \( \SL|_{E_{a}} \isom \SO_{\BPP^{1}_{\Lambda}}(d_{a}) \) and
\( \SL|_{E_{b}} \isom \SO_{\BPP^{1}_{\Lambda}}(d_{b}) \) with \( d_{a} \), \( d_{b} \geq 0 \), we have
\[ \OH^{1}(E_{a}, \SL|_{E_{a}} \otimes \SO_{E_{a}}(-Z_{a+1, b})) =
\OH^{1}(E_{b}, \SL|_{E_{b}} \otimes \SO_{E_{b}}(-Z_{a, b-1})) = 0.\]
Therefore, we have surjections
\( \OH^{0}(\SL|_{Z_{a, b}}) \to \OH^{0}(\SL|_{Z_{a+1, b}}) \)
and \( \OH^{0}(\SL|_{Z_{a, b}}) \to \OH^{0}(\SL|_{Z_{a, b-1}})\),
and isomorphisms
\( \OH^{1}(\SL|_{Z_{a, b}}) \isom \OH^{1}(\SL|_{Z_{a+1, b}})
\isom \OH^{1}(\SL|_{Z_{a, b-1}})\).
Thus, \eqref{lem:Lipman000:cond3} and \eqref{lem:Lipman000:cond4}
for \( (a, b) \) follow from those for
\( (a+1, b) \) and for \( (a, b-1) \).
By induction on \( b - a \), \eqref{lem:Lipman000:cond3} and \eqref{lem:Lipman000:cond4} hold
for any \( (a, b) \), and hence, \eqref{lem:Lipman000:cond1} and \eqref{lem:Lipman000:cond2}
hold. Thus, we are done.
\end{proof}

\begin{remn}
If one knows the vanishing \( \OH^{1}(M, \SO_{M}) = 0 \), then
Lemma~\ref{lem:Lipman000} is a special case of \cite[Theorem~(12.1)]{Lipman},
which generalizes \cite[Lemma~5]{ArtinIsol} stated over an algebraically closed field.
\end{remn}

\begin{corsub}\label{corsub:lem:Lipman000}
Let \( Y \) be an affine normal algebraic surface defined over
an algebraically closed field \( \Bbbk \) with a unique singular point \( P \).
If the exceptional locus for the minimal resolution of  a singularity
is a linear chain of smooth rational curves,
then \( (Y, P) \) is a rational singularity.
\end{corsub}

\begin{remsub}
In Corollary~\ref{corsub:lem:Lipman000},
the divisor \( Z \) in the proof of Lemma~\ref{lem:Lipman000}
is nothing but the fundamental cycle.
Thus, the rationality of \( (Y, P) \) is also derived from \( p_{a}(Z) = 0 \) by
\cite[Theorem~3]{ArtinIsol}.
\end{remsub}

\begin{corsub}\label{corsub:lem:Lipman0002}
Let \( \mu \colon M \to Y \), \( B_{1} \), \( B_{2} \), and
\( \{E_{i}\}_{i = 0}^{l+1} \) be as in
the situation of Theorem~\emph{\ref{thm:graph2toric}}.
Let \( a_{1} \), \( a_{2} \) be integers such that
\( a_{1} + qa_{2} \equiv 0 \bmod n \).
Then, \( p_{i}a_{1} + q_{i}a_{2} \equiv 0 \bmod n\) for all
\( 0 \leq i \leq l+1 \) for integers \( p_{i} \) and \( q_{i} \)
defined in Lemma~\emph{\ref{lem:v_i}},
\( a_{1}B_{1} + a_{2}B_{2} \) is Cartier, and
the equality
\begin{equation}\label{eq:muaD}
\mu^{*}(a_{1}B_{1} + a_{2}B_{2}) =
\sum\nolimits_{i = 0}^{l+1} \frac{p_{i}a_{1} + q_{i}a_{2}}{n} E_{i}
\end{equation}
of Cartier divisors hold on \( M\).
\end{corsub}

\begin{proof}
By Lemma~\ref{lem:v_i},
we have \( p_{i}a_{1} + q_{i}a_{2} \equiv 0 \bmod n \), and
we see that the right hand side of \eqref{eq:muaD} is
\( \mu \)-numerically trivial.
Hence by Lemma~\ref{lem:Lipman000},
the associated invertible sheaf to the divisor of the right hand side
is just the pullback of an invertible sheaf on \( Y \).
Thus, \eqref{eq:muaD} holds, since \( a_{1}B_{1} + a_{2}B_{2} \)
is the push-forward of the right hand side.
\end{proof}

\begin{proof}[Proof of Theorems~\emph{\ref{thm:graph2toric000} and \ref{thm:graph2toric}}]
First, we shall prove Theorem~\ref{thm:graph2toric000} assuming
Theorem~\ref{thm:graph2toric} to be true.
In Theorem~\ref{thm:graph2toric000}, \( \Lambda \) is a field, say \( \BKK \).
Thus, \( \Sigma \) is a \( \BKK \)-rational point.
Let \( R\) be the henselization of the local ring \( \SO_{Y, \Sigma} \) and
let \( \mu\sptilde \colon M\sptilde \to \Spec R \) be
the base change of \( \mu \colon M \to Y \) by \( \Spec R \to Y \).
Then,
\[ \Pic(M\sptilde) \to \Pic\left(\bigcup\nolimits_{i = 1}^{l} E_{i}\right) \]
is surjective by \cite[Lemma~(14.3)]{Lipman} or
\cite[IV, Corollaire~(21.9.12)]{EGA}.
Hence, by replacing \( Y \) with an \'etale neighborhood \( Y^{\circ} \to Y \)
of \( \Sigma \) in \( Y \), we may assume that
\( \Pic(M) \to \Pic(\bigcup\nolimits_{i = 1}^{l} E_{i}) \)
is surjective.
Then, we can find invertible sheaves \( \SL \), \( \SL' \) on \( M \)
such that
\[ \deg \SL|_{E_{1}} = \deg \SL'|_{E_{l}} = 1, \quad \text{ and } \quad
\deg \SL|_{E_{i}} = \deg \SL'|_{E_{j}} = 0\]
for \( i > 1 \) and \( j < l \).
Applying Lemma~\ref{lem:Lipman000} to \( \SL \) and \( \SL' \), we have
two affine prime divisors \( E_{0} \), \( E_{l+1} \) on \( M \) such that
\begin{itemize}
\item  \( \SL \isom \SO_{M}(E_{0}) \) and \( \SL' \isom \SO_{M}(E_{l+1}) \),

\item  \( E_{0} \cap E_{1} \) and \( E_{l} \cap E_{l+1} \)
are sections of \( M \to \Spec \Lambda \), and

\item  \( E_{0} \cap \bigcup\nolimits_{i = 2}^{l+1} E_{i} =
E_{l+1} \cap \bigcup\nolimits_{i = 0}^{l-1} E_{i} = \emptyset\).
\end{itemize}
%
Hence, the dual graph of \( \sum\nolimits_{i = 0}^{l+1}E_{i} \) is
the same as \eqref{eq:dualgraph4}.
We set \( B_{1} := \mu_{*}(E_{l+1}) \) and \( B_{2} := \mu_{*}(E_{0}) \).
Then, \( B_{1} \) and \( B_{2} \) are prime divisors on \( Y \)
with \( B_{1} \cap B_{2} = \Sigma\), set-theoretically.
Furthermore, \( E_{l+1} \) and \( E_{0} \) are the proper transforms of
\( B_{1} \) and \( B_{2} \), respectively.
Thus, \( (Y, B_{1}, B_{2}) \) satisfies the condition \( C(n, q)' \).
Hence, Theorem~\ref{thm:graph2toric000} is derived from
Theorem~\ref{thm:graph2toric}.

In the rest of the proof, we shall prove
Theorem~\ref{thm:graph2toric}.
If \( a_{1} \) and \( a_{2} \) are integers with
\( a_{1} + qa_{2} \equiv 0 \bmod n \),
then \( a_{1}B_{1} + a_{2}B_{2} \) is Cartier by
Corollary~\ref{corsub:lem:Lipman0002}.
We may assume that these Cartier
divisors \( a_{1}B_{1} + a_{2}B_{2} \) are all linearly
equivalent to zero by replacing \( Y \)
with an open neighborhood of \( \Sigma \), since \( \Lambda \) is local.

For \( 0 \leq i \leq l+1 \),
let \( \epsilon_{i} \) be a global section of \( \SO_{M}(E_{i}) \)
such that the divisor \( (\epsilon_{i})_{0} \) of zeros equals \( E_{i} \);
in other words, \( \epsilon_{i} \colon \SO_{M} \to \SO_{M}(E_{i}) \) is dual to
the natural injection \( \SO_{M}(-E_{i}) \subset \SO_{M} \).
Let \( a_{1} \), \( a_{2} \) be integers with
\( a_{1} + qa_{2} \equiv 0 \bmod n \).
Then, by \eqref{eq:muaD}, there is
a rational function \( \phi_{a_{1}, a_{2}} \) on \( Y \)
such that
\[ \mu^{*}(\phi_{a_{1}, a_{2}}) =
\prod\nolimits_{i = 0}^{l+1} \epsilon_{i}^{(p_{i}a_{1} + q_{i}a_{2})/n}. \]
For \(  1 \leq i \leq l + 1 \),
let \( M_{i} \) be the complement of \( \bigcup\nolimits_{k \ne i - 1,\, i} E_{k} \)
in \( M \). Then, \( M_{i} \) is a neighborhood of the intersection
\( E_{i-1} \cap E_{i}\).
We define
\[ s_{i} := \mu^{*}(\phi_{-q_{i}, p_{i}})|_{M_{i}} \quad \text{ and }
\quad t_{i} := \mu^{*}(\phi_{q_{i-1}, -p_{i-1}})|_{M_{i}} \]
for \( 1 \leq i \leq l+1 \).
By \eqref{eq:muaD}, we see that
\( s_{i} \) and \( t_{i} \) are regular on \( M_{i} \) satisfying
\[ \Div(s_{i})|_{M_{i}} = E_{i-1}|_{M_{i}} \quad \text{ and } \quad
\Div(t_{i})|_{M_{i}} = E_{i}|_{M_{i}}.\]
In particular, \( (s_{i}, t_{i}) \)
is a local coordinate system of \( M_{i} \)
along the intersection \( E_{i-1} \cap E_{i} \).
For \( 1 \leq i \leq l \), let \( M_{i}^{\star} \) be the intersection
\( M_{i} \cap M_{i+1} = M \setminus \bigcup\nolimits_{k \ne i} E_{k}\).
Then, we have
\begin{equation}\label{eq:transition2}
s_{i+1}|_{M_{i}^{\star}} = s_{i}^{b_{i}}t_{i}|_{M_{i}^{\star}}, \quad
t_{i+1}|_{M_{i}^{\star}} = s_{i}^{-1}|_{M_{i}^{\star}},
\end{equation}
similar to \eqref{eq:transition222}.
For \( 1 \leq i \leq l+1 \), let
\(\Phi_{i} \colon
M_{i} \to U_{i} = \Spec \Bbbk[\Bsigma_{i}^{\vee} \cap \MM]
\isom \BAA^{2}_{\Bbbk} \)
be the morphism defined by
\[ \Phi_{i}^{*}(\xi_{i}) = s_{i} \quad \text{ and } \quad
\Phi_{i}^{*}(\eta_{i}) = t_{i}, \]
which is \'etale along \( (E_{i-1} \cup E_{i}) \cap M_{i} \).
By \eqref{eq:transition222} and \eqref{eq:transition2}, the morphisms
\( \Phi_{i}\) for \( 1 \leq i \leq l+1 \) are glued to a morphism
\( \Phi \colon M = \bigcup M_{i} \to U = \bigcup U_{i} \), which is
\'etale along \( \bigcup\nolimits_{i = 0}^{l+1} E_{i} \), and which induces
\( \Phi^{*}(G_{i}) = E_{i} \) for all \( 0 \leq i \leq l+1 \).
Since \( \mu \circ \Phi \colon M \to U \to V \) contracts
the divisor \( E = \sum\nolimits_{i = 1}^{l}E_{i} \) to the section \( \BSigma \),
we have a morphism \( \tau \colon Y \to V \)
such that \( \nu \circ \Phi = \tau \circ \mu\), i.e.,
the diagram
\[ \begin{CD}
M @>{\mu}>> Y \\
@V{\Phi}VV @V{\tau}VV\\
U @>{\nu}>> V
\end{CD}\]
is commutative.
For the proof of Theorem~\ref{thm:graph2toric},
it is enough to prove that \( \tau \) is \'etale along \( \Sigma \).
Applying \cite[III, Th\'eor\`eme~(4.1.5)]{EGA},
we have isomorphisms
\[ \SO_{Y}^{\wedge} \isom \varprojlim\nolimits_{m}
\mu_{*}(\SO_{M}/\SO_{M}(-mE)) \quad \text{ and } \quad
\SO_{V}^{\wedge} \isom \varprojlim\nolimits_{m}
\nu_{*}(\SO_{U}/\SO_{U}(-mG)), \]
where \( E = \sum\nolimits_{i = 1}^{l} E_{i} \),
\( G = \sum\nolimits_{i = 1}^{l}G_{i} \), and
\( \SO_{Y}^{\wedge} \) (resp.\ \( \SO_{V}^{\wedge} \))
is the formal completion of
\( \SO_{Y} \isom \mu_{*}\SO_{M} \) (resp.\ \( \SO_{V} \isom \nu_{*}\SO_{U} \))
along \( \Sigma \) (resp.\ \( \BSigma \)).
On the other hand, \( \Phi \) induces an isomorphism
\[ \OH^{0}(U, \SO_{U}/\SO_{U}(-mG)) \isom \OH^{0}(M, \SO_{M}/\SO_{M}(-mE)) \]
of \( \Lambda \)-algebras for all \( m > 0 \), since
\( \Phi^{*}(G) = E \) and
\( \Phi|_{E} \colon E \to G \) is an isomorphism.
Thus, \( \tau \) induces an isomorphism
\( \SO_{U}^{\wedge} \isom \SO_{Y}^{\wedge}  \).
Hence, the morphism \( \tau_{t} \colon Y_{t} \to V_{t} \)
between the fibers of \( V \) and \( U \) over any point \( t \in \Spec \Lambda \)
induced by \( \tau \) is \'etale at the point \( \BSigma \cap Y_{t} \) by
\cite[IV, Th\'eor\`eme~(17.6.1) or Proposition~(17.6.3)]{EGA}.
Thus, \( \tau \) is \'etale along \( \Sigma \) by \cite[IV, Proposition~(17.8.2)]{EGA},
and this completes the proof.
\end{proof}

In the rest of Section~\ref{sect:dualgraph},
we shall show some local properties of
surface singularities that have a linear chain of smooth rational curves
as the exceptional divisor for the minimal resolution.

\begin{lem}\label{lem:imageGG000}
Let \(V\) be the toric surface \( \BTT_{\NN}(\Bsigma) \)
of type \((n, q)\) defined over \( \Bbbk \) and
let \(D_1\), \(D_2\), \(\nu \colon U \to V\), \(G_i\) be the same as above,
except that the singular point of \( V \) is denoted
by \( \bzero \) instead of \( \BSigma \).
Then, there are natural isomorphisms\emph{:}
\begin{align*}
\nu_{*}\SO_{U}(G_{0}) &\isom \SO_{V}(D_{2}), &
\nu_{*}\SO_{G_{0}}(G_{0}) &\isom \SExt^{1}_{\SO_{V}}(\SO_{D_{2}}, \SO_{V}), \\
\nu_{*}\SO_{U}(G_{l+1}) &\isom \SO_{V}(D_{1}), &
\nu_{*}\SO_{G_{l+1}}(G_{l+1}) &\isom \SExt^{1}_{\SO_{V}}(\SO_{D_{1}}, \SO_{V}).
\end{align*}
\end{lem}

\begin{proof}
Let \( \Bbbk(V) = \Bbbk(U) \) be the rational function field of \( V \)
and \( U \).
For a non-zero rational function
\( \varphi \in \Bbbk(V) \), if \( \Div(\varphi)_{V} + D_{2} \geq 0\)
for the principal divisor
\( \Div(\varphi)_{V} \) on \( V \), then we have
\[ n\Div(\varphi)_{U} + nG_{0}
= \nu^{*}(\Div(\varphi^{n})_{V} + nD_{2}) - \sum\nolimits_{i = 1}^{l} q_{i}G_{i}
\geq - \sum\nolimits_{i = 1}^{l} q_{i}G_{i}\]
for the principal divisor \( \Div(\varphi)_{U} \) on \( U \) by
applying \eqref{eq:muDD} in Remark~\ref{remsub:lem:VDDC}.
Thus,
\( \Div(\varphi)_{U} + G_{0} \geq 0 \), since \( q_{i} < n \)
for all \( 1 \leq i \leq l \).
On the other hand, if a non-zero rational function \( \psi \in \Bbbk(U)\)
satisfies
\( \Div(\psi)_{U} + G_{0} \geq 0\), then
\[ \Div(\psi)_{V} + D_{2} = \nu_{*}(\Div(\psi)_{U} + G_{0}) \geq 0.  \]
Therefore,
\( \OH^{0}(V, \SO_{V}(D_{2})) = \OH^{0}(U, \SO_{U}(G_{0})) \), and
we have an isomorphism \( \nu_{*}\SO_{U}(G_{0}) \isom \SO_{V}(D_{2}) \).
By applying \( \nu_{*} \) to the exact sequence
\( 0 \to \SO_{U} \to \SO_{U}(G_{0}) \to \SO_{G_{0}}(G_{0}) \to 0 \), we have
another exact sequence
\[ 0 \to \SO_{V} \to \SO_{V}(D_{2}) \to \nu_{*}\SO_{G_{0}}(G_{0}) \to 0, \]
since \( \OR^{1}\nu_{*}\SO_{U} = 0 \) by the rationality of
the toric singularity or by Corollary~\ref{corsub:lem:Lipman000}.
Hence, \( \nu_{*}\SO_{G_{0}}(G_{0}) \isom
\SExt^{1}_{\SO_{V}}(\SO_{D_{2}}, \SO_{V})\)
is derived from \( 0 \to \SO_{V}(-D_{2}) \to \SO_{V} \to \SO_{D_{2}} \to 0 \) by
applying \( \SExt^{i}_{\SO_{V}}(\bullet, \SO_{V}) \).
The other isomorphisms concerning \( G_{l+1} \) and \( D_{1} \) are obtained by
a similar way.
\end{proof}

Before going to Lemma~\ref{lem:Res},
we recall some basics on the ``sheaf of logarithmic one-forms'' on
the toric surface \( V = \BTT_{\NN}(\Bsigma) \) and
the residue homomorphism.
The sheaf \( \Omega^{1}_{\BTT_{\NN}/\Bbbk} \) of
one-forms on the torus \( \BTT_{\NN} = \Spec \Bbbk[\MM] \) is
trivial by the isomorphism
\[ \MM \otimes_{\BZZ} \SO_{\BTT_{\NN}} \xrightarrow{\isom}
\Omega^{1}_{\BTT_{\NN}/\Bbbk} \]
which maps \( m \otimes 1 \) to \( \Be(m)^{-1}\dd \Be(m) \)
for any \( m \in \MM \).
Regarding \( \BTT_{\NN} \) as the open subset
\( \BTT_{\NN}(\{0\}) = V \setminus D \)
of \( V = \BTT_{\NN}(\Bsigma) \), where \( D = D_{1} + D_{2} \),
we can extend the isomorphism to
\[ \theta \colon \MM \otimes_{\BZZ} \SO_{V} \xrightarrow{\isom}
\widetilde{\Omega}^{1}_{V/\Bbbk}(\log D) :=
j_{*}\Omega^{1}_{V^{\circ}/\Bbbk}(\log D^{\circ}),  \]
where \( j \) is the open immersion
\( V^{\circ} := V \setminus \{\bzero\} \subset V \), and
\( D^{\circ} \) is the divisor \( D|_{V^{\circ}} \)
(cf.\ \cite[(1.12), Proposition]{Ishida-Oda}, \cite[\S 4.3, Proposition]{Fulton}).
Note that \( V^{\circ} \) and \( D^{\circ} \) are non-singular.
The logarithmic tangent sheaf
\[ \Theta_{V/\Bbbk}(-\log D) :=
\SHom_{\SO_{V}}(\widetilde{\Omega}_{V/\Bbbk}^{1}(\log D), \SO_{V}) \]
is isomorphic to \( \NN \otimes_{\BZZ} \SO_{V} \)
by the dual of \( \theta \).
The double-dual \( (\Omega^{1}_{V/\Bbbk})^{\vee\vee} \) of the sheaf
\( \Omega^{1}_{V/\Bbbk} \) of relative one-forms is just
the kernel of the residue homomorphism
\[ \Res \colon \widetilde{\Omega}_{V/\Bbbk}^{1}(\log D) \to
\SO_{D_{1}} \oplus \SO_{D_{2}}.   \]
This residue homomorphism is given by the evaluation map
\[ \ev \colon \MM \ni m \mapsto
(m(e_{1}), m(e_{2})) \in \BZZ \oplus \BZZ  \]
in the sense that there is a commutative diagram
\[ \begin{CD}
\MM \otimes_{\BZZ}  \SO_{V} @>{\ev \otimes \id}>>
(\BZZ \oplus \BZZ) \otimes \SO_{V} \\
@V{\theta}V{\isom}V @VVV \\
\widetilde{\Omega}^{1}_{V/\Bbbk}(\log D) @>{\Res}>> \SO_{D_{1}} \oplus \SO_{D_{2}}
\end{CD}\]
where the right vertical arrow is the direct sum of the natural homomorphisms
\( \SO_{V} \to \SO_{D_{1}} \) and \( \SO_{V} \to \SO_{D_{2}} \).

\begin{lem}\label{lem:Res}
The residue homomorphism \( \Res \) is surjective
if and only if \( p \nmid n, \) where \( p = \chara \Bbbk \).
If \( p\mid n \), then
the cokernel of \( \Res \) is
the skyscraper sheaf at \( \bzero \)
corresponding to the residue field \( \Bbbk(\bzero) \) of \( \SO_{V, \bzero}\), and
the image of \( \Res \) is isomorphic to \( \SO_{D} \).
\end{lem}

\begin{proof}
The subgroup \( \MM \) of \(\MM_{0} = \BZZ h_{1} + \BZZ h_{2} \) is
generated by elements \( a_{1}h_{1} + a_{2}h_{2} \) such that
\( a_{1} + qa_{2} \equiv 0 \mod n \).
Here, \( \Res \circ \theta \) maps \( (a_{1}h_{1} + a_{2}h_{2}) \otimes 1 \) to
\( (a_{1}, a_{2})
\in \OH^{0}(D_{1}, \SO_{D_{1}}) \oplus \OH^{0}(D_{2}, \SO_{D_{2}}) \).
We shall show that
the cokernel of \( \Res \) is isomorphic to \( \Bbbk(\bzero)/n\Bbbk(\bzero) \) by
the homomorphism
\[ \phi \colon \SO_{D_{1}} \oplus \SO_{D_{2}} \ni
(\alpha_{1}, \alpha_{2}) \mapsto
(\alpha_{1}|_{\bzero} + q\alpha_{2}|_{\bzero}) \bmod n
\, \in \Bbbk(\bzero)/n\Bbbk(\bzero).\]
The homomorphism \( \phi \) is in the commutative diagram
\[ \begin{CD}
0 @>>> \SO_{V}(-D_{1}) \oplus \SO_{V}(-D_{2}) @>>>
M_{0} \otimes_{\BZZ} \SO_{V} @>>> \SO_{D_{1}} \oplus \SO_{D_{2}} @>>> 0 \\
@. @V{\phi'}VV @V{\phi''}VV @V{\phi}VV \\
@. (\BZZ/n\BZZ) \otimes_{\BZZ} \SI @>>> \SO_{V}/n\SO_{V}
@>>> \Bbbk(\bzero)/n\Bbbk(\bzero) @>>> 0
\end{CD}
\]
of exact sequences, where \( \SI \) is the defining ideal of \( \bzero \), and
the homomorphisms \( \phi' \) and \( \phi'' \) are also induced by
\[ \SO_{V} \oplus \SO_{V} \ni (\beta_{1}, \beta_{2})
\mapsto \beta_{1} + q\beta_{2} \in \SO_{V}. \]
Here, \( \phi' \) is surjective,
since \( \SI = \SO_{V}(-D_{1}) + \SO(-D_{2}) \)
(cf.\ Remark~\ref{remsub:dfn:dfn:toricnq:D1D2}) and
since \( \gcd(n, q) = 1 \). The kernel of \( \phi'' \) is just
the image of the natural homomorphism
\( M \otimes_{\BZZ} \SO_{V}\to M_{0}\otimes \SO_{V} \).
Thus, the image of
\( \Res \colon M \otimes_{\BZZ} \SO_{V} \to \SO_{D_{1}} \oplus \SO_{D_{2}} \) is
just the kernel of \( \phi \).
In particular, \( \Res \) is surjective if and only if \( p \nmid n \).

Assume that \( p \mid n \). Then, \( p \nmid q \) by \( \gcd(n, q) = 1 \).
Hence, the image of \( \Res \) is also the kernel of
\[ \SO_{D_{1}} \oplus \SO_{D_{2}} \ni
(\alpha_{1}, \alpha_{2}) \mapsto
\alpha_{1}|_{\bzero} + \alpha_{2}|_{\bzero}
\in \Bbbk(\bzero).\]
Thus, the image of \( \Res \) is just \( \SO_{D} \),
since  we have a natural exact sequence
\( 0 \to \SO_{D} \to \SO_{D_{1}} \oplus \SO_{D_{2}} \to \Bbbk(\bzero) \to 0 \).
\end{proof}

\begin{prop}\label{prop:tangentsheaves}
Let \( Y \) be a normal algebraic surface over an algebraically closed
field \( \Bbbk \) and let \( \mu \colon M \to Y \) be
the minimal resolution of singularities.
Assume that any connected component of the \( \mu \)-exceptional
locus \( E \) is a linear chain of rational curves.
Then, the following hold\emph{:}
\begin{enumerate}
\item \label{prop:tangentsheaves:isom} The natural injection
\( \mu_{*}\Theta_{M/\Bbbk}(-\log E) \to \mu_{*}\Theta_{M/\Bbbk}\)
is an isomorphism.

\item  \label{prop:tangentsheaves:vanish}
\( \OR^{i}\mu_{*}\Theta_{M/\Bbbk}(-\log E) = 0  \) for all \( i > 0 \).

\item  \label{prop:tangentsheaves:reflexive0}
The direct image sheaf \( \mu_{*}\Omega^{1}_{M/\Bbbk} \)
is reflexive. In other words,
\( (\Omega^{1}_{Y/\Bbbk})^{\vee\vee} \isom \mu_{*}\Omega^{1}_{M/\Bbbk} \).

\item  \label{prop:tangentsheaves:reflexive}
The natural injection \( \mu_{*}\Theta_{M/\Bbbk} \injmap \Theta_{Y/\Bbbk} \)
is not an isomorphism if and only if
\( p \mid n \) and \( q = n-1 \).

\end{enumerate}
\end{prop}

\begin{proof}
Since the assertions are \'etale local on \( Y \), we may assume that
\( Y \) is the toric surface \( V = \BTT_{\NN}(\Bsigma)\)
by Theorem~\ref{thm:graph2toric000}.
Let \( D_{1} \), \( D_{2} \), \( \bzero \), \( \nu \colon U \to V \),
and \( G_{i} \) be as before. Then, the minimal resolution \( M \to Y \) is
just \( U \to V \) and \( E = G = \sum\nolimits_{i = 1}^{l} G_{i}\).
We set \( \widehat{G} = \sum\nolimits_{i = 0}^{l+1} G_{i} = G + G_{0} + G_{l+1} \).
Since \( G \) and \( \widehat{G} \) are
simple normal crossing divisors,
we have a commutative diagram
\[ \begin{CD}
0 @>>> \Theta_{U/\Bbbk}(-\log \widehat{G}) @>>> \Theta_{U/\Bbbk} @>>>
\bigoplus\nolimits_{i = 0}^{l+1} \SO_{G_{i}}(G_{i}) @>>> 0\\
@. @VVV @| @VVV \\
0 @>>> \Theta_{U/\Bbbk}(-\log G) @>>> \Theta_{U/\Bbbk} @>>>
\bigoplus\nolimits_{i = 1}^{l} \SO_{G_{i}}(G_{i}) @>>> 0
\end{CD}\]
of exact sequences. The assertion \eqref{prop:tangentsheaves:isom}
is derived from the bottom sequence by taking \( \nu_{*} \), since
\( \nu_{*}\SO_{G_{i}}(G_{i}) = \OH^{0}(\BPP^{1}, \SO(-b_{i})) = 0 \)
for \( 1 \leq i \leq l \).
The commutative diagram above induces an exact sequence
\begin{equation}\label{eq:exactGG}
0 \to \Theta_{U/\Bbbk}(-\log \widehat{G}) \to \Theta_{U/\Bbbk}(-\log G) \to
\SO_{G_{0}}(G_{0}) \oplus \SO_{G_{l+1}}(G_{l+1}) \to 0.
\end{equation}
Here, \( \Theta_{U/\Bbbk}(-\log \widehat{G}) \isom \NN \otimes_{\BZZ} \SO_{U}\),
since \( \widehat{G} \) is the complement of the torus \( \BTT_{\NN} \)
in \( U = \BTT_{\NN}(\triangle) \).
Now, \( \OR^{i}\nu_{*}\SO_{U} = 0 \) for \( i > 0 \),
since \( V \) has only rational singularities
(cf.\ Corollary~\ref{corsub:lem:Lipman000}).
We have also \( \OR^{i}\nu_{*}\SO_{G_{0}}(G_{0}) =
\OR^{i}\nu_{*}\SO_{G_{l+1}}(G_{l+1}) = 0 \)
for \( i > 0 \), since \( \nu \)
induces isomorphisms \( G_{0} \to D_{2}\) and \( G_{l+1} \to D_{1} \).
Thus, the assertion \eqref{prop:tangentsheaves:vanish} is obtained by
applying \( \OR^{i}\nu_{*} \) to the exact sequence \eqref{eq:exactGG}.

For the assertion \eqref{prop:tangentsheaves:reflexive0},
we consider the commutative diagram
\[ \begin{CD}
\nu_{*}\Omega^{1}_{U/\Bbbk}(\log \widehat{G}) @>{\nu_{*}\Res}>>
\bigoplus\nolimits_{i = 0}^{l+1} \nu_{*}\SO_{G_{i}} \\
@A{\alpha}AA @AA{\beta}A \\
\widetilde{\Omega}^{1}_{V/\Bbbk}(\log D) @>{\Res}>>
\SO_{D_{1}} \oplus \SO_{D_{2}}
\end{CD}\]
obtained by comparing the residue homomorphisms on \( V \) and \( U \).
Here, \( \alpha \) is an isomorphism, since
\( \widetilde{\Omega}^{1}_{V}(\log D) \isom \MM \otimes_{\BZZ} \SO_{V} \)
and \( \Omega^{1}_{U}(\log \widehat{G}) \isom \MM \otimes_{\BZZ} \SO_{U} \).
The map \( \beta \) is an isomorphism
to \( \nu_{*}\SO_{E_{l+1}} \oplus \nu_{*}\SO_{E_{0}} \).
Thus, we have a homomorphism
\[ \gamma \colon (\Omega^{1}_{V/\Bbbk})^{\vee\vee}
\to \nu_{*}\Omega^{1}_{U/\Bbbk} \]
as the induced homomorphism between the kernels of the top and bottom
homomorphisms.
Then, \( \gamma \) is an isomorphism,
since it is an isomorphism on \( V \setminus \{\bzero\} \),
the source is reflexive, and the target is torsion free.
Hence, we have \eqref{prop:tangentsheaves:reflexive0}.

It remains to prove \eqref{prop:tangentsheaves:reflexive}.
Let \( \SF \) be the image of
\( \Res \colon \widetilde{\Omega}^{1}_{V/\Bbbk}(\log D)
\to \SO_{D_{1}} \oplus \SO_{D_{2}} \).
Then, by \eqref{eq:exactGG} and by Lemma~\ref{lem:imageGG000},
we have a commutative diagram
\[ \begin{CD}
0 @>>> \Theta_{V/\Bbbk}(-\log D) @>>> \nu_{*}\Theta_{U/\Bbbk}(-\log G) @>>>
\SExt^{1}_{\SO_{V}}(\bigoplus\nolimits_{i = 1}^{2} \SO_{D_{i}},
\SO_{V}) @>>> 0\\
@. @| @VVV @VVV \\
0 @>>> \Theta_{V/\Bbbk}(-\log D) @>>>
\Theta_{V/\Bbbk} @>>> \SExt^{1}_{\SO_{V}}(\SF, \SO_{V}) @>>> 0
\end{CD}\]
of exact sequences, in which the bottom one is obtained from
\[ 0 \to (\Omega^{1}_{V/\Bbbk})^{\vee\vee} \to
\widetilde{\Omega}^{1}_{V/\Bbbk}(\log D) \to \SF \to 0 \]
by taking \( \SExt^{i}_{\SO_{V}}(\bullet, \SO_{V}) \).
Hence, if \( p \nmid n \), then
\( \SF = \SO_{D_{1}} \oplus \SO_{D_{2}} \) by Lemma~\ref{lem:Res} and
\( \Theta_{V/\Bbbk} \isom \nu_{*}\Theta_{U/\Bbbk}(-\log G) \)
by the commutative diagram above.
Thus, we may assume that \( p \mid n \).
Then, by Lemma~\ref{lem:Res}, \( \SF = \SO_{D} \), and we obtain an exact sequence
\begin{multline*}
\cdots \to \SExt^{i}_{\SO_{V}}(\Bbbk(\bzero), \SO_{V}) \to
\SExt^{i}_{\SO_{V}}(\SO_{D_{1}} \oplus \SO_{D_{2}}, \SO_{V}) \\
\to \SExt^{i}_{\SO_{V}}(\SO_{D}, \SO_{V}) \to
\SExt^{i+1}_{\SO_{V}}(\Bbbk(\bzero), \SO_{V}) \to \cdots.
\end{multline*}
It is enough to determine when the cokernel of
\[ \SExt^{1}_{\SO_{V}}(\SO_{D_{1}} \oplus \SO_{D_{2}}, \SO_{V})
\to \SExt^{1}_{\SO_{V}}(\SO_{D}, \SO_{V}) \]
is not zero. By a standard argument, we see that the cokernel is isomorphic to that of
the homomorphism
\( \SO_{V}(D_{1}) \oplus \SO_{V}(D_{2}) \to \SO_{V}(D) \)
induced from the natural inclusions \( \SO_{V}(D_{1}) \injmap \SO_{V}(D) \) and
\( \SO_{V}(D_{2}) \injmap \SO_{V}(D) \).
Now, \( \OH^{0}(V, \SO_{V}(D)) \) is generated by the rational functions
\( \varphi \) on \( V \) such that \( \Div(\varphi) + D \geq 0 \).
Hence, \( \OH^{0}(V, \SO_{V}(D)) \) is generated by the following monomials
\( \xtt_{1}^{a_{1}}\xtt_{2}^{a_{2}} \)
as an \( \OH^{0}(V, \SO_{V}) \)-submodule of \( \Bbbk[\MM] = \OH^{0}(\BTT_{\NN}, \SO)\):
\begin{itemize}
\item  \( a_{1} \) and \( a_{2} \) are integers with \( a_{1} + qa_{2} \equiv 0 \mod n\).

\item  \( a_{1} \geq -1 \) and \( a_{2} \geq -1 \)
(cf.\ \eqref{eq:principaldiv000}).
\end{itemize}
Similarly, \( \OH^{0}(V, \SO_{V}(D_{1})) + \OH^{0}(V, \SO_{V}(D_{2}))\)
is generated by the following monomials \( \xtt_{1}^{b_{1}}\xtt_{2}^{b_{2}} \) as
an \( \OH^{0}(V, \SO_{V})  \)-submodule of \( \Bbbk[\MM] \):
\begin{itemize}
\item  \( b_{1} \) and \( b_{2} \) are integers with
\( b_{1} + qb_{2} \equiv 0 \mod n \).

\item  \( b_{1} \geq -1 \) and \( b_{2} \geq -1 \), but
\( \max\{b_{1}, b_{2}\} \geq 0 \).
\end{itemize}
Therefore, the cokernel is not zero if and only if
the monomial \( \xtt_{1}^{-1}\xtt_{2}^{-1} \) is contained
in \( \OH^{0}(V, \SO_{V}(D)) \).
This is just the case where \( q = n - 1 \).
Thus \eqref{prop:tangentsheaves:reflexive} is proved, and
we have finished the proof of Proposition~\ref{prop:tangentsheaves}.
\end{proof}

\begin{remsub}
If the injection of the assertion \eqref{prop:tangentsheaves:reflexive} is
an isomorphism, then \( \mu \)
is an ``equivariant resolution of singularities''
in the sense of Hironaka.
Hence, for toric singularities of type \( (n, q = n - 1) \),
equivariant resolutions do exist
if and only if \( p \nmid n \)
(cf.\ \cite[Theorem]{Wahl75} and an example of characteristic two
in \cite[page~345]{Artin74}).
\end{remsub}

\begin{cor}\label{cor:compareTangent}
In the situation of Proposition~\emph{\ref{prop:tangentsheaves}},
one has an isomorphism
\[ \OH^{2}(M, \Theta_{M/\Bbbk}(-\log E)) \xrightarrow{\isom}
\OH^{2}(Y, \Theta_{Y/\Bbbk}). \]
\end{cor}

\begin{proof}
Since \( \Theta_{Y/\Bbbk} \) is the double-dual of
\( \mu_{*}\Theta_{M/\Bbbk}(-\log E)\),
there is an exact sequence
\[ 0 \to \mu_{*}\Theta_{M/\Bbbk}(-\log E) \to \Theta_{Y/\Bbbk} \to \SG \to 0\]
for a coherent sheaf \( \SG \) with \( \dim \Supp \SG \leq 0\), which
induces an isomorphism
\[ \iota_{1} \colon \OH^{2}(Y, \mu_{*}\Theta_{M/\Bbbk}(-\log E)) \isom
\OH^{2}(Y, \Theta_{Y/\Bbbk}).\]
On the other hand, since
\( \OR^{i}\mu_{*}\Theta_{M/\Bbbk}(-\log E) = 0 \) for all \( i > 0 \) by
Proposition~\ref{prop:tangentsheaves}.\eqref{prop:tangentsheaves:vanish},
the Leray spectral sequence
for \( \mu \) induces an isomorphism
\[ i_{2} \colon \OH^{2}(X, \mu_{*}\Theta_{M/\Bbbk}(-\log E)) \isom
\OH^{2}(M, \Theta_{M/\Bbbk}(-\log E)). \]
Therefore, we obtain the claimed isomorphism as \( i_{1} \circ i_{2}^{-1} \).
\end{proof}


\section{Toric singularity of class T}
\label{sect:DefCT}

In this section, we introduce the notion of toric singularity
of class T and study its properties.
We first discuss some invariants
arising from toric singularities of class T:
The results here are already known in papers
such as \cite{WahlElliptic}, \cite{LW}, \cite{KSh},
\cite{Stevens91}, \cite{Manetti91}, \cite{Lee99},
but we shall give a self-contained proof.
Tables \ref{table:M} and \ref{table:M2} obtained here
are used in some calculations in Sections~\ref{sect:Global}
and \ref{sect:proof}.
Second, in Theorem~\ref{thm:localsmoothing} below,
we shall construct a special
smoothing (deformation) of toric singularities of class T,
which plays an important role in producing new surfaces.

\begin{dfn}\label{dfn:toricsing}
Let \( X \) be a normal algebraic surface defined over \( \Bbbk \) and
\( x \) a closed point. The germ \( (X, x) \) (in the \'etale topology)
is said to be a ``toric singularity of type \( (n, q) \)''
if the formal completion of \( \SO_{X, x} \)
is isomorphic to the formal completion of \( \SO_{V, \bzero} \)
for an affine toric surface \( V \) of type \( (n, q) \)
(cf.\ Definition~\ref{dfn:toricnq}) over \( \Bbbk \)
and the zero-dimensional orbit \( \bzero \).
\end{dfn}

\begin{remsub}
Note that by \cite[Corollary~2.6]{ArtinApprox},
the condition in Definition~\ref{dfn:toricsing}
is equivalent to the existence of a common \'etale neighborhood
of \( (X, x) \) and \( (V, \bzero) \).
\end{remsub}

\begin{remsub}
By Theorem~\ref{thm:graph2toric}, \( (X, x) \) is a toric singularity
if and only if the exceptional locus of the minimal resolution is
a linear chain of smooth rational curves.
\end{remsub}

\begin{dfn}\label{dfn:classT}
Let \( \ST_{\DNA}\)
be the set of triples \( (d, n, a) \) of positive integers
with \( n > a \) and \( \gcd(n, a) = 1 \).
A two-dimensional surface singularity is said to be of type \( T(d, n, a) \)
for a triplet \( (d, n, a) \in \ST_{\DNA} \) if it is a toric singularity of type
\( (dn^{2}, dna - 1) \).
The singularities of ``class T'' are the singularities of
type \( T(d, n, a) \) for all \( (d, n, a) \in \ST_{\DNA} \)
(cf.\ \cite[Proposition~3.10]{KSh}, \cite[\S4]{Manetti91}).
\end{dfn}

\begin{remsub}
The definition of class T in \cite{KSh} is different from ours.
Our definition of class T corresponds to that of
non-Gorenstein class T in \cite{KSh}.
\end{remsub}

Before going to the study of toric singularities of class T,
we prepare some invariants arising from each element of \( \ST_{\DNA} \).
Let \( (d, n, a) \) be a triplet in \(\ST_{\DNA} \).
We can define positive integers \( l \), \( b_{1} \), \ldots, \( b_{l} \)
by the property that
\( l \geq 1 \), \( b_{i} \geq 2 \) for all \( 1 \leq i \leq l\),
and \( dn^{2}/(dna - 1) = [b_{1}, \ldots, b_{l}] \).
Then, \( (b_{1}, \ldots, b_{l}) \ne (2, 2, \ldots, 2) \); for otherwise,
\( dn^{2} = dna \) contradicting \( n > a \).
By Lemma~\ref{lem:v_i}, we can define also non-negative integers
\( p_{i} \) and \( q_{i} \) for \( 0 \leq i \leq l + 1\)
by the following properties:
\begin{itemize}
\item  \( p_{0} = 0 < p_{1} = 1 < p_{2} < \cdots < p_{l} < p_{l+1} = dn^{2} \).

\item  \( q_{0} = dn^{2} > q_{1} = dna - 1 > q_{2} > \cdots > q_{l} = 1 > q_{l+1} = 0\)

\item  \( p_{i-1} + p_{i+1} = b_{i}p_{i}\) and \( q_{i-1} + q_{i+1} = b_{i}q_{i} \)
for all \( 1 \leq i \leq l \).
\end{itemize}
We set \( r_{i} := (p_{i} + q_{i})/(dn) \) for \( 0 \leq i \leq l+1 \).
Then, \( r_{0} = r_{l+1} = n \) and \( r_{1} = a \). Moreover, we have:

\begin{lem}\label{lem:ripl}
\begin{enumerate}
\item \label{lem:ripl:1}
\( r_{i} \) is a positive integer with \( 1 \leq r_{i} < n\)
for all \( 1 \leq i \leq l \).
\item \label{lem:ripl:2}
\( r_{i} \equiv ap_{i} \equiv -aq_{i} \mod n\) for all \( 1 \leq i \leq l \).
\item \label{lem:ripl:3}
\( r_{l} = n - a \), equivalently, \( p_{l} = dn(n-a) - 1 \).
\end{enumerate}
\end{lem}

\begin{proof}
The assertions \eqref{lem:ripl:1} and \eqref{lem:ripl:2}
follow from the convexity \( r_{i-1} + r_{i+1} = b_{i}r_{i} \) and
\( p_{i}q_{1} \equiv q_{i} \mod dn^{2} \)
for \( 1 \leq i \leq l \) (cf.\ Lemma~\ref{lem:v_i}.\eqref{lem:v_i:5}).
The last assertion \eqref{lem:ripl:3} is a consequence of the previous two assertions.
\end{proof}

\begin{dfn}
For \( (d, n, a) \in \ST_{\DNA} \), we define:
\begin{gather*}
\begin{xalignat*}{3}
B(d, n, a) &:= (b_{1}, b_{2}, \ldots, b_{l}), &
P(d, n, a) &:= (p_{1}, p_{2}, \ldots, p_{l}), &
Q(d, n, a) &:= (q_{1}, q_{2}, \ldots, q_{l}), \\
R(d, n, a) &:= (r_{1}, r_{2}, \ldots, r_{l}), &
C(d, n, a) &:= (c_{1}, c_{2}, \ldots, c_{l}) &
\text{where} \quad c_{i} &:= 1 - r_{i}/n,
\end{xalignat*}\\
\delta(d, n, a) := \sum\nolimits_{i = 1}^{l} b_{i} - (2l + 1),
\quad \text{ and } \quad
l(d, n, a) := l.
\end{gather*}
\end{dfn}

The following characterization of toric singularities of class T
is well-known:

\begin{lem}[{\cite[Proposition~5.9]{LW}}]\label{lem:dnaInteger}
Let \( (X, x) \) be a normal surface singularity such that the exceptional locus
of the minimal resolution \( \mu \colon M \to (X, x) \)
of singularity is a linear chain of smooth rational curves.
Then, the following two conditions are mutually equivalent\emph{:}
\begin{enumerate}
\item \label{lem:dnaInteger:1}
\( \Delta^{2} \) is a negative integer for the effective \( \BQQ \)-divisor
\( \Delta = \mu^{*}(K_{X}) - K_{M}\).

\item \label{lem:dnaInteger:2}
\( (X, x) \) is a toric surface singularity of class T.
\end{enumerate}
Moreover, if \( (X, x) \) is a singularity of type \( T(d, n, a) \),
then \( \Delta^{2} = -\delta(d, n, a) \),
and \( \Delta = \sum\nolimits_{i = 1}^{l} c_{i}E_{i} \)
for the linear chain \( E_{1} + \cdots + E_{l} \) of smooth rational curves
and for \( C(d, n, a) = (c_{1}, \ldots, c_{l}) \).
\end{lem}

\begin{proof}
By Theorem~\ref{thm:graph2toric},
we may assume that \( X \) is a toric surface \( V \) of type \( (n, q) \)
for some positive
integers \( n \), \( q \) with \( n > q \) and \( \gcd(n, q) = 1 \).
Thus, we can use the description of the minimal resolution
\( \nu \colon U \to V \) of the toric surface
given in Section~\ref{sect:dualgraph}.
Note that
\( \nu^{*}(K_{V} + D_{1} + D_{2}) =
K_{U} + \sum\nolimits_{i = 0}^{l+1} G_{i} \sim 0\).
Thus, we have
\[
\Delta = \nu^{*}(K_{V}) - K_{U} =
\sum\nolimits_{i = 0}^{l+1} G_{i} - \nu^{*}(D_{1} + D_{2})
=
\sum\nolimits_{i = 1}^{l} (1 - \frac{p_{i} + q_{i}}{n}) G_{i}
\]
for integers \( p_{i} \), \( q_{i} \) in Lemma~\ref{lem:v_i},
by \eqref{eq:muaD000} in Remark~\ref{remsub:lem:VDDC}.
In particular, if \( (n, q) = (dm^{2}, dma - 1) \)
for a triplet \( (d, m, a) \in \ST_{\DNA} \), then we have
\( \Delta = \sum_{i = 1}^{l} c_{i}G_{i} \) for
\( C(d, m, a) = (c_{1}, \ldots, c_{m}) \).
On the other hand,
\[ \Delta G_{i} = -K_{U}G_{i} =  2 + G_{i}^{2} = 2 - b_{i} \]
for all \( 1 \leq i \leq l \) by adjunction,
where \( n/q = [b_{1}, \ldots, b_{l}] \). Hence,
\[ \Delta^{2} = \sum\nolimits_{i = 1}^{l}
(1 - \frac{p_{i} + q_{i}}{n}) (2 - b_{i}), \]
and it is an integer if and only if
\[ \sum\nolimits_{i = 1}^{l} (p_{i} + q_{i})(2 - b_{i}) \equiv 0 \mod n. \]
Since \( b_{i} (p_{i}, q_{i}) = (p_{i-1}, q_{i-1}) + (p_{i+1}, q_{i+1}) \)
(cf.\ Lemma~\ref{lem:v_i}), we have
\begin{align*}
&\sum\nolimits_{i = 1}^{l} (p_{i} + q_{i})(2 - b_{i}) =
2\sum\nolimits_{i = 1}^{l} (p_{i} + q_{i}) -
\sum\nolimits_{i = 1}^{l} (p_{i-1} + q_{i-1})
- \sum\nolimits_{i = 1}^{l} (p_{i+1} + q_{i+1}) \\
&= p_{1} + p_{l} + q_{1} + q_{l}- (p_{0} + p_{l+1} + q_{0} + q_{l+1})
= q + q' + 2 - 2n,
\end{align*}
where \( 0 < q' < n\) with \( qq' \equiv 1 \bmod n \)
(cf.\ Lemma~\ref{lem:v_i}.\eqref{lem:v_i:5}).
Thus, \( \Delta^{2} \in \BZZ \) if and only if
\( q + q' + 2 \equiv 0 \bmod n \).
Since \( \gcd(n, q) = 1 \), this condition is equivalent to
\[ (q + 1)^{2} = q^{2} + 2q + 1 \equiv q(q + 2 + q') \equiv 0 \mod n. \]
By considering the prime factorization, we see that this is also equivalent to
either
\begin{enumerate}
    \renewcommand{\theenumi}{\roman{enumi}}
    \renewcommand{\labelenumi}{(\theenumi)}
\item \label{lem:dnaInteger:case1} \( n = q + 1 \), or
\item \label{lem:dnaInteger:case2} \( n = dm^{2} \) and \( q+1 = dma\) for
some \( (d, m, a) \in \ST_{\DNA} \).
\end{enumerate}
In case \eqref{lem:dnaInteger:case1},
\( (b_{1}, \ldots, b_{l}) = (2, 2, \ldots, 2) \) and
\( \Delta^{2} = 0 \).
In case \eqref{lem:dnaInteger:case2}, we have
\( q + q' + 2 = dm^{2}\) by Lemma~\ref{lem:ripl}.\eqref{lem:ripl:3}, and
\[
\Delta^{2} = \sum\nolimits_{i = 1}^{l} (2 - b_{i}) -
\frac{1}{dm^{2}}(q + q' + 2 - 2dm^{2})
= 2l + 1 -\sum\nolimits_{i = 1}^{l}b_{i} = -\delta(d, m, a) < 0.
\]
Thus, we are done.
\end{proof}

\begin{corsub}\label{corsub:lem:dnaInteger}
Let \( X \) be a normal projective surface whose non-Gorenstein
singularities are toric singularities of class T.
Then, \( K_{X}^{2} \) is an integer.
\end{corsub}

\begin{lem}[{cf.\ \cite[Proposition~3.11]{KSh}}]\label{lem:n=2}
For a triplet \( (d, n, a) \in \ST_{\DNA}\),
either \( b_{1} \geq 3 \) or \( b_{l} \geq 3 \) holds
for \( B(d, n, a) = (b_{1}, \ldots, b_{l}) \).
Assume that \( b_{1} \geq 3 \) and \( b_{l} \geq 3 \). Then,
\( (d, n, a) = (l, 2, 1) \).
Here,
\( B(1, 2, 1) = (4) \), \( B(2, 2, 1) = (3, 3) \),
and \(B(l, 2, 1) = (3, 2, \ldots, 2, 3) \) for \( l \geq 3\).
Moreover, for all \( 1 \leq i \leq l \),
\[ (p_{i}, q_{i}, r_{i}, c_{i}) = (2i-1, 2(l - i) + 1, 1, 1/2). \]
\end{lem}

\begin{proof}
Assume first that \( l = 1 \). Then \( dn^{2} = b_{1}(dna - 1) \) and
\( \gcd(dn^{2}, dna - 1) = 1 \) imply that \( dn^{2} = b_{1} \)
and \( dna = 2 \). Hence
\( (d, n, a) = (1, 2, 1) \) and \( b_{1} = 4 \).
In this case, \( (p_{1}, q_{1}, r_{1}, c_{1}) = (1, 1, 1, 1/2) \).

Thus, we may assume that \( l \geq 2 \).
Since \( p_{l} = dn(n-a) - 1 \) by Lemma~\ref{lem:ripl}, we have
\begin{equation}\label{eq:n=2:1}
\frac{n}{b_{1}} + \frac{1}{dn}\leq a < \frac{n}{b_{1}-1} + \frac{1}{dn}
\quad \text{and} \quad
\frac{n}{b_{l}} + \frac{1}{dn}\leq n - a < \frac{n}{b_{l}-1} + \frac{1}{dn}
\end{equation}
from
\( 0 \leq q_{2} = b_{1}q_{1} - q_{0} < q_{1}\) and
\( 0 \leq p_{l-1} = b_{l}p_{l} - p_{l+1} < p_{l}  \).
In particular,
\[ \frac{1}{b_{1}} + \frac{1}{b_{l}} + \frac{2}{dn^{2}} \leq 1, \]
and hence, \( b_{1} \geq 3\) or \( b_{l} \geq 3 \) holds.
Assume that \( b_{1} \geq 3 \) and \( b_{l} \geq 3 \). Then,
\[ \frac{n}{2} - \frac{1}{dn}
\leq n (1 - \frac{1}{b_{l}-1}) - \frac{1}{dn} < a <
\frac{n}{b_{1}-1} + \frac{1}{dn} \leq \frac{n}{2} + \frac{1}{dn}\]
by \eqref{eq:n=2:1}.
Since \( n \geq 2 \) and \( a \in \BZZ \), we have
\( a = n/2 \). Thus, \( (n, a) = (2, 1) \), since \( \gcd(n, a) = 1 \).
Therefore,
\( q_{0} = p_{l+1} = 4d \), \( q_{1} = p_{l} = 2d - 1 \),
and hence, \( r_{0} = r_{l+1} = 2 \).
For \( 1 \leq i \leq l \), we have \( r_{i} = 1 \),
since \( 0 < r_{i} < n=2 \) by Lemma~\ref{lem:ripl}.\eqref{lem:ripl:1}.
Thus, \( c_{i} = 1/2 \) and
\( p_{i} + q_{i} = 2d \) for all \( 1 \leq i \leq l\).
As a consequence,
\[ b_{i} = \frac{r_{i-1} + r_{i+1}}{r_{i}} =
\begin{cases}
2, & \text{ if } 1 < i < l; \\
3, & \text{ if } i = 1 \text{ or } i = l.
\end{cases}
\]
The equalities \( (p_{i}, q_{i}) = (2i-1, 2(d-i) + 1) \)
are shown by induction using
\( p_{i-1} + p_{i+1} = b_{i}p_{i}  \) and
\( q_{i-1} + q_{i+1} = b_{i}q_{i} \)
with the initial values \( p_{1} = q_{l} = 1 \) and \( p_{0} = q_{l+1} = 0 \).
\end{proof}

If \( (d, n, a) \in \ST_{\DNA} \),
then \( (d, n, n-a) \), \( (d, 2n - a, n) \), \( (d, n+a, a) \in \ST_{\DNA} \).
Thus, we have three maps \( \boldsymbol{i}\), \( \boldsymbol{t}_{L} \),
\( \boldsymbol{t}_{R} \colon \ST_{\DNA} \to \ST_{\DNA}\) defined by
\[ \boldsymbol{i}(d, n, a) = (d, n, n-a), \quad
\boldsymbol{t}_{L}(d, n, a) = (d, 2n - a, n), \quad
\boldsymbol{t}_{R}(d, n, a) = (d, n+a, a).
\]
Thus, \( \boldsymbol{i} \) is an involution, and
\( \boldsymbol{t}_{L}
= \boldsymbol{i} \circ \boldsymbol{t}_{R} \circ \boldsymbol{i} \).
Concerning these maps, we have:

\begin{lem}%
[{cf.\ \cite[(2.8.2)]{WahlElliptic}, \cite[Proposition~3.11]{KSh}, %
\cite[Lemma~3.4]{Stevens91}, \cite[Theorem~17]{Manetti91}, %
\cite[Theorem~15]{Lee99}}]\label{lem:going-up}
Let \( (d, n, a) \) be a triplet in \( \ST_{\DNA} \) with
\(  B(d, n, a) = (b_{1}, b_{2}, \ldots, b_{l}) \) and
\( R(d, n, a) = (r_{1}, r_{2}, \ldots, r_{l}) \).
Then, the following hold\emph{:}
\begin{align*}
B(d, 2n - a, n) &= (2, b_{1}, \ldots, b_{l-1}, b_{l} + 1), &
\delta(d, 2n - a, n) &= \delta(d, n, a) + 1, \\
R(d, 2n - a, n) &= (n = r_{1} + r_{l}, r_{1}, \ldots, r_{l}), &
l(d, 2n - a, n) &= l(d, n, a) + 1, \\
B(d, n+a, a) &= (b_{1} + 1, b_{2}, \ldots, b_{l}, 2), &
\delta(d, n+a, a) &= \delta(d, n, a) + 1, \\
R(d, n+a, a) &= (r_{1}, \ldots, r_{l}, n = r_{1} + r_{l}), &
l(d, n+a, a) &= l(d, n, a) + 1.
\end{align*}
\end{lem}

\begin{proof}
We set
\[ P(d, n, a) = (p_{1}, \ldots, p_{l}), \quad
Q(d, n, a) = (q_{1}, \ldots, q_{l}), \quad \text{and}
\quad C(d, n, a) = (c_{1}, \ldots, c_{l}). \]
Concerning the map
\( \boldsymbol{i} \colon (d, n, a) \mapsto (d, n, n - a)\),
we have
\begin{align*}
B(d, n, n-a) &= (b_{l}, b_{l-1}, \ldots, b_{1}), &
\delta(d, n, n-a) &= \delta(d, n, a), \\
P(d, n, n-a) &= (p_{l}, p_{l-1}, \ldots, p_{1}), &
Q(d, n, n-a) &= (q_{l}, q_{l-1}, \ldots, q_{1}), \\
R(d, n, n-a) &= (r_{l}, r_{l-1}, \ldots, r_{1}), &
C(d, n, n-a) &= (c_{l}, c_{l-1}, \ldots, c_{1}),
\end{align*}
by \( p_{l} = dn(n-a) - 1 \) (cf.\ Lemma~\ref{lem:ripl}.\eqref{lem:ripl:3}).
Hence, the equalities for \( \boldsymbol{t}_{L}(d, n, a) = (d, 2n - a, n) \)
are derived from those for \( \boldsymbol{t}_{R}(d, n, a) = (d, n + a, a) \)
by \( \boldsymbol{t}_{L}
= \boldsymbol{i} \circ \boldsymbol{t}_{R} \circ \boldsymbol{i} \).
Thus, it is enough to prove the equalities for \( (d, n + a, a) \).
We set
\[ (p'_{0}, q'_{0}) := (0, d(n+a)^{2}) \quad \text{and} \quad
(p'_{l+2}, q'_{l+2}) := (d(n+a)^{2}, 0), \]
and for \( 1 \leq i \leq l+1 \), we set
\begin{equation}\label{eq:p'iq'i}
(p'_{i}, q'_{i}) := \frac{1}{dn^{2}}(p_{i}, q_{i})
\begin{pmatrix}
dn(n+a) - 1 & 1 \\
-1 & dn(n+a) + 1
\end{pmatrix}.
\end{equation}
Then,
\( p'_{i} \) and \( q'_{i} \) are positive integers
by Lemma~\ref{lem:ripl}.\eqref{lem:ripl:2}, and we have
\[ (p'_{i-1}, q'_{i-1}) + (p'_{i+1}, q'_{i+1}) = \begin{cases}
(b_{1} + 1)(p'_{1}, q'_{1}), & \text{ for } i = 1, \\
b_{i}(p'_{i}, q'_{i}), & \text{ for } 2 \leq i \leq l, \\
2(p'_{l+1}, q'_{l+1}), & \text{ for } i = l + 1,
\end{cases}\]
for all \( 1 \leq i \leq l+1 \).
Moreover, \( q'_{1} = da(n+a) - 1 \) by \eqref{eq:p'iq'i}.
Therefore,
\( B(d, n+a, a) = (b_{1} + 1, b_{2}, \ldots, b_{l}, 2) \),
\( P(d, n+a, a) = (p'_{1}, p'_{2}, \ldots, p'_{l+1}) \), and
\( Q(d, n+a, a) = (q'_{1}, q'_{2}, \ldots, q'_{l+1}) \).
In particular,
\[ \delta(d, n+a, a) = 3 + \sum\nolimits_{i = 1}^{l} b_{i} - (2(l+1) + 1)
= \delta(d, n+a, a) + 1.\]
By \eqref{eq:p'iq'i}, we have
\[ \frac{p'_{i} + q'_{i}}{d(n+a)} =
\frac{dn(n+a)(p_{i} + q_{i})}{d(n+a) \cdot dn^{2}} =
\frac{p_{i} + q_{i}}{dn} = r_{i} \]
for all \( 1 \leq i \leq l+1 \).
Hence, \( R(d, n+a, a) = (r_{1}, \ldots, r_{l}, n = r_{1} + r_{l}) \).
Thus, we are done.
\end{proof}

\begin{table}[t]
\begin{tabular}{c|c|c|l}
\hline
\( (d, n, a) \) & \( \delta\) & \( l \) &
\( \begin{pmatrix}
b_{1} & b_{2} & \cdots & b_{l} \\
r_{1} & r_{2} & \cdots & r_{l}
\end{pmatrix}
\) \\
\hline \hline
\( (1, 2, 1) \) & \( 1 \) & \( 1 \) & \( \begin{pmatrix} 4 \\
1 \\
\end{pmatrix}\) \\
\hline
\( (k, 2, 1) \) & \( 1 \) & \( k \) & \( \begin{pmatrix}
3 & 2 & \cdots & 2 & 3 \\
1 & 1 & \cdots & 1 & 1
\end{pmatrix}\)\\
\hline
\( (1, m, 1) \) & \( m - 1\) & \( m-1 \) &
\( \begin{pmatrix}
m+2 & 2 & 2 & \cdots & 2 \\
1 & 2 & 3 & \cdots & m-1 \\
\end{pmatrix}\)\\
\hline
\( (1, 2m - 1, m) \) & \( m \) & \( m \) &
\( \begin{pmatrix}
2 & m + 2 & 2 & \cdots & 2 & 3 \\
m & 1 & 2 & \cdots & m-2 & m-1\\
\end{pmatrix}\)\\
\hline
\( (k, m, 1) \) & \( m-1 \) & \( m + k - 1 \) &
\( \begin{pmatrix}
m+1 & 2 & \cdots & 2 & 3 & 2 & 2 & \cdots & 2 \\
1 & 1 & \cdots & 1 & 1 & 2 & 3 & \cdots & m-1
\end{pmatrix}\)\\
\hline
\( (1, 3m - 1, m) \) & \( m + 1\) & \( m + 1\) &
\( \begin{pmatrix}
3 & m + 2 & 2 & \cdots & 2 & 3 & 2\\
m & 1 & 2 & \cdots & m-2 & m-1 & 2m - 1\\
\end{pmatrix}\)\\
\hline
\end{tabular}

\vspace{1ex}

\hfill (Here, \( k \geq 2 \) and \( m \geq 3 \))
\caption{Invariants related to singularities of class T, Part I}
\label{table:M}
\end{table}

\begin{table}[ht]
\begin{tabular}{c|c|c|l}
\hline
\( (d, n, a) \) & \( \delta \) & \( l \) &
\( \begin{pmatrix}
b_{1} & b_{2} & \cdots & b_{l} \\
r_{1} & r_{2} & \cdots & r_{l}
\end{pmatrix} \) \\
\hline \hline
\( (1, 11, 3) \) & \( 5\) & \( 5 \) &
\( \begin{pmatrix}
4 & 5 & 3 & 2 & 2\\
3 & 1 & 2 & 5 & 8\\
\end{pmatrix}\)\\
\hline
\( (1, 19, 5) \) & \( 7\) & \( 7 \) &
\( \begin{pmatrix}
4 & 7 & 2 & 2 & 3 & 2 & 2\\
5 & 1 & 2 & 3 & 4 & 9 & 14\\
\end{pmatrix}\)\\
\hline
\( (1, 19, 13) \) & \( 8\) & \( 8 \) &
\( \begin{pmatrix}
2 & 2 & 9 & 2 & 2 & 2 & 2 & 4\\
13 & 7 & 1 & 2 & 3 & 4 & 5 & 6\\
\end{pmatrix}\)\\
\hline
\( (3, 23, 4) \) & \( 8\) & \( 10\) &
\( \begin{pmatrix}
6 & 5 & 2 & 3 & 2 & 3 & 2 & 2 & 2 & 2 \\
4 & 1 & 1 & 1 & 2 & 3 & 7 & 11 & 15 & 19
\end{pmatrix}\)\\
\hline
\( (1, 25, 17) \) & \( 10\) & \( 10 \) &
\( \begin{pmatrix}
2 & 2 & 11 & 2 & 2 & 2 & 2 & 2 & 2 & 4 \\
17 & 9 & 1 & 2 & 3 & 4 & 5 & 6 & 7 & 8
\end{pmatrix}\)\\
\hline
\( (1, 35, 6) \) & \( 10\) & \( 10 \) &
\( \begin{pmatrix}
6 & 8 & 2 & 2 & 2 & 3 & 2 & 2 & 2 & 2 \\
6 & 1 & 2 & 3 & 4 & 5 & 11 & 17 & 23 & 29
\end{pmatrix}\)\\
\hline
\( (1, 63, 34) \) & \( 11\) & \( 11 \) &
\( \left(\begin{array}{ccccccccccc}
2 & 7 & 7 & 2 & 2 & 3 & 2 & 2 & 2 & 2 & 3\\
34 & 5 & 1 & 2 & 3 & 4 & 9 & 14 & 19 & 24 & 29
\end{array}\right)\)\\
\hline
\( (1, 252, 145) \) & \( 13\) & \(  13 \) &
\( \left(\begin{array}{ccccccccccccc}
2 & 4 & 6 & 2 & 6 & 2 & 4 & 2 & 2 & 2 & 3 & 2 & 3\\
145 & 38 & 7 & 4 & 1 & 2 & 3 & 10 & 17 & 24 & 31 & 69 & 107
\end{array} \right)\)\\
\hline
\end{tabular}

\vspace{1ex}

\caption{Invariants related to singularities of class T, Part II}
\label{table:M2}
\end{table}

\begin{corsub}[{cf.\ \cite[Proposition~3.11]{KSh}, \cite[Theorem~17]{Manetti91}, %
\cite[Proposition~20]{Lee99}}] \label{corsub:going-up}
Any element \( (d, n, a) \) of \( \ST_{\DNA} \) is obtained from
\( (d, 2, 1) \) by a successive compositions of
maps \( \boldsymbol{t}_{L} \) and \( \boldsymbol{t}_{R} \).
The number of the compositions equals
\( \delta(d, n, a) - 1 = l(d, n, a)  - d \).
In particular, \( \sum\nolimits_{i = 1}^{l} b_{i} = 3l + 2 - d\).
\end{corsub}

\begin{proof}
If \( b_{1} \geq 3\) and \( b_{l} \geq 3 \), then
\( (d, n, a) = (d, 2, 1) \), \( l(d, 2, 1) = d \), and
\( \delta(d, 2, 1) = 1\) by Lemma~\ref{lem:n=2}.
If \( b_{1} = 2 \), then \( b_{l} \geq 3 \) by Lemma~\ref{lem:n=2}, and
\( (d, n, a) = \boldsymbol{t}_{L}(d, a, n - 2a) \) with
\( l(d, n, a) = l(d, a, n - 2a) + 1 \) and
\( \delta(d, n, a) = \delta(d, a, n - 2a) + 1\)
by Lemma~\ref{lem:going-up}.
Similarly,
if \( b_{l} = 2 \), then \( b_{1} \geq 3 \), and
\( (d, n, a) = \boldsymbol{t}_{R}(d, n - a, a) \) with
\( l(d, n, a) = l(d, n - a, a) + 1 \) and
\( \delta(d, n, a) = \delta(d, n-a, a) + 1 \).
Hence, we are done by induction on \( l(d, n, a) \).
\end{proof}

In Tables~\ref{table:M} and \ref{table:M2},
we list \( \delta = \delta(d, n, a) \), \( l = l(d, n, a) \),
\( B(d, n, a) \), and \( R(d, n, a) \) for typical elements
\( (d, n, a) \in \ST_{\DNA} \), some of which are used later.
For the numbers \( c_{i} \), we have:

\begin{corsub}[{cf.\ \cite[Corollary~17]{Lee99}}]\label{corsub:going-up2}
Let \( (d, n, a) \) be a triplet in \( \ST_{\DNA} \)
with \( B(d, n, a) = (b_{1}, \ldots, b_{l}) \),
\( R(d, n, a) = (r_{1}, \ldots, r_{l}) \),
and \( C(d, n, a) = (c_{1}, \ldots, c_{l}) \).
If \( c_{i} \leq 1/2 \) \emph{(}or equivalently, \( r_{i}/n \geq 1/2\)\emph{)}
for some \( i \) and if \( n > 2 \), then
either \( b_{j} = 2 \) for all \( j \leq i \) or
\( b_{k} = 2 \) for all \( k \leq i \).
\end{corsub}

\begin{proof}
Assume the contrary.
Then, \( b_{j} \geq 3 \) and \( b_{k} \geq 3 \) for some \( j \leq i \leq k \).
By Lemma~\ref{lem:n=2}, we have \( b_{1} = 2 \) or \( b_{l} = 2 \).
Thus, we may assume that \( b_{1} \geq 3 \) and \( b_{l} = 2 \).
In particular, we may put \( j = 1 \).
Then, \( (d, n, a) = \boldsymbol{t}_{R}(d, n-a, a) \),
and
\[ B(d, n - a, a) = (b_{1} - 1, b_{2}, \ldots, b_{l-1}) \quad
\text{and} \quad
R(d, n - a, a) = (r_{1}, r_{2}, \ldots, r_{l-1})\]
by Lemma~\ref{lem:going-up}.
In particular,
\[ c'_{i} = 1 - r'_{i}/(n-a) < 1 - r_{i}/n = c_{i} \leq 1/2, \]
where \( C(d, n - a, a) = (c'_{1}, \ldots, c'_{l-1}) \).
Hence, in order to derive a contradiction,
we may assume that \( k = l-1 \).

If \( b_{1} - 1 \geq 3 \), then \( n - a = 2 \) and \( r_{i} = 1 \)
by Lemma~\ref{lem:n=2};  this is a contradiction, since
\( r_{i}/n \leq 1/3  \).
Thus, \( b_{1} = 3 \) and
\( (d, n - a, a) = \boldsymbol{t}_{L}(d, a, 3a - n)  \)
by Lemma~\ref{lem:going-up}. In particular, \( 2a < n < 3a\).
Moreover,
\[ B(d, a, 3a - n) = (b_{2}, \ldots, b_{l-2}, b_{l-1} - 1) \quad \text{and} \quad
R(d, a, 3a - n) = ( r_{2}, \ldots, r_{l-1}).\]
Thus, \( r_{i} < a \) by Lemma~\ref{lem:ripl} applied to \( (d, a, 3a - n) \),
and hence, \( r_{i}/n < a/n < 1/2 \); this is a contradiction.
\end{proof}

Finally in Section~\ref{sect:DefCT}, we shall prove:

\begin{thm}\label{thm:localsmoothing}
Let \( \Lambda \) be a complete discrete valuation ring or a field.
Let \( V \) be a toric \( \Lambda \)-scheme of type \( (dn^{2}, dna - 1) \)
\emph{(}cf.\ Definition~\emph{\ref{dfn:toricnq})}
for some positive integers \( d \), \( n \), \( a \) with
\( n > a \) and \( \gcd(n, a) = 1 \).
Then, there exist a flat family \( \SV \to T \) of normal affine surfaces
over an open subset \( T \) of the affine line \( \BAA^{1}_{\Lambda} \)
and a section \( \sigma \colon \Spec \Lambda \to T \)
satisfying the following conditions\emph{:}
\begin{enumerate}
\item \label{thm:localsmoothing:1}
\( \SV \times_{T, \sigma} \Spec \Lambda \isom V\).

\item \label{thm:localsmoothing:2}
\( \SV \to T \) is smooth over \( T \setminus \sigma(\Spec \Lambda) \).

\item \label{thm:localsmoothing:3}
\( \SV \) is normal, and \( rK_{\SV}\) is Cartier
with \(\SO_{\SV}(rK_{\SV})|_{V} \isom \SO_{V}(rK_{V}) \)
for any integer \( r \) divisible by \( n \).
\end{enumerate}
\end{thm}

\begin{remsub}
Our idea of the proof of Theorem~\ref{thm:localsmoothing}
is taken from the proof of \cite[Proposition~4.19]{logdelPezzo},
which treats a special case.
However, the proof of Proposition~4.19 contains an error
when \( p \mid n \).
The error is corrected by the present proof.
\end{remsub}

\begin{remsub}
Roughly speaking, when \( \Lambda  \) is a field,
Theorem~\ref{thm:localsmoothing}
asserts that a toric singularity
of class T has a ``\( \BQQ \)-Gorenstein smoothing''
(cf.\ \cite{KSh}).
The proof of Theorem~\ref{thm:localsmoothing} is essentially the same
as in the proof of \cite[Proposition~3.10]{KSh} when \( \Lambda \)
is a field of characteristic zero.
\end{remsub}

We recall the construction of \( V \) in Section~\ref{sect:dualgraph}:
We have a free abelian group \( \NN_{0} = \BZZ e_{1} + \BZZ e_{2} \)
with the base \( (e_{1}, e_{2}) \) and the cone \( \Bsigma = \Cone(e_{1}, e_{2})
\subset \NN_{0} \otimes \BRR \)
such that \( V = \BTT_{\NN}(\Bsigma) \) for a free abelian group
\( \NN = \NN_{0} + \BZZ v \) defined by the vector
\[ v = \frac{1}{dn^{2}}(e_{1} + (dna - 1)e_{2}). \]
Let \( \MM \) and \( \MM_{0} \) be the dual abelian groups of \( \NN \)
and \( \NN_{0} \), respectively.
Then, \( R_{0} = \Lambda[\Bsigma^{\vee} \cap \MM_{0}] \) is
a polynomial \( \Lambda \)-algebra generated by two-variables
\( \xtt_{1} \), \( \xtt_{2} \) which correspond to
the dual basis of \( (e_{1}, e_{2}) \).
In particular, \( V_{0} = \BTT_{\NN_{0}}(\Bsigma) \isom  \BAA^{2}_{\Lambda} \).
The affine coordinate ring
\( R := \OH^{0}(V, \SO_{V}) = \Lambda[\Bsigma^{\vee} \cap \MM] \) is
a \( \Lambda \)-subalgebra of \( \Lambda[\xtt_{1}, \xtt_{2}] \) generated by
the monomials \( \xtt_{1}^{k_{1}}\xtt_{2}^{k_{2}} \)
satisfying \( k_{1} + (dna - 1)k_{2} \equiv 0 \bmod dn^{2} \).

In this situation, we define a subgroup \( \NN_{1} \) of \( \NN \) by
\[ \NN_{1} :=  \NN_{0} + \BZZ nv =
\NN_{0} + \BZZ \frac{1}{dn}(e_{1} + (dn - 1)e_{2}).\]
Then, we have a toric \( \Lambda \)-scheme
\( V_{1} = \BTT_{\NN_{1}}(\Bsigma) \),
and finite surjective morphisms \( \tau_{0} \colon V_{0} \to V_{1} \)
and \( \tau \colon V_{1} \to V \)
associated with the inclusions \( \NN_{0} \subset \NN_{1} \)
and \( \NN_{1} \subset \NN \).
The affine coordinate ring \( R_{1} := \OH^{0}(V_{1}, \SO_{V_{1}})
= \Lambda[\Bsigma^{\vee} \cap \MM_{1}] \),
where \( \MM_{1} \) is the dual abelian group of \( \NN_{1} \),
is a \( \Lambda \)-subalgebra of \( \Lambda[\xtt_{1}, \xtt_{2}] \) generated by
three monomials \( \xtt_{1}^{dn} \), \( \xtt_{2}^{dn} \), \( \xtt_{1}\xtt_{2} \).
Thus,
\[ R_{1} \isom \Lambda[\utt_{1}, \utt_{2}, \ztt]/(\utt_{1}\utt_{2} - \ztt^{dn}) \]
for three variables \( \utt_{1} \), \( \utt_{2} \), \( \ztt \), where
the isomorphism above is defined by
\( \utt_{1} = \xtt_{1}^{dn} \), \( \utt_{2} = \xtt_{2}^{dn} \),
and \( \ztt = \xtt_{1}\xtt_{2} \).
Since \( n \) is the index of \( \NN/\NN_{1} \), the group subscheme
\( \Bmu_{n} = \Ker (\BTT_{\NN_{1}} \to \BTT_{\NN}) \) of
\( \BTT_{\NN_{1}} \) acts on \( V_{1} \) and its quotient scheme is just \( V \).
Here, the action of \( \Bmu_{n} \) on \( V_{1} = \Spec R_{1} \) is given by
\begin{equation}\label{eq:mu_n action}
(\utt_{1}, \utt_{2}, \ztt) \mapsto
(\utt_{1} \otimes \ttt, \utt_{2} \otimes \ttt^{-1}, \ztt \otimes \ttt^{a}),
\end{equation}
where \( \Bmu_{n} \) is regarded as
\(\Spec \Lambda[\ttt, \ttt^{-1}]/(\ttt^{n} - 1) \).
The action of \( \Bmu_{n} \) on \( V_{1} \) is induced from
that on \( \Spec \Lambda[\utt_{1}, \utt_{2}, \ztt]  \) given by
\eqref{eq:mu_n action}.
Then, the \( \Bmu_{n} \)-invariant part \( R\sptilde  \) of
\( \Lambda[\utt_{1}, \utt_{2}, \ztt] \) is generated by
monomials \( \utt_{1}^{m_{1}}\utt_{2}^{m_{2}}\ztt^{m_{3}} \)
such that \( m_{i} \geq 0 \) for all \( 1 \leq i \leq 3 \) and
\( m_{1} - m_{2} + am_{3} \equiv 0 \bmod n \).
We see that \( \Spec R\sptilde \) is isomorphic to
\( \BTT_{\NN\sptilde}(\Bsigma\sptilde) \)
for the affine toric \( \Lambda \)-scheme
\( \BTT_{\NN\sptilde}(\Bsigma\sptilde) \) of relative dimension three
defined as follows:
Let \( \NN_{1}\sptilde \) be a free abelian group
\( \bigoplus\nolimits_{i = 1}^{3}\BZZ e\sptilde_{i} \)
of rank three
and let \( \Bsigma\sptilde \)
be the cone \( \sum\nolimits_{i = 1}^{3}\BRR_{\geq 0}e_{i}\sptilde \).
The abelian group \( \NN\sptilde \) is a subgroup of
\( \NN_{1}\sptilde \otimes \BQQ \) defined by
\[ \NN\sptilde := \NN_{1}\sptilde +
\BZZ\, \frac{1}{n}(e\sptilde_{1} - e\sptilde_{2} + a e\sptilde_{3}). \]
Then, \( V = \Spec R \) is a hypersurface (or a Cartier divisor) of
\( V\sptilde := \Spec R\sptilde \) defined by the principal ideal
\( (\utt_{1}\utt_{2} - \ztt^{dn}) \).

\begin{lem}\label{lem:GorIndex}
The smallest integer \( r \) such that \( rK_{V} \) \emph{(}resp.\
\( rK_{V\sptilde} \)\emph{)} is Cartier, is \( n \). Thus,
\( n \) is equal to the Gorenstein index of \( V \)
\emph{(}resp.\ \( V\sptilde \)\emph{)}.
\end{lem}

\begin{proof}
Let \( (k_{1}, k_{2}, k_{3}) \) be an integral vector
such that \( k_{1} - k_{2} + ak_{3} \equiv 0 \bmod n \).
Then, the principal divisor associated to the monomial
\( \utt_{1}^{k_{1}}\utt_{2}^{k_{2}}\ztt^{k_{3}} \) on \( V \) is
expressed as
\[ \Div(\utt_{1}^{k_{1}}\utt_{2}^{k_{2}}\ztt^{k_{3}})_{V} =
(k_{1}dn + k_{3})D_{1} + (k_{2}dn + k_{3})D_{2}\]
by \eqref{eq:muaD000} in Remark~\ref{remsub:lem:VDDC}.
Note that \( D := D_{1} + D_{2} \sim -K_{V} \).
Hence, \( n \) is the Gorenstein index of \( V \).
In fact,
\( nD = \Div(\ztt^{n}) \) is Cartier, and
if \( jD \) is an Cartier divisor for an integer \( j \),
then \( j = k_{1}dn + k_{3} = k_{2}dn + k_{3}\)
for some integral vector \( (k_{1}, k_{2}, k_{3}) \) above; hence
\( k_{1} = k_{2} \) and \( j \equiv k_{3} \equiv 0 \mod n\).

For \( 1 \leq i \leq 3 \), let  \( D\sptilde_{i} \) be
the \( \BTT_{\NN\sptilde} \)-invariant divisor on \( V\sptilde \)
corresponding to the ray \( \BRR_{\geq 0}e\sptilde_{i} \).
Since \( \gcd(n, a) = 1 \), we see that
\[ \BRR_{\geq 0}e\sptilde_{i} \cap \NN\sptilde = \BZZ_{\geq 0} e\sptilde_{i} \]
for \( 1 \leq i \leq 3 \).
Then, we have the following equality similar to
\eqref{eq:muaD000}:
\[ \Div(\utt_{1}^{k_{1}}\utt_{2}^{k_{2}}\ztt^{k_{3}})_{V\sptilde} =
k_{1}D\sptilde_{1} + k_{2}D\sptilde_{2}  + k_{3}D\sptilde_{3}.\]
Here, we know also
\( D\sptilde := \sum\nolimits_{i = 1}^{3} D\sptilde_{i} \sim -K_{V\sptilde} \)
and that \( nD\sptilde_{i} \) is Cartier for all \( i \).
Suppose that \( jD\sptilde \) is Cartier for an integer \( j \). Then,
\( j = k_{1} = k_{2} = k_{3} \) for an
integral vector \( (k_{1}, k_{2}, k_{3}) \) such that
\( k_{1} - k_{2} + ak_{3} \equiv 0 \bmod n\); hence \( j \equiv 0 \bmod n  \).
Therefore, the Gorenstein index of \( V\sptilde \) is also \( n \).
\end{proof}

\begin{remsub}
\( V \) and \( V\sptilde \) are \( \BQQ \)-factorial, since
the cones \( \Bsigma \) and \( \Bsigma\sptilde \) are simplicial.
\end{remsub}

\begin{remsub}
It is well-known in characteristic zero that
the singularity on \( V\sptilde \) is a cyclic quotient
terminal singularity of type \( \frac{1}{n}(1, -1, a) \)
(cf.\ \cite[\S 3]{ReidCan}, \cite[(4.13)]{ReidTerm},
\cite[(5.1), (5.2)]{ReidYoung}).
Even in the positive characteristic case, the singularity is ``terminal''
in the sense that there is a resolution
\( \mu\sptilde \colon M\sptilde \to V\sptilde\) of singularity
such that
\[ K_{M\sptilde} = {\mu\sptilde}^{*}(K_{V\sptilde}) + \sum a_{i}E_{i}\sptilde \]
for \( \mu\sptilde \)-exceptional prime divisors \( E_{i}\sptilde \) and
positive rational numbers \( a_{i} \).
In fact, a toric resolution of \( V\sptilde \) is taken independently of
the characteristic, and the discrepancy \( a_{i} \) is also independent of
the characteristic.
\end{remsub}

\begin{remsub}\label{remsub:lem:GorIndex}
If \( X \) is a \( \BQQ \)-Gorenstein normal algebraic \( \Lambda \)-scheme
and if \( Y \) is a normal Cartier divisor of \( X \), then
\( r(K_{X} + Y)|_{Y} \sim rK_{Y} \) for the Gorenstein index \( r \) of \( X \).
Indeed, the left hand side is Cartier and is linearly equivalent to
the right hand side on the non-singular locus of \( Y \).
In particular, \( Y \) is also \( \BQQ \)-Gorenstein and \( r \) is divisible
by the \( \BQQ \)-Gorenstein index of \( Y \).
Applying this to Lemma~\ref{lem:GorIndex},
we see that the restriction \( nK_{V\sptilde}|_{V} \)
is linearly equivalent to \( nK_{V} \).
\end{remsub}

\begin{proof}[Proof of Theorem~\emph{\ref{thm:localsmoothing}}]
We consider an algebra
\[ R_{1}^{\sharp} :=
\Lambda[\utt_{1}, \utt_{2}, \ztt, \stt]/
(\ztt^{dn} - \utt_{1}\utt_{2} - \stt(\ztt^{n} + 1)) \]
over the polynomial ring \( \Lambda[\stt] \) of one variable.
This is flat over \( \Lambda[\stt] \), since it is Cohen--Macaulay and
every fiber of \( \Spec R^{\sharp}_{1} \to \Spec \Lambda[\stt] \)
is equidimensional.
We consider the \( \Bmu_{n} \)-action on \( \Spec R_{1}^{\sharp} \)
given by
\[ (\utt_{1}, \utt_{2}, \ztt, \stt) \mapsto
(\ttt \utt_{1}, \ttt^{-1}\utt_{2}, \ttt^{a}\ztt, \stt), \]
where \( \Bmu_{n} = \Spec \Lambda[\ttt, \ttt^{-1}]/(\ttt^{n} - 1)\).
Let \( R^{\sharp} \) be the \( \Bmu_{n} \)-invariant subring,
which is the \( \Lambda[\stt] \)-submodule
generated by monomials
\( \utt_{1}^{k_{1}}\utt_{2}^{k_{2}}\ztt^{k_{3}}\) with
\( k_{1} - k_{2} + ak_{3} \equiv 0 \mod n \).
We set \( V^{\sharp} := \Spec R^{\sharp} \),
\( S := \Spec \Lambda[\stt] = \BAA^{1}_{\Lambda} \), and
let \( \sigma \colon \Spec \Lambda \to S \) be the section
defined by \( \Lambda[\stt] \to \Lambda[\stt]/(\stt) \isom \Lambda \).
Then, \( V^{\sharp} \times_{S, \sigma} \Spec \Lambda \isom V \).
Since \( R^{\sharp} \) is
the \( \Bmu_{n} \)-invariant ring of \( R_{1}^{\sharp} \),
which is a direct summand of \( R_{1}^{\sharp} \), we see that
\( V^{\sharp} \) is normal and \( V^{\sharp} \to S \) is flat.
On the other hand,
the affine scheme \( V^{\sharp} \) is isomorphic to
the hypersurface of \( V\sptilde \times_{\Spec \Lambda} S \)
defined by \( \ztt^{dn} - \utt_{1}\utt_{2} - s(\ztt^{n} + 1) = 0\).
By Lemma~\ref{lem:GorIndex}, \( V\sptilde \times_{\Spec \Lambda} S \) is
\( \BQQ \)-Gorenstein with index \( n \).
Hence, \( V^{\sharp} \) is also \( \BQQ \)-Gorenstein with index \( r \)
dividing \( n \), since \( V^{\sharp} \) is normal
(cf.\ Remark~\ref{remsub:lem:GorIndex}).
Here, \( r = n \) by Lemma~\ref{lem:GorIndex}, since
\( V \) is also a hypersurface of \( V^{\sharp} \).
Consequently, we have
\( \SO_{V^{\sharp}}(nK_{V^{\sharp}})|_{V} \sim \SO_{V}(nK_{V}) \)
(cf.\ Remark~\ref{remsub:lem:GorIndex}).

We shall study the singularity of fibers of
\( V^{\sharp} \to S \) outside the section \( \sigma(\Spec \Lambda) \)
defined by \( \stt = 0 \).
Let \( R^{\sharp}\langle 1 \rangle \) be the affine coordinate ring of
the affine open subset
\( \{ \utt_{1} \ne 0, \stt \ne 0\} \) of \( \Spec R^{\sharp} \).
Then, \( R^{\sharp}\langle 1 \rangle \) is isomorphic to
\[
\Lambda[\ytt_{1}^{\pm 1}, \ytt_{2}, \ytt_{3}, \stt^{\pm 1}]/
(\ytt_{3}^{dn}\ytt_{1}^{ad} - \ytt_{2} - \stt(\ytt_{3}^{n}\ytt_{1}^{a} + 1)) \isom
\Lambda[\ytt_{1}^{\pm 1}, \ytt_{3}, \stt^{\pm 1}]\]
by the correspondence
\( (\ytt_{1}, \ytt_{2}, \ytt_{3}) =
(\utt_{1}^{n}, \utt_{1}\utt_{2}, \ztt\utt_{1}^{-a}) \).
Thus, \( R^{\sharp}\langle 1 \rangle \) is smooth over \( \Lambda[\stt^{\pm 1}] \).
Similarly, the affine coordinate ring
\( R^{\sharp}\langle 2 \rangle \) of \( \{\utt_{2} \ne 0, \stt \ne 0\} \)
is smooth over \( \Lambda[\stt^{\pm 1}] \), since it is isomorphic to
\[ \Lambda[\ytt_{1}, \ytt_{2}^{\pm 1}, \ytt_{3}, \stt^{\pm 1}]/
(\ytt_{3}^{dn}\ytt_{2}^{-ad} - \ytt_{1} - \stt(\ytt_{3}^{n}\ytt_{2}^{-a} + 1))
\isom \Lambda[\ytt_{2}^{\pm 1}, \ytt_{3}, \stt^{\pm 1}]\]
by the correspondence
\( (\ytt_{1}, \ytt_{2}, \ytt_{3}) =
(\utt_{1}\utt_{2}, \utt_{2}^{n}, \ztt\utt_{2}^{a}) \).

The affine coordinate ring \( R^{\sharp}\langle 3 \rangle \)
of \( \{\ztt \ne 0, \stt \ne 0\} \)
is isomorphic to
\[ \Lambda[\ytt_{1}, \ytt_{2}, \ytt_{3}^{\pm 1}, \stt^{\pm 1}]
/(\ytt_{3}^{d} - \ytt_{1}\ytt_{2} - \stt(\ytt_{3} + 1)) \]
by the correspondence
\( (\ytt_{1}, \ytt_{2}, \ytt_{3}) = (\utt_{1}\ztt^{-a'}, \utt_{2}\ztt^{a'}, \ztt^{n}) \),
where \( 0 < a' < n\) with \( aa' \equiv 1 \mod n \).
Let \( p \) be the characteristic of the residue field of \( \Lambda \).
If \( p \mid d \) or \( p \mid d - 1 \), then
\( R^{\sharp}\langle 3 \rangle \) is smooth over \( \Lambda[\stt^{\pm 1}] \)
by the Jacobian criterion.
If \( p \nmid d \) and \( p \nmid d - 1 \), then \( R^{\sharp}\langle 3 \rangle \)
is smooth over \( \Lambda[\stt^{\pm 1}, (\stt - c)^{-1}] \) for
\( c = d^{d}/(d-1)^{d-1} \) by the Jacobian criterion. Thus, the open subset
\( T := \Spec \Lambda[\stt, (\stt - c)^{-1}] \) of \( S = \BAA^{1}_{\Lambda} \),
\( \sigma \colon \Spec \Lambda \to T \subset S\),
and
\( \SV := V^{\sharp} \times_{S} T \to T\) satisfy
all the conditions of Theorem~\ref{thm:localsmoothing}.
\end{proof}


\section{Review of deformation theory}
\label{sect:Deformation}

We review the deformation theory of \(\Bbbk\)-schemes for
the fixed algebraically closed field \( \Bbbk \).

\begin{dfn}\label{dfn:defo}
Let \( X \) be a \( \Bbbk \)-scheme.
Let \( T \) be a scheme and let \( o \) be a \( \Bbbk \)-rational point
of \( T \), which is just a morphism \( o \colon \Spec \Bbbk \to T \).
A {\it deformation} of \( X \) over \( T \) with the reference point \( o \)
is a pair \( (Y/T, \iota) \) of
a flat morphism \( Y \to T \) of schemes
and an isomorphism \( \iota \colon Y \times_{T, o} \Spec \Bbbk \isom X\)
over \( \Spec \Bbbk \).
\end{dfn}

In this section,
we fix a Noetherian local ring \( \Lambda \)
which is either \( \Bbbk \)
or a complete discrete valuation ring
with residue field \( \Bbbk \).
For example,
the ring of Witt vectors of \( \Bbbk \) is a candidate of \( \Lambda \).

We recall some important notions mainly from
Schlessinger's article \cite{Sc68}.
Let \( \SC_{\Lambda} \) be the category of Artinian local \( \Lambda \)-algebras
with residue field \( \Bbbk \) and
let \( \widehat{\SC}_{\Lambda} \) be the category of complete
Noetherian local \( \Lambda \)-algebras \( \GR = (\GR, \GM_{\GR}) \)
such that \( \GR/\GM_{\GR}^{n} \in \SC_{\Lambda} \) for all \( n \).
An object \( \GR \) of \( \widehat{\SC}_{\Lambda} \) defines
a functor \( h_{\GR} \colon \SC_{\Lambda} \to (\Sets) \)
to the category of sets by \( h_{\GR}(A) = \Hom_{\widehat{\SC}_{\Lambda}}(\GR, A) \).
For two functors \( F \), \( G \colon \SC_{\Lambda} \to (\Sets) \),
a morphism \( \phi \colon F \to G \) means a natural transformation of functors.
The morphism \( \phi \colon F \to G \) is called \emph{smooth}
(in the sense of Schlessinger \cite[(2.2)]{Sc68}) if the natural map
\( F(B) \to F(A) \times_{G(A)} G(B) \) is surjective for
any surjection \( B \to A \) in \( \SC_{\Lambda} \).

The deformation functor \( \Def_{X}  \colon \SC_{\Lambda} \to
(\Sets) \)  of a \( \Bbbk \)-scheme \( X \) is defined as follows.
An infinitesimal deformation
of \( X \) to an algebra \( A \) of \( \SC_{\Lambda} \)
is a deformation \( (X_{A}, \iota) \) of \( X \) over \( \Spec A \)
with \( \GM_{A} \) as a \( \Bbbk \)-rational reference point .
Here, \( X_{A} \) is a flat \( A \)-scheme and
\( \iota \) is an isomorphism \( X_{A} \times_{\Spec A} \Spec \Bbbk \isom X \).
Another deformation \( (X'_{A},\iota') \) of \( X \)
is isomorphic to \( (X_{A}, \iota) \) if there is a morphism
\( \phi \colon X_{A} \to X'_{A} \) over \( \Spec A \) such that
\( \iota' =  \phi \circ \iota \).
Note that this morphism \( \phi \) is indeed an isomorphism over \( \Spec A \),
since \( X_{A} \) and \( X'_{A} \) are both homeomorphic to \( X \)
and since \( \phi \) induces the identity on \( X \).
We define \( \Def_{X}(A) \) to be the set of isomorphism classes
of deformations of \( X \) to \( A \).
Then, \( \Def_{X} \) gives rise to a functor
\(\SC_{\Lambda} \to (\Sets)\), which
is called the {\it deformation functor} of \( X \).
When \( X \) is an affine scheme \( \Spec R \), we write \( \Def_{R} \)
for \( \Def_{X} \).

\begin{dfn}[Pro-couple, formal deformation, and hull]\label{dfn:pro-couple}
Let \( \GR \) be an object of \( \widehat{\SC}_{\Lambda} \) and
set \( \GR_{n} :=  \GR/\GM_{\GR}^{n+1}\) for \( n \geq 0 \).
Let \( \xi \) be an element of
\[ \widehat{\Def_{X}}(\GR) :=
\varprojlim\nolimits_{n} \Def_{X}(\GR_{n}).\]
Note that to give an element of \( \widehat{\Def_{X}}(\GR) \) is equivalent to
giving a morphism \( h_{\GR} \to \Def_{X} \) of functors.
In \cite{Sc68}, the pair \( (\GR, \xi) \) is called a \emph{pro-couple}.
For each \( n \), the element
\( \xi \) defines a flat \( \GR_{n} \)-scheme \( X_{n} \)
with an isomorphism
\( \iota_{n} \colon X_{n} \times_{\Spec \GR_{n}} \Spec \Bbbk \isom X\).
Moreover, \( (X_{n}, \iota_{n}) \) form an inductive system.
Thus, we have a formal scheme \( \GX = \GX_{\xi} := \varinjlim X_{n} \)
flat over the affine formal scheme \( \Spf \GR \) with an isomorphism
\( \iota \colon \GX \times_{\Spf \GR} \Spec \Bbbk \isom X \)
(cf.\ \cite[I, Proposition~(10.6.3)]{EGA}).
The pair \( (\GX/\Spf \GR, \iota) \)
or the system \( (X_{n}, \iota_{n})_{n \geq 0} \) is
called a \emph{formal deformation} of \( X \)
over \( \GR \) (or over \( \Spf \GR \)).
\end{dfn}

\begin{remsub}
In the definition above, if \( X \) is an algebraic \( \Bbbk \)-scheme, i.e.,
a \( \Bbbk \)-scheme of finite type, then
\(\GX\) is Noetherian and
\( \GX \to \Spf \GR \) is a morphism of finite type by
\cite[I, Corollaire~(10.6.4), Proposition~(10.13.1), D\'efinition~(10.13.3)]{EGA}.
If \( X \) is proper over \( \Bbbk \), then \(\GX \to \Spf \GR\)
is proper (cf.\ \cite[III, (3.4.1)]{EGA}).
\end{remsub}

\begin{dfn}\label{dfn:effective,algebraizable,hull}
Let \( (\GR, \xi) \) be a pro-couple of \( \Def_{X} \) and
let \( \GX \to \Spf \GR \) be the formal deformation
of \( X \) associated to \( (\GR, \xi) \).

\begin{enumerate}
\item  The pro-couple (or the formal deformation)
is said to be \emph{effective} if
there is a flat \( \GR \)-scheme \(\SX \) such that
\( \SX \) is a deformation of \( X \) over \( \Spec \GR \) and that
\( (\GR, \xi) \) is induced from
the inductive system \( \SX_{n} = \SX \times_{\Spec \GR} \Spec \GR_{n} \);
equivalently the \( \GM_{\GR} \)-adic completion of \( \SX \) is
isomorphic to \( \GX \)
as a formal scheme over \( \Spf \GR \).

\item  The pro-couple (or the formal deformation)
is said to be \emph{algebraizable} if
there is a deformation \( Y \to T \) of \( X \) with
a \( \Bbbk \)-rational reference point \( o \in T\) such that
\begin{itemize}
\item \( T \) is a scheme of finite type over \( \Lambda \),

\item  the completion of the local ring \( \SO_{T, o} \) is
isomorphic to \( \GR \), and

\item  the formal completion of \( Y \) along the fiber \( X \) is isomorphic to
\( \GX \) over \( \Spf \GR \).
\end{itemize}

\item The pro-couple \( (\GR, \xi) \)
is called a \emph{pro-representable hull}
(or a \emph{hull}, for short) of \( \Def_{X} \) if the morphism
\( h_{\GR} \to \Def_{X} \) corresponding to \( \xi \) is smooth and
induces the bijection
\[ \Bt(h_{\GR}) := h_{\GR}\left(\Bbbk[\ep]/(\ep^{2})\right)
\to \Bt(\Def_{X}) := \Def_{X}\left(\Bbbk[\ep]/(\ep^{2})\right) \]
between the tangent spaces (cf.\ \cite[Definition~2.7]{Sc68}).
\end{enumerate}
\end{dfn}

\begin{remsub}\label{remsub:dfn:effective,algebraizable,hull:1}
The hull is unique up to non-canonical isomorphism
(cf.\ \cite[Proposition~2.9]{Sc68}).
The existence of hull of \( \Def_{X} \) is known in the cases
when \( X \) is proper over \( \Bbbk \)
and when \( X \) is affine with only isolated singularities
(cf.\ \cite[Proposition~3.10]{Sc68}).
\end{remsub}

\begin{remsub}\label{remsub:dfn:effective,algebraizable,hull:200}
Let \( X \) be a projective \( \Bbbk \)-scheme and let
\( \GX \to \Spf \GR \) be the formal deformation associated
with a pro-representable hull \( (\GR, \xi) \) of \( \Def_{X} \).
If \( \GX \) admits a relatively ample invertible sheaf over \( \Spf \GR \),
then this is effective by a projective morphism \(\SX \to \Spec \GR \),
by an application \cite[III, Th\'eor\`eme~(5.4.5)]{EGA} of
Grothendieck's existence theorem, i.e.,
\( \GX \) is the \( \GM_{\GR} \)-adic completion of \( \SX \).
Moreover, this is algebraizable by Artin's result \cite[Theorem~1.6]{ArtinFormal}.
\end{remsub}

\begin{remsub}\label{remsub:dfn:effective,algebraizable,hull:2}
Let \( X \) be an equidimensional affine algebraic \( \Bbbk \)-scheme with
only isolated singularities. Then, the pro-representable hull of \( \Def_{X} \)
is effective by \cite[Chapitre~IV, Th\'eor\`eme~7]{Elkik}
and is algebraizable by \cite[Theorem~1.6]{ArtinFormal}.
\end{remsub}

\begin{lem}\label{lem:functorsisom}
If \( X \) is an affine algebraic \( \Bbbk \)-scheme
with a unique singular point \( x \) and
if \( (X', x') \) is an \'etale neighborhood of \( (X, x) \),
then there is a smooth morphism
\( \Def_{X} \to \Def_{\SO_{X', x'}} \) of functors on \( \SC_{\Lambda} \)
inducing an isomorphism between the tangent spaces.
\end{lem}

\begin{proof}
We have a morphism \( \Def_{X} \to \Def_{\SO_{X', x'}} \)
by \cite[IV, Th\'eor\`eme (18.1.2)]{EGA}.
The smoothness and the isomorphism between the tangent spaces
are derived from \cite[Theorem 4.10.(b)]{Rim}.
\end{proof}

\begin{remsub}\label{remsub:lem:functorsisom00:nonsingular}
If \( (X, x) \) is non-singular, then
\( \Def_{\SO_{X, x}}(A) \) consists of one element for any \( A \in \SC_{\Lambda} \).
In other words, the canonical morphism \( \Def_{R} \to h_{\Lambda} \) of functors is
an isomorphism.
This follows for example from \cite[IV, Proposition~(18.1.1)]{EGA} and
the formal smoothness (cf.\ \cite[IV, D\'efinition~(17.1.1)]{EGA})
of \( \Spec \SO_{X, x} \to \Spec \Bbbk \).
\end{remsub}

\begin{lemsub}\label{lem:functorsisom01}
Let \( (X, x)  \) and \( (X', x') \) be as in Lemma~\emph{\ref{lem:functorsisom}}.
Let \( Y \to T\) be a deformation of \( \Spec \SO_{X, x} \) over
a scheme \( T \) which is
with a \( \Bbbk \)-rational reference point \( o \in T \) such that \( Y \) is
an affine scheme of a local ring.
Then, there exist a deformation \( Y' \to T\) of \( \Spec \SO_{X', x'} \) over \( T \)
with reference point \( o \) and a formally \'etale morphism
\( Y' \to Y \) over \( T \) inducing \( \Spec \SO_{X', x'} \to \Spec \SO_{X, x} \)
as a morphism between the fibers over \( o \).
\end{lemsub}

\begin{proof}
This follows from \cite[IV, Proposition~(18.1.1)]{EGA},
since \( \SO_{X, x} \) and \( \SO_{X', x'}\) are
essentially of finite type over \( \Bbbk \).
\end{proof}

\begin{dfn}\label{dfn:Defloc}
Let \( X \) be an algebraic \( \Bbbk \)-scheme
and \( P \) a \( \Bbbk \)-rational point.
We write \( \Def_{(X, P)} := \Def_{\SO_{X, P}} \).
When \( X \) is non-singular outside a finite set
(i.e., has only isolated singularities), we define a functor
\( \Def_{X}^{(\loc)} \) on \( \SC_{\Lambda} \) by
\[ \Def^{(\loc)}_{X}(A) := \prod\nolimits_{P \in \Sing X}\Def_{(X, P)}(A)  \]
for \( A \in \SC_{\Lambda} \), where \( \Sing X \) stands for the singular locus.
\end{dfn}

\begin{remsub}\label{remsub:dfn:Defloc}
There is a natural morphism \( \Def_{X} \to \Def_{(X, P)} \) of functors
for any \( \Bbbk \)-rational point \( P \in X \) which
sends \( (X_{A}, \iota) \in \Def_{X}(A) \) to the pair consisting of
the local \( A \)-algebra \( \SO_{X_{A}, \iota(P)} \) and the
isomorphism \( \SO_{X, P} \isom  \SO_{X_{A}, \iota(P)} \otimes_{A} \Bbbk \)
induced by \( \iota \).
As a consequence, we have a natural morphism
\( \Def_{X} \to \Def_{X}^{(\loc)} \) of functors when
\( X \) has only isolated singularities.
\end{remsub}

We shall give a proof of the following well-known:

\begin{thm}[{\cite[Proposition~6.4]{WahlSmoothings}}]\label{thm:H2T=0}
Let \( X \) be an algebraic \( \Bbbk \)-scheme with only isolated singularities.
Assume that \( \OH^{2}(X, \Theta_{X/\Bbbk}) = 0 \),
where \( \Theta_{X/\Bbbk}\) denotes the tangent sheaf
\( \SHom_{\SO_{X}}(\Omega^{1}_{X/\Bbbk}, \SO_{X}) \).
Then, the morphism \( \Def_{X} \to \Def_{X}^{(\loc)} \) of functors is smooth.
\end{thm}

\begin{proof}
Let \( B \to A \) be a surjection in \( \SC_{\Lambda} \) with the kernel \( I \)
satisfying \( I\GM_{B} = 0 \).
It suffices to prove that
\begin{equation}\label{eq:smoothXXloc}
\Def_{X}(B) \to \Def_{X}(A) \times_{\Def^{(\loc)}_{X}(A)} \Def^{(\loc)}_{X}(B)
\end{equation}
is surjective.
An element of the right-hand side of \eqref{eq:smoothXXloc}
consists of
\begin{itemize}
\item a flat \( A \)-scheme \( X_{A} \) with an isomorphism
\( \iota_{A} \colon X_{A} \times_{\Spec A} \Spec \Bbbk \isom X\), and

\item  flat \( B \)-algebras \( S^{(P)}_{B} \) for any \( P \in \Sing X \)
with isomorphisms
\( \iota_{B}^{(P)} \colon S^{(P)}_{B} \otimes_{B} \Bbbk\isom \SO_{X, P} \),
\end{itemize}
where, for any \( P \in \Sing X \),
we can find an isomorphism
\( \Psi_{P} \colon \SO_{X_{A}, P} \xrightarrow{\isom} S^{(P)}_{B} \otimes_{B} A \)
such that the composite
\[ \SO_{X_{A}, P} \xrightarrow{\Psi_{P}}
S^{(P)}_{B} \otimes_{B} A \to S^{(P)}_{B} \otimes_{B} \Bbbk
\xrightarrow{\iota_{B}^{(P)}} \SO_{X, P}\]
is the homomorphism induced by \( \iota_{A} \).

We apply the obstruction theory of infinitesimal deformations
using the cotangent complexes (cf.\ \cite{Illusie}, \cite{LS}):
Let \( \BLL_{Z/Y} \) be the cotangent complex
(as an object of the derived category \( D^{-}(\text{Coh}(Z)) \)) for a morphism
\( Z \to Y \) of schemes. For a coherent sheaf \( \SF \) on \( Z \), we denote
the associated cohomology groups/sheaves by
\begin{align*}
\OT^{i}(Z/Y, \SF) &= \BExt^{i}(\BLL_{Z/Y}, \SF) =
\OH^{i}(\RHom_{\SO_{Z}}(\BLL_{Z/Y}, \SF)), \\
\ST^{i}(Z/Y, \SF) &= \underline{\BExt}^{i}(\BLL_{Z/Y}, \SF)
= \SH^{i}(\SRHom_{\SO_{Z}}(\BLL_{Z/Y}, \SF))
\end{align*}
for \( i \geq 0 \). When \( Y\) is the affine scheme \( \Spec A \),
we write ``\( /A \)'' instead of ``\( /\Spec A \)''.
If \( Z \) is also affine in addition,
then \( Z \) is replaced with the coordinate ring.
We recall a few properties on \( \OT^{i}(Z/Y, \SF) \) and
\( \ST^{i}(Z/Y, \SF) \).
\begin{enumerate}
    \renewcommand{\theenumi}{\roman{enumi}}
    \renewcommand{\labelenumi}{(\theenumi)}
\item  \label{cot:1}
\( \OT^{0}(Z/Y, \SF)
\isom \SHom_{\SO_{Z}}(\Omega^{1}_{Z/Y}, \SF)\)
(cf.\ \cite[Chapitre~II, (1.2.7.4) and Corollaire 1.2.4.3]{Illusie}).

\item \label{cot:2}
If \( Z \to Y \) is smooth, then
\( \OT^{i}(Z/Y, \SF) = \ST^{i}(Z/Y, \SF) = 0 \)
for all \( i > 0 \) (cf.\ \cite[Chapitre~III, Proposition~3.1.2]{Illusie}).

\item \label{cot:3} If \( Z \to Y \) is the base change of a flat morphism
\( Z_{1} \to Y_{1} \) by an affine morphism \( Y \to Y_{1} \),
then \( \OT^{i}(Z/Y, \SF) \isom
\OT^{i}(Z_{1}/Y_{1}, u_{*}\SF) \)
for the induced affine morphism \( u \colon Z \to Z_{1} \)
(cf.\ \cite[2.3.2]{LS}, \cite[Chapitre~II, Corollaire~2.3.11]{Illusie}).
\end{enumerate}

We want to find a \( B \)-scheme \( X_{B} \), whose structure sheaf \( \SO_{X_{B}} \) is
sitting inside a commutative diagram
\[ \begin{CD}
0 @>>> I\SO_{X_{A}} @>>> \SO_{X_{B}} @>>> \SO_{X_{A}} @>>> 0 \\
@. @AAA @AAA @AAA \\
0 @>>> I @>>> B @>>> A @>>> 0
\end{CD}\]
of algebra extensions, where the top exact sequence means that
the ideal sheaf of \( X_{A} \) in \( X_{B} \) is square zero and is isomorphic to
\( I\SO_{X_{A}} \) as an \( \SO_{X_{A}} \)-module, and where
the vertical arrows are natural homomorphisms for the \( A \)-scheme \( X_{A} \)
and the \( B \)-scheme \( X_{B} \).
Note that \( I\SO_{X_{A}} \isom I \otimes_{\Bbbk} \SO_{X} \),
since \( I\GM_{B} = 0 \). We write \( I\SO_{X} := I \otimes_{\Bbbk} \SO_{X} \).
The obstruction class \( \ob(X_{A}) \) for the existence of \( X_{B} \)
lies in the cohomology group
\( T^{2}(X_{A}/A, I\SO_{X_{A}}) \)
(cf.\ \cite[Chapitre~III, Th\'eor\`em~2.1.7.(i)]{Illusie})
which is isomorphic to \( T^{2}(X/\Bbbk, I\SO_{X}) \) by \eqref{cot:3} above.
We consider the spectral sequence
\[ E_{2}^{p, q} = \OH^{p}(X, \ST^{q}(X/\Bbbk, I\SO_{X})) \Rightarrow
E^{p+q} = \OT^{p+q}(X/\Bbbk, I\SO_{X}). \]
Then, \( E_{2}^{2, 0} = 0 \) by \( \OH^{2}(X, \Theta_{X/\Bbbk}) = 0 \),
since \( \ST^{0}(X/\Bbbk, I\SO_{X}) = \Theta_{X/\Bbbk}\otimes_{\Bbbk} I  \)
by \eqref{cot:1}.
Moreover, \( E_{2}^{1, 1} = 0 \), since \( \Sing X \) is finite and
since
\( \ST^{1}(X/\Bbbk, I\SO_{X}) = \ST^{1}(X/\Bbbk, \SO_{X}) \otimes_{\Bbbk} I\)
is supported on \( \Sing X \) by \eqref{cot:2}.
In particular,
\[ \OT^{1}(X/\Bbbk, I\SO_{X}) = E^{1} \to
E_{2}^{0, 1} =
\prod\nolimits_{P \in \Sing X} \OT^{1}(\SO_{X, P}/\Bbbk, I\SO_{X, P})\]
is surjective, and
\[ \OT^{2}(X/\Bbbk, I\SO_{X}) = E^{2} \to E_{2}^{0, 2}   =
\prod\nolimits_{P \in \Sing X} \OT^{2}(\SO_{X, P}/\Bbbk, I\SO_{X, P}) \]
is injective.
The class \( \ob(X_{A}) \) lies in the kernel of
\( E^{2} \to E_{2}^{0, 2}  \),
since \( S^{(P)}_{A} \) possesses a lifting \( S^{(P)}_{B} \) to \( B \).
Thus, \( \ob(X_{A}) = 0 \). As a consequence,
we have a flat \( B \)-scheme \( X_{B} \) with
an isomorphism
\( \iota_{B/A} \colon X_{B} \times_{\Spec B} \Spec A \isom X_{A} \) over \( A \).
In other words, we have an element \( (X_{B}, \iota_{B}) \) of \( \Def_{X}(B) \)
which is mapped to \( (X_{A}, \iota_{A}) \in \Def_{X}(A) \).
However, \( \SO_{X_{B}, P} \) may not be isomorphic to \( S^{(P)}_{B} \)
as a lift of \( S^{(P)}_{A} \) to \( B \).
But, by a usual obstruction theory of deformations,
the difference of two lifts lies in
\( T^{1}(S^{(P)}_{A}/A, I\SO_{X, P}) \isom
T^{1}(\SO_{X, P}/\Bbbk, I\SO_{X, P})\)
(cf.\ \cite[Chapitre~III, Th\'eor\`em~2.1.7.(ii)]{Illusie}).
Since \( E^{1} \to E^{0, 1}_{2}\) is surjective,
we can replace the lift \( X_{B} \) by another one so that \( \SO_{X_{B}, P} \)
is isomorphic to \( S^{(P)}_{B} \) for any \( P \in \Sing X \).
Thus, \( (X_{B}, \iota_{B}) \) is mapped to the given element of
the right-hand side of \eqref{eq:smoothXXloc}.
Therefore, \eqref{eq:smoothXXloc} is surjective.
\end{proof}

Finally in Section~\ref{sect:Deformation},
we shall give the following result on algebraization
related to Theorem~\ref{thm:H2T=0}.
Roughly speaking, the result says that under suitable conditions,
any given local algebraic deformations of isolated singularities
extend to a global algebraic deformation.

\begin{thm}[algebraization]\label{thm:algn}
Let \( X \) be a normal projective variety defined over
an algebraically closed field \( \Bbbk \) with only isolated singularities.
Assume that
\begin{enumerate}
    \renewcommand{\theenumi}{\roman{enumi}}
    \renewcommand{\labelenumi}{(\theenumi)}
\item \label{thm:algn:ass1} \( \OH^{2}(X, \Theta_{X/\Bbbk}) = 0 \),

\item \label{thm:algn:ass2}
the formal deformation \( \GX \to \Spf \GR \) associated with
a pro-representable hull of \( \Def_{X} \) is projective, i.e.,
\( \GX \) admits a relatively ample invertible sheaf over \( \Spf \GR \).
\end{enumerate}
Let \( T \) be an algebraic \( \Lambda \)-scheme
and let \( o \in T \) be a \( \Bbbk \)-rational point.
For every singular point \( P \in X \), assume that
we are given an affine \'etale neighborhood \( (U_{(P)}, P')\)
of \( P \) in \( X \) with a unique singular point and
given a deformation \( \SU_{(P)} \to T \) of \( U_{(P)} \)
with reference point \( o \). Then, there exist
\begin{itemize}
\item  an \'etale neighborhood \( (T', o') \) of \( (T, o) \), and

\item  a projective morphism \( Z \to T' \) which is
a deformation of \( X \) with reference point \( o' \)
\end{itemize}
such that, for any singular point \( P \),
\( (Z, P) \) and \( (\SU_{(P)}, P') \) have a common \'etale neighborhood,
i.e., the formal completions of
the local rings \( \SO_{Z, P} \) and \( \SO_{\SU_{(P)}, P'} \)
are isomorphic to each other
\emph{(}\cite[Corollary~(2.6)]{ArtinApprox}\emph{)}.
\end{thm}

\begin{proof}
By Remark~\ref{remsub:dfn:effective,algebraizable,hull:2} and
by Lemma~\ref{lem:functorsisom},
for any point \( P \in \Sing X \), there exists
an algebraic deformation \( W_{(P)} \to T_{(P)} \)
of an affine open neighborhood of \( P \) over
an algebraic \( \Lambda \)-scheme \( T_{(P)} \) such that
the formal completion of the fiber induces a hull
of the deformation functor \( \Def_{(X, P)} \).
By replacing \( (T, o) \) with an \'etale neighborhood \( (T', o') \),
we have a morphism \( \varphi_{(P)} \colon T \to T_{(P)} \) such that
the deformations \( \SU_{(P)} \to T \) and \( W_{(P)} \times_{T_{(P)}} T \)
are equivalent to each other
in the sense of \cite[Example~(4.5)]{ArtinVersal}
by the versality there.
In particular, \( (\SU_{(P)}, P') \) and
\( (W_{(P)} \times_{T_{(P)}} T, P) \) have a common \'etale neighborhood.

By the assumption \eqref{thm:algn:ass2} and
by Remark~\ref{remsub:dfn:effective,algebraizable,hull:200},
there exist an algebraic \( \Lambda \)-scheme
\( S \) with a \( \Bbbk \)-rational point \( b \) and a flat projective morphism
\( W \to S \) such that
\begin{itemize}
\item  \( W \to S\) is a deformation of \( X \) with the reference point \( b \),

\item the completion of \( \SO_{S, b} \) is isomorphic to \( \GR \), and

\item  \( \GX \) is isomorphic over \( \Spf \GR \) to
the formal completion of \( W \) along the fiber over \( b \).
\end{itemize}
Since \( \Spec \SO_{W, P} \) is a deformation
of \( \Spec \SO_{X, P} \) over \( S \) for any \( P \in \Sing X \),
after replacing \( (S, b) \) with an \'etale neighborhood,
we have a morphism \( \phi_{(P)} \colon S \to T_{(P)}\)
such that \( (W, P) \) and \( (W_{(P)} \times_{T_{(P)}} S, P) \)
have a common \'etale neighborhood as in \cite[Example~(4.5)]{ArtinVersal}.

Let \( \phi \colon S \to \prod_{P \in \Sing X} T_{(P)} \) be the
morphism defined by \( \{\phi_{(P)}\} \). Then, \( \phi \) is smooth
at \( b \) by Theorem~\ref{thm:H2T=0}. Hence, we may assume that \(
\phi \) is smooth by replacing \( S \) with an open neighborhood of
\( b \). Let \( \varphi \colon T \to \prod_{P \in \Sing X} T_{(P)}
\) be the morphism defined by \( \{\varphi_{(P)}\} \). Since \( \phi
\) is smooth, after replacing \( (T, o) \) with an \'etale
neighborhood, we have a morphism \( \psi \colon T \to S \) such that
\( \psi(o) = b \) and \( \varphi = \phi \circ \psi \). Then, the
base change \( Z = W \times_{S} T \to T \) satisfies the expected
conditions.
\end{proof}

\begin{remsub}\label{remsub:H2O=0}
The assumption \eqref{thm:algn:ass2} above is satisfied if
\( \OH^{2}(X, \SO_{X}) = 0 \)
(cf.\ \cite[Exp.~III, Proposition~7.2]{SGA1}).
\end{remsub}


\section{Deformations of certain projective surfaces with toric singularities of class T}
\label{sect:DeformationClassT}

We shall construct some algebraic deformations of
a projective normal surface \( X \)
with toric singularities of class T under extra assumptions.
We treat in Theorem~\ref{thm:globalsmoothing} only deformations
over \( \Bbbk \), but in Theorem~\ref{thm:DVRsmoothing},
deformations over a complete discrete
valuation ring with residue field \( \Bbbk \).
Here, we allow also
rational double points on \( X \) in Theorem~\ref{thm:globalsmoothing}.
As a corollary of Theorem~\ref{thm:globalsmoothing},
in Corollary~\ref{cor:log del Pezzo index two},
we shall give a correct proof of \cite[Theorem~5.16]{logdelPezzo}
that any log del~Pezzo surface of index two
admits a \( \BQQ \)-Gorenstein smoothing to del~Pezzo surfaces.
Theorem~\ref{thm:DVRsmoothing} is applied to the study of
the algebraic fundamental groups of smooth geometric fibers
in Corollary~\ref{cor:Pi1} and Remark~\ref{remsub:cor:Pi1}.

To begin with,
we recall the following well-known result.

\begin{lem}\label{lem:SmoothingLCI}
Let \( X \) be an affine algebraic \( \Bbbk \)-variety
with a \( \Bbbk \)-rational point \( P \) such that
\( X \setminus \{P\} \) is non-singular and
\( (X, P) \) is a local complete intersection singularity,
i.e., the local ring \( \SO_{X, P} \) is a complete intersection.
Then, there exist an affine open neighborhood \( X' \),
an affine flat morphism \( \SX \to T \)
over a non-singular curve \( T \),
and a \( \Bbbk \)-rational point \( o \in T \)
such that
\begin{enumerate}
\item  \( \SX \times_{T} o \isom X' \),

\item  \( \SX \to T \) is smooth over \( T \setminus \{o\} \).
\end{enumerate}
In particular, the singularity \( (X, P) \) admits a smoothing.
\end{lem}

\begin{proof}
The local ring \( \SO_{X, P} \) is isomorphic to
the localization of
\[ A = \Bbbk[\xtt_{1}, \ldots, \xtt_{n + l}]/(f_{1}, \ldots, f_{l}) \]
at the origin \( \{\xtt_{1} = \cdots = \xtt_{n+l} = 0\}\)
for a certain regular sequence
\( f_{1} \), \ldots, \( f_{l} \), where \( n = \dim X \).
Hence, we may assume that \( X = \Spec A \) and \( P \) is the origin.
For \( 1 \leq i \leq l \),
let \( F_{i} \) be the homogeneous polynomial in
\( \Bbbk[\xtt_{0}, \xtt_{1}, \ldots, \xtt_{n+l}] \) such that
\( F_{i}(1, \xtt_{1}, \ldots, \xtt_{n+l}) = f_{i} \) and
\( \xtt_{0} \nmid F_{i} \).
Then, the complete intersection closed subscheme
\( \overline{X} \subset \BPP^{n+l}_{\Bbbk} \)  defined by
\( \{F_{1} = \cdots = F_{l} = 0\} \) is regarded as
the closure of \( X\), i.e.,
\( \overline{X} \cap D_{+}(\xtt_{0}) = X\),
where \( D_{+}(\xtt_{0}) = \{\xtt_{0} \ne 0\}  \).
By Bertini's theorem, since \( \Bbbk \) is algebraically closed,
we can take homogeneous polynomials \( G_{1} \), \ldots, \( G_{l} \) with
\( \deg G_{i} = \deg F_{i} \) such that
\( \{G_{i} = 0\} \) is non-singular for all \( 1 \leq i \leq l \)
and \( \sum_{i = 1}^{l} \{G_{i} = 0\}  \) is a simple normal crossing divisor.
For \( s \), \( t \in \Bbbk \), let
\( H_{i}(s, t) \) be the divisor \( \{sF_{i} + tG_{i} = 0\} \).
Then, there is an open neighborhood \( T' \subset \BPP^{1}\) of \( (0:1) \)
such that, for any closed point \( (s : t) \in T'\),
\( H_{i}(s, t) \) is non-singular for all \( 1 \leq i \leq l \)
and \( \sum_{i= 1}^{l}H_{i}(s, t) \) is a simple normal crossing divisor.
Thus, for \( T = T' \cup \{(1:0)\} \),
\[ \SX := \{H_{1}(s, t) = \cdots = H_{l}(s, t) = 0\} \cap D_{+}(\xtt_{0}) \to T\]
is a desired morphism with \( o = (1:0) \in T \).
\end{proof}

By \cite[Corollary~6]{ArtinIsol}, we have:

\begin{corsub}\label{corsub:SmoothingLCI}
Any rational Gorenstein surface singularity \emph{(}rational double point\emph{)}
admits a smoothing.
\end{corsub}

The following is our main technical tool for constructing
desired surfaces of general type.

\begin{thm}\label{thm:globalsmoothing}
Let \( \Bbbk \) be an algebraically closed field. Let \( X \) be a
normal projective surface defined over \( \Bbbk \) whose
singularities are rational double points or toric singularities of
class T. Assume that \( X \) satisfies the following two
conditions\emph{:}
\begin{enumerate}
    \renewcommand{\theenumi}{\roman{enumi}}
    \renewcommand{\labelenumi}{(\theenumi)}
\item \label{thm:globalsmoothing:ass1} \( \OH^{2}(X, \Theta_{X/\Bbbk}) = 0 \).

\item  \label{thm:globalsmoothing:ass2}
\( \OH^{2}(X, \SO_{X}) = 0 \).
\end{enumerate}
Then, there is a deformation \( \SX \to T \) of \( X \)
over a non-singular algebraic curve \( T \) defined over \( \Bbbk \)
with a reference \( \Bbbk \)-rational point \( o \in T\)
such that\emph{:}
\begin{enumerate}
\item \label{thm:globalsmoothing:1}
\( \SX \to T \) is a projective morphism and
it is smooth over \( T \setminus \{o\} \).

\item \label{thm:globalsmoothing:2}
\( \SX \) is normal, \( rK_{\SX} \) is Cartier, and
\( \SO_{\SX}(rK_{\SX})|_{X} \isom \SO_{X}(rK_{X}) \)
for the Gorenstein index \( r \) of \( X \).
\end{enumerate}
In particular, after replacing \( T \) with an open neighborhood of \( o \),
the following hold for any \( \Bbbk \)-rational point \( t \) of \( T \setminus \{o\} \)
and the fiber \( X_{t} := \SX \times_{T} t \) over \( t \)\emph{:}
\begin{enumerate}
    \addtocounter{enumi}{2}
\item  \label{thm:globalsmoothing:3}
\( X_{t} \) is a non-singular projective surface defined over \( \Bbbk \).

\item \label{thm:globalsmoothing:4}
\( \dim \OH^{i}(X_{t}, \SO_{X_{t}}) = \dim \OH^{i}(X, \SO_{X}) \)
for all \( i \geq 0 \).

\item  \label{thm:globalsmoothing:5} \( K_{X_{t}}^{2} = K_{X}^{2}\).

\item \label{thm:globalsmoothing:6}
\( \OH^{2}(X_{t}, \Theta_{X_{t}/\Bbbk}) = 0 \).

\item \label{thm:globalsmoothing:7}
If \( K_{X} \) \emph{(}resp.\ \( -K_{X} \)\emph{)} is ample,
then so is \( K_{X_{t}} \) \emph{(}resp.\ \( -K_{X_{t}} \)\emph{)}.

\item \label{thm:globalsmoothing:8}
If \( K_{X} \) \emph{(}resp.\ \( -K_{X} \)\emph{)} is nef and big,
then so is \( K_{X_{t}} \) \emph{(}resp.\ \( -K_{X_{t}} \)\emph{)}.
\end{enumerate}
\end{thm}

\begin{proof}
In Theorem~\ref{thm:globalsmoothing}, it is enough to consider the deformation theory only
in the case where \( \Lambda = \Bbbk \).
First of all, we shall prove \eqref{thm:globalsmoothing:3}--\eqref{thm:globalsmoothing:8}
assuming \eqref{thm:globalsmoothing:1} and \eqref{thm:globalsmoothing:2}.
The assertions \eqref{thm:globalsmoothing:3} and \eqref{thm:globalsmoothing:4}
follow from \eqref{thm:globalsmoothing:1},
the assumption~\eqref{thm:globalsmoothing:ass2}, and
from the upper semi-continuity theorem for the flat morphism \( \SX \to  T \).
Since \( (1/2)r^{2}K_{X}^{2} \) is the leading coefficient of
the Hilbert polynomial \( \chi(X, \SO_{X}(mrK_{X})) \)
with respect to the variable \( m \),
\eqref{thm:globalsmoothing:5} follows from \eqref{thm:globalsmoothing:2} and
the upper semi-continuity theorem for the flat morphism \( \SX \to  T \).
For \eqref{thm:globalsmoothing:6}, let us consider the relative tangent sheaf
\( \Theta_{\SX/T} :=
\SHom_{\SO_{\SX}}(\Omega^{1}_{\SX/T}, \SO_{\SX}) \),
where the canonical homomorphism
\( \Theta_{\SX/T} \otimes_{\SO_{\SX}} \SO_{X} \to
\Theta_{X/\Bbbk}\)
is an isomorphism outside \( \Sing X \) but
another canonical homomorphism
\( \Theta_{\SX/T} \otimes_{\SO_{\SX}} \SO_{X_{t}} \to \Theta_{X_{t}/\Bbbk} \)
is an isomorphism. Then, we have
\( \OH^{2}(X, \Theta_{\SX/T} \otimes_{\SO_{\SX}} \SO_{X})
= 0\) by \( \OH^{2}(X, \Theta_{X/\Bbbk}) = 0 \).
By the upper semi-continuity theorem applied to
the sheaf \( \Theta_{\SX/T} \) flat over \( T\) and by
the base change isomorphism, we have the vanishing
\eqref{thm:globalsmoothing:6}.

The assertion \eqref{thm:globalsmoothing:7} is derived from \eqref{thm:globalsmoothing:2}
and \cite[III, Th\'eorem\`e~(4.7.1)]{EGA}.
The last assertion \eqref{thm:globalsmoothing:8} is derived from also
a general property.
The detail is as follows (cf.\ \cite[Chapter~III, \S 4a, Problem]{ZDA}):
Let \( D \) be a Cartier divisor on \( \SX \) such that \( D_{o} := D|_{X}  \)
is nef and big. It suffices to show that \( D_{t} = D|_{X_{t}} \) is also nef and big
for any point \( t  \) in a neighborhood of \( o \).
Let \( H \) be another Cartier divisor on \( \SX \)
over \( T \). Then, \( D_{o}H_{o} = D_{t}H_{t} \),
where \( H_{t} := H|_{X_{t}} \) by considering
the Hilbert polynomial \( \chi(X_{t}, \SO_{\SX}(m_{1}H + m_{2}D)|_{X_{t}})\)
of two variables \( m_{1} \), \( m_{2} \).
In particular, \( D_{t}^{2} = D_{o}^{2}> 0 \) and \( D_{t}H_{t} > 0 \)
for an ample divisor \( H_{t} \). Hence, \( D_{t} \) is big.
Thus, by replacing \( T \) with an affine open neighborhood of \( o \),
we have an effective divisor \( G \) on \( \SX \) such that \( mD \sim G \)
for some \( m > 0 \). Here, \( D_{t} \) is nef if and only if
\( D_{t}C \geq 0 \) for any irreducible component \( C \) of
\( G_{t} = G|_{X_{t}} \).
By eliminating the vertical components of \( G \), we may assume that
any irreducible component \( G_{\lambda} \) of \( G \) dominates \( T \).
Then, \( 0 \leq D_{o} G_{\lambda, o} = D_{t}G_{\lambda, t}\), where
\( G_{\lambda, t} = G_{\lambda}|_{X_{t}} \).
Replacing \( T \) by the Stein factorization of \( G_{\lambda} \to T \)
if necessary, we see that \( D_{t} \) is nef for a general point of \( T \).

We shall construct a deformation \( \SX \to T \) satisfying
\eqref{thm:globalsmoothing:1} and \eqref{thm:globalsmoothing:2} by
applying Theorem~\ref{thm:algn}.
Note that the assumption \eqref{thm:algn:ass1} of Theorem~\ref{thm:algn}
is identical to \eqref{thm:globalsmoothing:ass1} above
and the other assumption \eqref{thm:algn:ass2} of Theorem~\ref{thm:algn}
is satisfied by \eqref{thm:globalsmoothing:ass2} above
(cf.\ Remark~\ref{remsub:H2O=0}).
For a singular point \( P \), we shall construct a deformation
\( \SU_{(P)} \to T \) of an affine \'etale neighborhood of \( (X, P) \)
over a non-singular curve \( T \) as follows.
If \( P \) is a rational double point, then we set
\( \SU_{(P)} \to T \) to be a deformation obtained in
Lemma~\ref{lem:SmoothingLCI} (cf.\ Corollary~\ref{corsub:SmoothingLCI}),
which is a smoothing of \( (X, P) \).
Note that \( \SU_{(P)} \) is normal and Gorenstein.
If \( (X, P) \) is a toric singularity of class T, then,
by Theorem~\ref{thm:graph2toric000},
there is an affine \'etale neighborhood \( (U_{(P)}, P') \)
of \( (X, P) \) which is also an \'etale neighborhood of
\( (V_{(P)}, \bzero) \) for an affine toric surface \( V_{(P)} \)
with the closed orbit \( \bzero \). In this case, by applying
Theorem~\ref{thm:localsmoothing} to \( V_{(P)}\)
and Lemma~\ref{lem:functorsisom01} to
\( (U_{(P)}, P') \to (V_{(P)}, \bzero) \), we have a deformation
\( \SU_{(P)} \to T \) of \( U_{(P)} \) with
a \( \Bbbk \)-rational reference point \( o \)
such that
\begin{itemize}
\item  \( \SU_{(P)} \to T \) is smooth over \( T \setminus \{o\} \),

\item  \( \SU_{(P)} \) is normal, \( rK_{\SU_{(P)}} \) is Cartier, and
\[ \SO_{\SU_{(P)}}(rK_{\SU_{(P)}})|_{U_{(P)}} \isom
\SO_{U_{(P)}}(rK_{U_{(P)}}) \]
for the Gorenstein index \( r \) of \( (X, P) \).
\end{itemize}
Note that in the construction above, we can take a common non-singular curve \( T \)
and a reference point \( o \).

Applying Theorem~\ref{thm:algn} to \( \SU_{(P)} \to T \)
for any \( P \in \Sing X \), after replacing \( (T, o) \) with
an \'etale neighborhood,
we have a deformation \( \SX \to T \) of \( X \) satisfying the required conditions
\eqref{thm:globalsmoothing:1} and \eqref{thm:globalsmoothing:2}.
Thus, we are done.
\end{proof}

We shall correct the proof of
\cite[Theorem~5.16]{logdelPezzo} concerning \( \BQQ \)-Gorenstein
smoothings of log del~Pezzo surfaces of index two
by proving the following stronger assertion as an application of
Theorem~\ref{thm:globalsmoothing}.

\begin{cor}\label{cor:log del Pezzo index two}
Let \( X \) be a log del~Pezzo surface of index two
over an algebraically closed field \( \Bbbk \), i.e.,
\( X \) is a normal projective surface,
\( X \) has only log-terminal singularities, \( -K_{X} \)
is an ample non-Cartier divisor, and \( 2K_{X} \) is Cartier
\emph{(}cf.\ \cite[Definition~3.2]{logdelPezzo}\emph{)}.
Then, there is a projective deformation
\( \SX \to T \) of \( X \) over a non-singular curve \( T \)
over \( \Bbbk \) such that
\begin{itemize}
\item  \( \SX \to T \) is smooth outside \( X \),

\item \( 2K_{\SX} \) is Cartier, and

\item any closed fiber of \( \SX \to T \) other than \( X \) is a del~Pezzo
surface with \( K^{2} = K_{X}^{2} \).
\end{itemize}
\end{cor}

\begin{proof}
A singular point of \( X \) is
either a rational double point or
a singular point of type \( \SKK_{n} \) (cf.\ \cite[Lemma~4.15]{logdelPezzo}),
where \( \SKK_{n} \) is just
the toric singularity of type \( \frac{1}{4n}(1, 2n-1) \):
this is of type \( T(n, 2, 1) \)
(cf.\ Definition~\ref{dfn:classT}).
We have \( \OH^{2}(X, \SO_{X}) = \OH^{0}(X, \SO_{X}(K_{X}))^{\vee} = 0 \)
by the Serre duality theorem, since \( -K_{X} \) is ample.
Hence, by Theorem~\ref{thm:globalsmoothing}, it suffices to show
\( \OH^{2}(X, \Theta_{X/\Bbbk}) = 0 \), or equivalently,
\( \OH^{0}(X, \SHom_{\SO_{X}}(\Theta_{X/\Bbbk}, \SO_{X}(K_{X}))) = 0 \)
by Serre duality.
We know that \( \OH^{0}(X, \SO_{X}(-K_{X})) \ne 0 \).
In fact, for the
minimal resolution \( \mu \colon M \to X \) of singularities, we have
\begin{align*}
&\OH^{0}(X, \SO_{X}(-K_{X})) \isom
\OH^{0}(M, \SO_{M}(K_{M} + \mu^{*}(-2K_{X}))), \\
&\OH^{2}(M, \SO_{M}(K_{M} + \mu^{*}(-2K_{X}))) \isom
\OH^{0}(M, \mu^{*}\SO_{X}(2K_{X}))^{\vee} = 0, \quad \text{and} \\
&\dim \OH^{0}(M, \SO_{M}(K_{M} + \mu^{*}(-2K_{X}))) \geq
\chi(M, \SO_{M}(K_{M} + \mu^{*}(-2K_{X}))) \\
&\phantom{xxxxx} =
\frac{1}{2} (K_{M} + \mu^{*}(-2K_{X})) \mu^{*}(-2K_{X}) + 1 =
K_{X}^{2} + 1 > 0.
\end{align*}
Taking a non-zero section of \( \SO_{X}(-K_{X}) \), we obtain  an injection
\( \SO_{X}(K_{X}) \injmap \SO_{X} \), and
hence an injection
\[ \SHom_{\SO_{X}}(\Theta_{X/\Bbbk}, \SO_{X}(K_{X}))
\injmap \SHom_{\SO_{X}}(\Theta_{X/\Bbbk}, \SO_{X})
\isom (\Omega_{X/\Bbbk}^{1})^{\vee\vee},  \]
where the right sheaf is isomorphic to
\( \mu_{*}\Omega^{1}_{M/\Bbbk}\)
by Proposition~\ref{prop:tangentsheaves}.\eqref{prop:tangentsheaves:reflexive0}.
Since  \( M \) is rational,
\[ \OH^{0}(X, \SHom_{\SO_{X}}(\Theta_{X/\Bbbk}, \SO_{X}(K_{X})))
\subset \OH^{0}(M, \Omega_{M/\Bbbk}^{1}) = 0. \]
Thus, we are done.
\end{proof}

\begin{thm}\label{thm:DVRsmoothing}
Let \( \Lambda \) be a complete discrete valuation ring
with an algebraically closed residue field \( \Bbbk \).
Let \( X_{\Lambda} \to \Spec \Lambda \) be a flat projective morphism
satisfying the following two conditions\emph{:}
\begin{enumerate}
    \renewcommand{\theenumi}{\roman{enumi}}
    \renewcommand{\labelenumi}{(\theenumi)}
\item \label{thm:DVRsmoothing:ass1}
The closed fiber \( X = X_{\Bbbk}\) is a normal projective surface
with only toric singularities of class T satisfying
\( \OH^{2}(X, \Theta_{X/\Bbbk}) = \OH^{2}(X, \SO_{X}) = 0 \).

\item \label{thm:DVRsmoothing:ass2}
For any singular point \( P \) of \( X \), let \( (n_{(P)}, q_{(P)}) \)
be the type of the toric singularity.
Then, there exist an affine neighborhood \( Y_{(P)} \) of \( P \)
in \( X_{\Lambda} \) and
two prime divisors \( B_{1\, (P)}\), \( B_{2\, (P)} \) on \( Y_{(P)} \)
containing \( P \) such that \( (Y_{(P)}, B_{1\, (P)}, B_{2\, (P)})  \)
satisfies the condition \( C(n_{(P)}, q_{(P)})' \) over \( \Spec \Lambda \)
\emph{(}cf.\ Definition~\emph{\ref{dfn:cond_b})}.
\end{enumerate}
Then, there exist an algebraic deformation \( Z \to T \)
of \( X \) over an algebraic smooth \( \Lambda \)-scheme \( T \)
of relative dimension one,
a \( \Bbbk \)-rational point \( o \in T\), and
a section \( \sigma \colon \Spec \Lambda \to T \) such that\emph{:}
\begin{enumerate}
\item \label{thm:DVRsmoothing:1}
\( \sigma(\GM_{\Lambda}) = o \) and
\( Z \times_{T, \sigma} \Spec \Lambda \isom X_{\Lambda} \).
\item \label{thm:DVRsmoothing:2}
\( Z \) is smooth over \( T \setminus \sigma(\Spec \Lambda) \).
\item  \label{thm:DVRsmoothing:3}
\( Z \) is normal, \( rK_{Z}  \) is Cartier, and
\( \SO_{Z}(rK_{Z})|_{X_{\Lambda}} \isom \SO_{X_{\Lambda}}(rK_{X_{\Lambda}}) \)
for the Gorenstein index \( r \) of \( X \).
\end{enumerate}
\end{thm}

\begin{proof}
By Theorems~\ref{thm:graph2toric}, \ref{thm:localsmoothing} and
by Lemma~\ref{lem:functorsisom01},
we have an algebraic smooth \( \Lambda \)-scheme \( T \)
of relative dimension one,
a \( \Bbbk \)-rational reference point \( o \in T\),
a section \( \sigma \colon \Spec \Lambda \to T \)
with \( \sigma(\GM_{\Lambda}) = o \),
and a flat family \( V_{(P)} \to T \) of normal affine surfaces
for any singular point \( P \) of \( X \) such that
\begin{enumerate}
    \renewcommand{\theenumi}{$V$\arabic{enumi}}
    \renewcommand{\labelenumi}{(\theenumi)}
\item \label{item:V1}
\( V_{(P)} \times_{T, \sigma}
\Spec \Lambda \isom Y_{(P)} \),

\item \label{item:V2}
\( V_{(P)} \to T \) is smooth over
\( T \setminus \sigma(\Spec \Lambda) \),

\item \label{item:V3}
\( V_{(P)} \) is normal, \( r_{P}K_{V_{(P)}} \) is Cartier
with \( \SO_{V_{(P)}}(r_{P}K_{V_{(P)}})|_{Y_{(P)}}
\isom \SO_{Y_{(P)}}(r_{P}K_{Y_{(P)}}) \)
for the Gorenstein index \( r_{P} \) of the toric singularity \( P \in X\).
\end{enumerate}
In fact, \( (Y_{(P)}, P) \) is an \'etale neighborhood of
an affine toric surface at the closed orbit by Theorem~\ref{thm:graph2toric},
and the toric surface admits a deformation satisfying
conditions \eqref{thm:localsmoothing:1}--\eqref{thm:localsmoothing:3}
of Theorem~\ref{thm:localsmoothing} since it is of class T, and finally
by Lemma~\ref{lem:functorsisom01}, we can lift the deformation to that
of \( Y_{(P)} \). Here, the section in Theorem~\ref{thm:localsmoothing}
induces the section \( \sigma \), since \( \Lambda \) is a complete
discrete valuation ring.

Let \( \GX \to \Spf \GR \) be the formal deformation associated with
a hull \( (\GR, \xi) \) of \( \Def_{X} \).
By the assumption \eqref{thm:DVRsmoothing:ass1},
we can find an algebraization \( W \to S \) of \( \GX \to \Spf \GR \)
as in the proof of Theorem~\ref{thm:algn} by Remark~\ref{remsub:H2O=0}.
Here, \( W \to S \) is a projective flat morphism over an
algebraic \( \Lambda \)-scheme \( S \),
the fiber over a \( \Bbbk \)-rational point \( b \in S \)
is isomorphic to \( X \), the completion of \( \SO_{S, b} \) is isomorphic to
\( \GR \), and
the formal completion of \( W \) along \( X \) is isomorphic to \( \GX \).
Since \( X_{\Lambda} \to \Spec \Lambda \) is a deformation of \( X \)
with reference point \( \GM_{\Lambda} \), we have a surjection
\( \GR \to \Lambda \) such that \( \GX \times_{\Spf \GR} \Spf \Lambda \)
is isomorphic to the formal completion of \( X_{\Lambda} \) along \( X \).
Hence, for the induced section
\( \sigma_{1} \colon \Spec \Lambda \to \Spec \GR \to \Spec \SO_{S, b} \to S \),
we have an isomorphism
\( W \times_{S, \sigma_{1}} \Spec \Lambda \isom X_{\Lambda}  \)
by \cite[III, Th\'eor\`eme~(5.4.1)]{EGA}.

Let \( W_{(P)} \to T_{(P)} \) be an algebraization of
the formal deformation associated with a hull of
\( \Def_{(X, P)} \) as in the proof of Theorem~\ref{thm:algn}.
After replacing \( (S, b) \) with an \'etale neighborhood and
replacing \( (T, o) \) with an \'etale neighborhood,
we have morphisms \( \phi \colon S \to \prod_{P \in \Sing X} T_{(P)} \)
and \( \varphi \colon T \to \prod_{P \in \Sing X} T_{(P)} \)
such that \( \varphi \circ \sigma = \phi \circ \sigma_{1}\) (cf.\ \eqref{item:V1})
and that, for any \( P \in \Sing X \),
\begin{itemize}
\item  \( (W, P) \) and \( (W_{(P)} \times_{T_{(P)}} S, P) \)
have a common \'etale neighborhood,
\item  \( (V_{(P)}, P) \) and \( (W_{(P)} \times_{T_{(P)}} T, P) \)
have a common \'etale neighborhood.
\end{itemize}
Let \( S_{T} \to T\) be the base change of \( \phi \) by \( \varphi \).
By replacing \( T\) with an open neighborhood of \( o \), we may assume that
\( S_{T} \to T \) is smooth, since \( \phi \) is smooth at \( b \)
by Theorem~\ref{thm:H2T=0}.
Now, we have a section
\( \sigma' = (\sigma, \sigma_{1}) \colon \Spec \Lambda \to S_{T}\).
Then, after replacing \( (T, o) \) with an \'etale neighborhood,
we have a section \( \psi \colon T \to S_{T} \)
such that \( \psi \circ \sigma = \sigma' \).
In fact, \( S_{T} \) is \'etale over
\( \BAA^{k}_{T} := \BAA^{k}_{\Lambda} \times_{\Spec \Lambda} T \) for some \( k \), and
\( \sigma' \) induces a section of \( \BAA^{k}_{\Lambda} \)
over \( \Spec \Lambda \). Here, we may assume that the section is defined by
\( \ttt_{1} = \cdots = \ttt_{k} = 0 \) for
\( \BAA^{k}_{\Lambda} = \Spec \Lambda[\ttt_{1}, \ldots, \ttt_{k}] \).
The closed subscheme of \( \BAA^{k}_{T} \) defined by
\( \ttt_{1} = \cdots = \ttt_{k} = 0 \) is isomorphic to \( T \).
Hence a connected component of the pullback of the closed subscheme by
\( S_{T} \to \BAA^{k}_{T} \) gives a desired \'etale neighborhood.
Let \( Z \to T \) be the base change of \( W \to S \) by \( T \to S_{T} \to S \).
Then, \( Z \to T \) is a deformation of \( X \) with reference point \( o \)
satisfying the condition \eqref{thm:DVRsmoothing:1}, and
\( (Z, P) \) and \( (V_{(P)}, P) \) have a common \'etale neighborhood
for any \( P \in \Sing X \).
Hence, the other conditions \eqref{thm:DVRsmoothing:2} and \eqref{thm:DVRsmoothing:3}
are derived from \eqref{item:V2} and \eqref{item:V3},
and we have finished the proof.
\end{proof}

\begin{cor}[Fundamental group]\label{cor:Pi1}
Let \( X_{\Lambda} \to \Spec \Lambda \) be the flat projective morphism in
Theorem~\emph{\ref{thm:DVRsmoothing}} satisfying the two conditions
\eqref{thm:DVRsmoothing:ass1} and \eqref{thm:DVRsmoothing:ass2}.
Assume that the field of fractions of \( \Lambda \) is of characteristic zero.
Let \( X \) be the closed fiber of \( X_{\Lambda} \to \Spec \Lambda \).
Let \( \BKK \) be an algebraically closed field
containing \( \Lambda \) and let \( X_{\BKK} \) be the fiber product
\( X_{\Lambda} \times_{\Spec \Lambda} \Spec \BKK\).
Then, \( \OH^{2}(X_{\BKK}, \SO_{X_{\BKK}}) =
\OH^{2}(X_{\BKK}, \Theta_{X_{\BKK}/\BKK}) = 0 \).
Moreover, there exist a deformation \( \SX \to C \) of \( X \)
over a non-singular algebraic curve \( C \) defined over \( \Bbbk \),
and a deformation \( \SY \to D \) of \( X_{\BKK} \)
over a non-singular algebraic curve \( D \) defined over \( \BKK \)
such that
\( \SX \to C \) and \( \SY \to D \) satisfy
the conditions corresponding to
\eqref{thm:globalsmoothing:1} and \eqref{thm:globalsmoothing:2}
of Theorem~\emph{\ref{thm:globalsmoothing}} and
that there is a surjection
\[ \pi_{1}^{\alg}(Y_{d}) \to \pi_{1}^{\alg}(X_{c}) \]
of algebraic fundamental groups for any
smooth fibers \( X_{c} \) and \( Y_{d} \)
of \( \SX \to C \) and \( \SY \to D \) over
closed points \( c \in C\) and \( d \in D\), respectively.
\end{cor}

\begin{proof}
First, we shall show:
\( \OH^{2}(X_{\BKK}, \SO_{X_{\BKK}}) =
\OH^{2}(X_{\BKK}, \Theta_{X_{\BKK}/\BKK}) = 0 \).
Since \( \OH^{2}(X, \SO_{X}) = 0 \), we have
\( \OH^{2}(X_{\Lambda}, \SO_{X_{\Lambda}}) = 0 \)
by the upper semi-continuity theorem for the flat morphism
\( X_{\Lambda} \to \Spec \Lambda \), and
\( \OH^{2}(X_{\BKK}, \SO_{X_{\BKK}}) = 0 \) by the flat base change isomorphism
\( \OH^{2}(X_{\Lambda}, \SO_{X_{\Lambda}}) \otimes_{\Lambda} \BKK \isom
\OH^{2}(X_{\BKK}, \SO_{X_{\BKK}})\).
The vanishing of \( \OH^{2}(X_{\BKK}, \Theta_{X_{\BKK}/\BKK}) \) is shown as in
the proof of Theorem~\ref{thm:globalsmoothing}.\eqref{thm:globalsmoothing:6}:
For the relative tangent sheaf
\( \Theta_{X_{\Lambda}/\Lambda} :=
\SHom_{\SO_{X_{\Lambda}}}(\Omega^{1}_{X_{\Lambda}/\Lambda}, \SO_{X_{\Lambda}}) \),
the canonical homomorphism
\( \Theta_{X_{\Lambda}/\Lambda} \otimes_{\SO_{X_{\Lambda}}} \SO_{X} \to
\Theta_{X/\Bbbk}\)
is an isomorphism outside \( \Sing X \) but
another canonical homomorphism
\( \Theta_{X_{\Lambda}/\Lambda} \otimes_{\Lambda} \BKK \to \Theta_{X_{\BKK}/\BKK} \)
is an isomorphism. Thus, we have
\( \OH^{2}(X, \Theta_{X_{\Lambda}/\Lambda} \otimes_{\SO_{X_{\Lambda}}} \SO_{X})
= 0\) by \( \OH^{2}(X, \Theta_{X/\Bbbk}) = 0 \).
By the upper semi-continuity theorem applied to
the sheaf \( \Theta_{X_{\Lambda}/\Lambda} \) flat over \( \Lambda \) and by
the base change isomorphism, we have the vanishing
\( \OH^{2}(X_{\BKK}, \Theta_{X_{\BKK}/\BKK}) = 0 \).

Second, we shall define \( \SX \to C \) and \( \SY \to D \).
Let \( Z \to T \) be the deformation of \( X \) obtained in
Theorem~\ref{thm:DVRsmoothing}. By the surjection \( \Lambda \surjmap \Bbbk \) and
the injection \( \Lambda \injmap \BKK \), we define
\begin{align*}
\SX := Z \times_{\Spec \Lambda} \Spec \Bbbk &\to C := T \times_{\Spec \Lambda} \Spec \Bbbk,
\quad \text{ and }\\
\SY := Z \times_{\Spec \Lambda} \Spec \BKK &\to D := C \times_{\Spec \Lambda} \Spec \BKK.
\end{align*}
Then, \( \SX \to C \) is a deformation of \( X \) with reference point
\( o = C \cap \sigma(\Spec \Lambda) \), and
\( \SY \to D \) is a deformation of \( X_{\BKK} \) with the reference point
\( o_{\BKK} := D \times_{Z} \sigma(\Spec \Lambda) \).
Moreover, these deformations satisfy
the conditions corresponding to \eqref{thm:globalsmoothing:1} and
\eqref{thm:globalsmoothing:2} of Theorem~\ref{thm:globalsmoothing}.

Finally, we shall compare several algebraic fundamental groups
using results in \cite[Exp.~X]{SGA1} concerning with
Grothendieck's specialization theorem
\cite[Exp.~X, Corollaire~2.4, Th\'eor\`eme~3.8]{SGA1}.
Let \( \Bbbk_{1} \) be the algebraic closure of
the function field \( \Bbbk(C) \)
of \( C \) and set \( \SX_{\Bbbk_{1}} \) to be
the fiber product \( \SX \times_{C} \Spec \Bbbk_{1}  \).
Then, for any smooth closed fiber \( X_{c} \) of \( \SX \to C \),
we have a surjection
\begin{equation}\label{eq:surjpi1:1}
\pi_{1}^{\alg}(\SX_{\Bbbk_{1}}) \surjmap \pi_{1}^{\alg}(X_{c})
\end{equation}
by \cite[Exp.~X, Corollaire~2.4, Th\'eor\`eme~3.8]{SGA1}.

Let \( \BKK_{1} \) be an algebraically closed field
containing the function field \( \BKK(D) \) of \( D \),
and let \( \SY_{\BKK_{1}} \) be
the fiber product \( \SY \times_{D} \Spec \BKK_{1} \).
The geometric generic points
\( \Spec \BKK_{1} \to D \to T\) and
\( \Spec \Bbbk_{1} \to C \to T \)
are lying on the open subset \( T \setminus \sigma(\Spec \Lambda) \),
and the corresponding geometric fibers of \( Z \to T \) are
\( \SY_{\BKK_{1}} \) and \( \SX_{\Bbbk_{1}} \), respectively.
By \cite[Exp.~X, Corollaire~2.4, Th\'eor\`eme~3.8]{SGA1},
we have a surjection
\begin{equation}\label{eq:surjpi1:2}
\pi_{1}^{\alg}(\SY_{\BKK_{1}}) \surjmap \pi_{1}^{\alg}(\SX_{\Bbbk_{1}})
\end{equation}
Since \( \chara(\BKK) = 0 \), by \cite[Exp.~X, Th\'eorem\`e~3.8, Corollaire~3.9]{SGA1},
we have an isomorphism
\begin{equation}\label{eq:isompi1:3}
\pi_{1}^{\alg}(\SY_{\BKK_{1}}) \isom \pi_{1}^{\alg}(Y_{d})
\end{equation}
for any smooth closed fiber \( Y_{d} \) of \( \SY \to D \).
Thus, we have a desired surjection
\( \pi_{1}^{\alg}(Y_{d}) \surjmap \pi_{1}^{\alg}(X_{c})  \)
from \eqref{eq:surjpi1:1}--\eqref{eq:isompi1:3}.
\end{proof}

\begin{remsub}\label{remsub:cor:Pi1}
The calculation of \( \pi_{1}^{\alg}(Y_{d}) \) in Corollary~\ref{cor:Pi1}
is reduced to the case over the complex number field \( \BCC \), as follows.
Let \( \BKK \), \( X_{\BKK} \), and \( \SY \to D \) be as in Corollary~\ref{cor:Pi1}.
Here, \( \SY \to D \) is a deformation of \( X_{\BKK} \) with
a reference \( \BKK \)-rational point \( b := o_{\BKK} \in D \) as in the proof of
Corollary~\ref{cor:Pi1}.
Then, there is a finitely generated field \( \BKK_{0} \) over the field
\( \BQQ \) of rational numbers such that \( X_{\BKK} \), \( \SY \to D \),
\( o_{\BKK} \in D\),
\( d \in D \setminus \{b\}\),
and every point of \( \Sing X_{\BKK} = \{P_{1}, \ldots, P_{k}\}\)
descend to over \( \BKK_{0} \). Namely, there exist
algebraic schemes \( X_{0} \) and \( D_{0} \) over \( \Spec \BKK_{0} \),
\( \BKK_{0} \)-rational points \( b_{0} \) and \( d_{0} \) of \( D_{0} \),
\( \BKK \)-rational points \( P_{1, 0} \), \ldots, \( P_{k, 0} \) of \( X_{0} \),
and a morphism \( \SY_{0} \to D_{0} \) such that
\begin{gather*}
X_{\BKK} \isom X_{0} \times_{\Spec \BKK_{0}} \Spec \BKK, \quad
D \isom D_{0} \times_{\Spec \BKK_{0}} \Spec \BKK, \\
(b \in D) \isom (b_{0} \in D_{0}) \times_{\Spec \BKK_{0}} \Spec \BKK, \quad
(d \in D) \isom (d_{0} \in D_{0}) \times_{\Spec \BKK_{0}} \Spec \BKK, \\
(P_{i} \in X) \isom (P_{i, 0} \in X_{0}) \times_{\Spec \BKK_{0}} \Spec \BKK
\quad (1 \leq i \leq k), \\
(\SY \to D) \isom (\SY_{0} \to D_{0}) \times_{\Spec \BKK_{0}} \Spec \BKK,
\end{gather*}
with the following properties (cf.\ \cite[IV, Propositions~(2.5.1), (2.7.1), (6.7.4),
Corollaire~(2.7.2)]{EGA}):
\begin{itemize}
\item  \( X_{0} \) is a normal projective integral \( \BKK_{0} \)-scheme,
and \( X_{0} \setminus \{P_{i, 0}\}_{1 \leq i \leq k} \) is smooth over \( \Spec \BKK_{0} \).

\item \( D_{0} \) is a smooth algebraic \( \BKK_{0} \)-scheme.

\item \( \SY_{0} \) is normal and integral.

\item  \( \SY_{0} \to D_{0} \) is a projective flat morphism
whose fiber over \( b_{0} \) is identified with \( X_{0} \).

\item \( \SY_{0} \to D_{0} \) is smooth on
\( \SY_{0}^{\circ} := \SY_{0} \setminus \{P_{i, 0}\}_{1 \leq i \leq k} \).
\end{itemize}
Moreover, \( rK_{\SY_{0}/D_{0}} \) is Cartier
for the relative canonical divisor \( K_{\SY_{0}/D_{0}} \) and
for the index \( r \) of \( X \), and there is an isomorphism
\[ \SO_{\SY_{0}}(rK_{\SY_{0}/D_{0}}) \isom
j_{*}(\omega^{\otimes r}_{\SY_{0}^{\circ}/D_{0}}) \]
for the relative dualizing sheaf \( \omega_{\SY_{0}^{\circ}/D_{0}} \) where
\( j \colon \SY^{\circ}_{0} \injmap \SY_{0} \) denotes the open immersion.
These properties are also derived from the corresponding properties on \( \SY \to D \).
Note that the fiber \( Y_{d} \) of \( \SY \to D \) over \( d \)
is the base change of the fiber \( Y_{d_{0}}\) of
\( \SY_{0} \to D_{0} \) over \( d_{0} \) by
\( \Spec \BKK \to \Spec \BKK_{0}\).
Taking a field extension \( \BKK_{0} \subset \BCC \), we set
\begin{gather*}
X_{\BCC} := X_{0} \times_{\Spec \BKK_{0}} \Spec \BCC, \quad
D_{\BCC} := D_{0} \times_{\Spec \BKK_{0}} \Spec \BCC, \\
(b_{\BCC} \in D_{\BCC}) := (b_{0} \in D_{0}) \times_{\Spec \BKK_{0}} \Spec \BCC, \quad
(d_{\BCC} \in D_{\BCC}) := (d_{0} \in D_{0}) \times_{\Spec \BKK_{0}} \Spec \BCC, \\
(P_{i, \BCC} \in D_{\BCC}) := (P_{i, 0} \in D_{0}) \times_{\Spec \BKK_{0}} \Spec \BCC,
\quad (1 \leq i \leq k), \\
(\SY_{\BCC} \to D_{\BCC}) := (\SY_{0} \to D_{0}) \times_{\Spec \BKK_{0}} \Spec \BCC.
\end{gather*}
By considering also the descent of the minimal resolution of singularities
of \( X \),
we may assume that \( X_{\BCC} \) has only toric singularities of class T.
In fact, the exceptional locus over \( P_{i, \BCC} \) of the minimal resolution
of \( X_{\BCC} \)
is the linear chain of rational curves with the same self-intersection numbers
as that for \( P_{i} \); thus
\( (X_{\BCC}, P_{i, \BCC}) \) is a toric singularity by
Theorem~\ref{thm:graph2toric000}.
Then, \( \SY_{\BCC} \to D_{\BCC} \) is a \( \BQQ \)-Gorenstein smoothing of
\( X_{\BCC} \) in the sense of \cite[Section~3]{KSh} (cf.\ \cite[Corollary~3.6]{KSh}),
since \( \SY_{\BCC} \) is \( \BQQ \)-Gorenstein.
For the smooth fiber \( Y_{\BCC, d_{\BCC}} \) of \( \SY_{\BCC} \to D_{\BCC} \)
over \( d_{\BCC} \), we have an isomorphism
\[ \pi_{1}^{\alg}(Y_{d})  \isom \pi_{1}^{\alg}(Y_{\BCC, d_{\BCC}})\]
by \cite[Exp.~X, Corollaire~1.8]{SGA1}.
\end{remsub}


\section{Simply connected surfaces of general type with \texorpdfstring{$p_{g} = q = 0$}{pg = q = 0}}
\label{sect:Global}

We apply the results in
Sections~\ref{sect:dualgraph}--\ref{sect:DeformationClassT} to
construct algebraically simply connected surfaces \( \BSS\) of general type
with \(p_g=q=0\) and \( 1 \leq K^2 \leq 4\) which is defined over
the given algebraically closed field \( \Bbbk \), where
\[ p_{g} = p_{g}(\BSS) = \dim \OH^{0}(\BSS, \SO_{\BSS}(K_{\BSS}))
= \dim \OH^{2}(\BSS, \SO_{\BSS}), \qquad
q = q(\BSS) = \dim \OH^{1}(\BSS, \SO_{\BSS}),\]
and \( K = K_{\BSS} \) denotes the canonical divisor.
An outline of our method is as follows.
We first construct a normal projective rational surface \( X \)
with only toric singularities of class T
satisfying the following conditions:
\begin{itemize}
\item
\( K_{X} \) is nef and big (or ample),

\item
\( K_{X}^{2} \) equals the given number \( K^{2} > 0 \),

\item
\( \OH^{2}(X, \Theta_{X}) = 0 \).
\end{itemize}
Note that \( \OH^{i}(X, \SO_{X}) = 0 \) for any \( i > 0 \)
since \( X \) is rational and has only rational singularities.
For the construction of \( X \),
we follow the method used in \cite{LeePark}, \cite{PPS1}, and
\cite{PPS2}, which is however considered over the field \( \BCC \)
of complex numbers.
By Theorem~\ref{thm:globalsmoothing},
we have a projective deformation \( \SX \to T \) of \( X \)
over a non-singular curve \( T \)
defined over \( \Bbbk \) satisfying the conditions
\eqref{thm:globalsmoothing:1} and \eqref{thm:globalsmoothing:2} of
Theorem~\ref{thm:globalsmoothing}.
Here, a general closed smooth fiber \( X_{t} \) of \( \SX \to T \)
is a non-singular projective surface \( \BSS \) having
the following properties:

\begin{itemize}
\item
\( q(\BSS) = p_{g}(\BSS) = 0 \) and \( \chi(\BSS, \SO_{\BSS}) = 1 \).

\item
\( \BSS \) is a minimal surface of general type with
\( K_{\BSS}^{2} = K^{2} \).

\item
\( \OH^{2}(\BSS, \Theta_{\BSS/\Bbbk}) = 0 \). In particular,
\( \BSS \) is liftable to characteristic zero.
\end{itemize}

For the second condition above, note that
\( K_{X_{t}} \) is ample if \( K_{X} \) is.
If \( X_{t} \) is algebraically simply connected, then
this is one of the surfaces what we want to get.
In order to construct a simply connected one, we select the deformation
\( \SX \to T \) by considering
a lifting problem to characteristic zero.
Namely, we construct \( X \) as the closed fiber of
a flat family \( X_{\Lambda} \to \Spec \Lambda \) for
a complete discrete valuation ring \( \Lambda \) of mixed characteristic
with the residue field \( \Bbbk \) satisfying the assumptions
\eqref{thm:DVRsmoothing:1} and
\eqref{thm:DVRsmoothing:2} of
Theorem~\ref{thm:DVRsmoothing}.
Then, by Corollary~\ref{cor:Pi1},
we have a deformation \( \SX \to T \)
satisfying the conditions
\eqref{thm:globalsmoothing:1} and \eqref{thm:globalsmoothing:2} of
Theorem~\ref{thm:globalsmoothing} in which \( \pi_{1}^{\alg}(X_{t}) \)
is dominated by \( \pi_{1}^{\alg} \) of a \( \BQQ \)-Gorenstein smoothing
of a geometric generic fiber of \( X_{\Lambda} \to \Spec \Lambda \).
Looking at the construction of \( X_{\Lambda} \),
we shall prove the simply connectedness of the \( \BQQ \)-Gorenstein smoothing
of a geometric generic fiber from the argument used in the proof of
\cite[Theorem~3.1]{LeePark}, \cite[Theorem~3.1]{PPS1}, and
\cite[Proposition~2.1]{PPS2}. The argument shows especially
that, when \( \Bbbk = \BCC \),
\( X \setminus \Sing X \) is topologically simply connected.
In this way, the new \( X_{t} \) is shown to be algebraically simply connected,
and we have a desired surface.

\begin{rem}\label{rem:K2bound}
In the construction above,
the number \( K^{2} \) must be between \( 1 \) and \( 4 \).
In fact, we may assume that \( \chara(\Bbbk) = 0 \), and
in this case, we have \( \OH^{0}(X_{t}, \Theta_{X_{t}/\Bbbk}) = 0 \),
since the automorphism group of the surface \( X_{t} \) of general type is finite.
Thus, by Riemann--Roch,
\[ \dim \OH^{1}(X_{t}, \Theta_{X_{t}/\Bbbk}) = -\chi(X_{t}, \Theta_{X_{t}/\Bbbk}) =
10 - 2K^{2},\]
where we use \( \chi(X_{t}, \SO_{X_{t}}) = 1 \), \( K_{X_{t}}^{2} = K^{2} \), and
\(\OH^{2}(X_{t}, \Theta_{X_{t}/\Bbbk}) = 0\).
Hence, \( 1 \leq K^{2} \leq 5 \).
Moreover, our \( X_{t} \) is a fiber of a \( \BQQ \)-Gorenstein smoothing of
a rational surface with only cyclic quotient singularities of class T.
Hence, \( X_{t} \) has a non-trivial deformation by
the existence of the coarse moduli of surfaces of general type
(cf.\ \cite[Theorem~1.3]{Gieseker}, \cite[Corollary~5.7]{KSh}).
Thus, \( \dim \OH^{1}(X_{t}, \Theta_{X_{t}/\Bbbk}) > 0 \) and \( K^{2} \leq 4 \).
\end{rem}

We shall explain the construction of
\( X \) and \( X_{\Lambda} \to \Spec \Lambda \)
from suitable cubic pencils on \( \BPP^{2} \)
step by step in Section~\ref{sect:Global}.
We also give sufficient conditions for the surface \( X \)
to be a desired one.
Explicit examples of the cubic pencils are given in
Section~\ref{sect:proof}, and
the Main Theorem is proved using these examples.

Let us fix a complete discrete valuation ring \( \Lambda \)
of mixed characteristic with the residue field \( \Bbbk \).
Let \( K^{2} \) be a given positive integer.

\begin{step}\label{step:pencil}
We first take two cubic homogeneous
polynomials \( \phi_{0} \), \( \phi_{\infty} \)
from \( \BZZ[\xtt, \ytt, \ztt] \), and let
\( \Phi_{0} \) and \( \Phi_{\infty} \) be the divisors
of zeros of \( \phi_{0} \) and \( \phi_{\infty} \), respectively,
on \( \BPP^{2}_{\Bbbk} = \Proj \Bbbk[\xtt, \ytt, \ztt]\).
Let \( \Phi \) be the cubic pencil defined by \( \Phi_{0} \) and
\( \Phi_{\infty} \).
For \( c \in \Bbbk \), let \( \Phi_{c} \) be the divisor of zeros
of \( \phi_{0} + c\phi_{\infty} \).
We here require the following conditions on \( \Phi \):
\begin{enumerate}
    \renewcommand{\labelenumi}{($\SC$\theenumi)}
    \makeatletter
    \renewcommand{\p@enumi}{$\SC$}
    \makeatother
\item \label{condC:Phi}
The base locus of \( \Phi \), i.e.
\( \Phi_{0} \cap \Phi_{\infty} \), is a finite set.

\item  \label{condC:nodal}
There exist at least two values of \( c \in \Bbbk \setminus \{0\}\)
such that \( \Phi_{c} \) is singular. Moreover, if \(\Phi_{c}\) is singular for
\( c \ne 0 \), then it is a nodal rational curve.
\end{enumerate}

Let \( Y \to \BPP^{2}_{\Bbbk} \) be the elimination
of the base locus of \( \Phi \) which is a succession of blowings up
at points. Then, we have a minimal elliptic fibration
\( \pi \colon Y \to \BPP^{1}_{\Bbbk} \) with a section
such that \( \SO_{Y}(-K_{Y}) \isom \pi^{*}\SO(1) \).
\end{step}

\begin{lem}\label{lem:piLambda}
In the situation of Step~\emph{\ref{step:pencil}},
there is a flat morphism
\( \pi_{\Lambda} \colon Y_{\Lambda} \to \BPP^{1}_{\Lambda} \)
such that \( \pi \) is the base change of \( \pi_{\Lambda} \) by
the closed immersion \( \BPP^{1}_{\Bbbk} \to \BPP^{1}_{\Lambda} \).
Let \( c_{1} \), \( c_{2} \), \ldots, \( c_{n} \) be
the values \( c \in \Bbbk\) such that \( \Phi_{c} \) is singular.
Then, \( c_{i} \) are regarded as
elements of \( \Lambda \), and hence
\( \pi_{\Lambda} \) is a smooth elliptic fibration outside
the sections of \( \BPP^{1}_{\Lambda} \to \Spec \Lambda \)
defined by \( (\utt : \vtt) = (1 : c_{i}) \) for \( 1 \leq i \leq n \)
and \( (\utt : \vtt) = (0:1) \),
where  \( (\utt:\vtt) \) is a homogeneous coordinate of \( \BPP^{1} \).
\end{lem}

\begin{proof}
By construction, \( \Phi_{0} \) and \( \Phi_{\infty} \) are defined over \( \BZZ \),
and the each center of the blowing up between \( Y \)
and \( \BPP^{2}_{\Bbbk} \) is a point whose coordinates are algebraic
over \( \BZZ \).
Thus, the point is defined over the Henselian local ring \( \Lambda \).
Hence, \( Y \to \BPP^{2}_{\Bbbk} \) extends to
a birational morphism \( Y_{\Lambda} \to \BPP^{2}_{\Lambda}\) which is
a succession of blowings up along centers over \( \Spec \Lambda \).
Moreover, by the pencil over \( \Lambda \), we have
an elliptic fibration
\( \pi_{\Lambda} \colon Y_{\Lambda} \to \BPP^{1}_{\Lambda} \)
extending \( \pi \).
The elements \( c_{1} \), \ldots, \( c_{n} \) can be regarded as
elements of \( \Lambda \), since \( \Lambda \) is Henselian, and
\( \pi_{\Lambda} \) is smooth outside the sections above.
\end{proof}

\begin{step}\label{step:Y}
For the elliptic fibration \( \pi \) in
\emph{Step}~\ref{step:pencil}, after choosing two values \( c_{1} \),
\(c_{2} \in \Bbbk \setminus \{0\} \) such that \( \Phi_{c_{1}} \) and \(
\Phi_{c_{2}} \) are singular, we define \( \bar{F}_{1} \) and \(
\bar{F}_{2} \) to be the proper transforms of \( \Phi_{c_{1}} \) and
\( \Phi_{c_{2}} \) in \( Y \), respectively. Note that, for \( i = 1
\), \( 2 \), \( \bar{F}_{i} \) is the fiber over the point \(
(1:c_{i}) \in \BPP^{1}_{\Bbbk} \), and it is a singular fiber of
type \( \text{I}_{1} \) in Kodaira's notation (cf.\
\cite[Theorem~6.2]{Kodaira}). We define \( \bar{F} := \bar{F}_{1} +
\bar{F}_{2} \). Let \( \bar{G}^{+} \) be the union of all the \(
(-2) \)-curves on \( Y \). Note that every \( (-2) \)-curve on \( Y \)
is an irreducible component of a reducible fiber by \(
\SO_{Y}(-K_{Y}) \isom \pi^{*}\SO(1) \), and thus it is an
irreducible component of the total transform of \( \Phi_{0} +
\Phi_{\infty} \) in \( Y \) (cf.\ \eqref{condC:nodal}).
We choose \( (-2) \)-curves \(
\bar{G}_{1} \), \( \bar{G}_{2} \), \ldots, \( \bar{G}_{k} \) on \( Y
\) and set \( \bar{G} := \sum\nolimits_{i = 1}^{k} \bar{G}_{i} \).
Moreover, we choose horizontal curves \( \bar{S}_{1} \), \(
\bar{S}_{2} \), \ldots, \( \bar{S}_{l} \) with respect to \( \pi \)
such that each \( \bar{S}_{j} \) is either exceptional for \( Y \to
\BPP_{\Bbbk}^{2} \) or its image in \( \BPP^{2}_{\Bbbk} \) is a line
defined over the ring of algebraic integers. We set \( \bar{S} :=
\sum\nolimits_{j = 1}^{l}\bar{S}_{j} \), and also
\[ \bar{B} := \bar{F} + \bar{G} + \bar{S} \quad \text{and} \quad
\bar{B}^{+} := \bar{F} + \bar{G}^{+} + \bar{S}.\]
Note that, by construction and by Lemma~\ref{lem:piLambda},
any irreducible component of \( \bar{B}^{+} \) is realized as
the closed fiber of
a prime divisor of \( Y_{\Lambda} \) over \( \Spec \Lambda \).
We require the following conditions on \( Y \), \( \bar{B}^{+} \), and
\( \bar{G} \):
\begin{enumerate}
    \addtocounter{enumi}{2}
    \renewcommand{\labelenumi}{($\SC$\theenumi)}
    \makeatletter
    \renewcommand{\p@enumi}{$\SC$}
    \makeatother
\item \label{condC:NC1}
\( \bar{B}^{+} \) is a simple normal crossing divisor outside the nodes of
\( \bar{F} \).

\item  \label{condC:LinIndep}
\( \SO_{Y}(\bar{G}_{1}) \), \ldots, \( \SO_{Y}(\bar{G}_{k}) \), and
\( \pi^{*}\SO(1) \) are linearly independent in \( \Pic(Y)
\otimes_{\BZZ} \Bbbk \).
\end{enumerate}
We do not require the following additional conditions \eqref{condA:1}
and \eqref{condA:2} which are however useful to check the ampleness of \( K_{X} \)
(cf.\ Proposition~\ref{prop:step:M}.\eqref{prop:step:M:3} below):

\begin{enumerate}
    \renewcommand{\labelenumi}{($\SA$\theenumi)}
    \makeatletter
    \renewcommand{\p@enumi}{$\SA$}
    \makeatother

\item  \label{condA:1}
The open set \( Y \setminus (\bar{S} \cup \bar{G}) \) is affine.
\item \label{condA:2}
The open set \( Y \setminus (\bar{S} \cup \bar{G})  \) does not contain
any \( (-2) \)-curve.
\end{enumerate}
\end{step}

\begin{lem}\label{lem:determinant}
In the situation of Step~\emph{\ref{step:Y}},
the condition \eqref{condC:LinIndep} is
satisfied if
\[ \det (\bar{G}_{i}\bar{G}_{j}) \not\equiv 0 \mod \chara(\Bbbk). \]
\end{lem}

\begin{proof}
For \( p := \chara(\Bbbk) \), suppose that
\[ \sum\nolimits_{i = 1}^{k} c_{i}\bar{G}_{i} + c \bar{F}_1 \sim pH \]
for some integers \( c_{i} \) and \( c \), and for the fiber
\( \bar{F}_1 \) of \( \pi \) and a Cartier divisor \( H \). Then, \(
c_{i} \equiv 0 \bmod p \) for any \( 1 \leq i \leq k\) by \(
\det(\bar{G}_{i}\bar{G}_{j}) \not\equiv 0 \bmod p \) and by \(
\bar{F}_1 \bar{G}_{i} = 0 \). Moreover, \( c \equiv 0 \bmod p \),
since \( \bar{F}_1\Sigma = 1 \) for a section \( \Sigma \) of \( \pi
\). Thus, \eqref{condC:LinIndep} is satisfied.
\end{proof}

\begin{lem}\label{lem:step:Y}
In the situation of Step~\emph{\ref{step:Y}},
\[ \OH^{0}(Y, \Omega^{1}_{Y/\Bbbk}(\log (\bar{G} + \bar{F}_{2})) = 0. \]
\end{lem}

\begin{proof}
Since \( \bar{G} + \bar{F}_{2}\) is normal crossing (cf.\ \eqref{condC:NC1}),
we can consider the exact sequence
\begin{equation}\label{eq:shortexactlog G}
0 \to \Omega^{1}_{Y/\Bbbk} \to \Omega^{1}_{Y/\Bbbk}(\log (\bar{G} + \bar{F}_{2}))
\to \bigoplus\nolimits_{i = 1}^{k} \SO_{\bar{G}_{i}} \oplus
\SO_{F_{2}} \to 0,
\end{equation}
where \( F_{2} \) is the normalization of \( \bar{F}_{2} \).
Let \( f_{2} \in \OH^{1}(Y, \Omega^{1}_{Y/\Bbbk})\) be
the image of \( 1 \in \OH^{0}(\SO_{F_{2}}) \)
by the connecting homomorphism
of the long exact sequence of cohomology groups
associated with \eqref{eq:shortexactlog G}.
Similarly, let \( g_{i} \in \OH^{1}(Y, \Omega^{1}_{Y/\Bbbk})\) be the image of
\( 1 \in \OH^{0}(\SO_{G_{i}}) \) for \( 1 \leq i \leq k \).
Then, \( f_{2} \) and \( g_{i} \) can be regarded as
the first Chern classes (with respect to ``\( \dd\!\log \)'')
of the invertible sheaves
\( \SO_{Y}(\bar{F}_{2}) \) and \( \SO_{Y}(\bar{G}_{i}) \), respectively.
In fact,
we have a commutative diagram
\[ \begin{CD}
1 @>>> \SO_{Y}^{\times} @>>> \SO_{Y}(* (\bar{G} + \bar{F}_{2}))^{\times} @>>>
\SH^{0}_{\bar{G} + \bar{F}_{2}}(\mathcal{D}\textit{iv}_{Y}) @>>> 0 \\
@. @V{\dd\!\log}VV @V{\dd\!\log}VV @VVV \\
0 @>>> \Omega^{1}_{Y/\Bbbk} @>>>
\Omega^{1}_{Y/\Bbbk}(\log (\bar{G} + \bar{F}_{2})) @>>>
\bigoplus\nolimits_{i = 1}^{k} \SO_{\bar{G}_{i}} \oplus \SO_{F_{2}} @>>> 0
\end{CD}\]
of exact sequence, where
\( \dd\!\log (u) = u^{-1}\dd u\) for a rational function \( u \), and
for an effective divisor \( R \),
\begin{itemize}
\item  \( \SO_{Y}(*R)^{\times} \) stands for the sheaf of invertible rational functions
on \( Y \) regular on the open subset
\( Y \setminus \Supp R \), and

\item \( \SH^{0}_{R}(\mathcal{D}\textit{iv}_{Y}) \) stands for the sheaf of Cartier
divisors on \( Y \) supported on \( \Supp R \).
\end{itemize}
Note that, as global sections of
\( \SH^{0}_{\bar{G} + \bar{F}_{2}}(\mathcal{D}\textit{iv}_{Y}) \),
\( \bar{F}_{2} \) and
\( \bar{G}_{i} \) are mapped to \( 1 \in \SO_{F_{2}} \) and
\( 1 \in \SO_{\bar{G}_{i}} \), respectively, by the right vertical homomorphism.
The first Chern class map
\( c \colon \Pic(Y) = \OH^{1}(Y, \SO_{Y}^{\times})
\to \OH^{1}(Y, \Omega^{1}_{Y/\Bbbk}) \) is
induced by the left vertical homomorphism: \( \dd\!\log \). Therefore,
\( f_{2} = c(\SO_{Y}(\bar{F}_{2})) \) and \( g_{i} = c(\SO_{Y}(\bar{G}_{i})) \).
For the natural bilinear form
\[ \langle \phantom{x}, \phantom{y} \rangle
\colon \OH^{1}(Y, \Omega^{1}_{Y/\Bbbk}) \times
\OH^{1}(Y, \Omega^{1}_{Y/\Bbbk})  \to
\OH^{2}(Y, \Omega^{2}_{Y/\Bbbk}) \isom \Bbbk, \]
the value \( \langle c(\SL_{1}), c(\SL_{2}) \rangle \) in \( \Bbbk \)
equals the intersection number \( \SL_{1}\cdot\SL_{2} \)
modulo \( \chara(\Bbbk) \)
for any invertible sheaves \( \SL_{1} \), \( \SL_{2} \) on \( Y \)
(cf.\ \cite[Chapter.~V, Exer.~1.8]{HartshorneGIT}).
Now, \( \Pic(Y) \) is a unimodular lattice for the intersection pairing,
since \( Y \) is rational. As a consequence,
\( c \) induces an injection
\( \Pic(Y) \otimes_{\BZZ} \Bbbk \to \OH^{1}(Y, \Omega^{1}_{Y/\Bbbk}) \).
Hence, \( g_{1}\), \ldots, \(g_{k}\), and \( f_{2} \)
are linearly independent by \eqref{condC:LinIndep}. Therefore,
the connecting homomorphism for \eqref{eq:shortexactlog G} is injective, and
we have
\[ \OH^{0}(Y, \Omega^{1}_{Y/\Bbbk}(\log (\bar{G} + \bar{F}_{2}))) =
\OH^{0}(Y, \Omega^{1}_{Y/\Bbbk}) = 0. \qedhere\]
\end{proof}

\begin{step}\label{step:Z}
We consider the blowing up \( \tau \colon Z \to Y \)
at the two nodes \( P_{1} \) and \( P_{2} \)
of \( \bar{F}_{1} \) and \( \bar{F}_{2} \), respectively.
Let \( F_{i} \) be the proper transform of \( \bar{F}_{i} \) in \( Z \)
and let \( J_{i} \) be the exceptional divisor \( \tau^{-1}(P_{i}) \)
for \( i = 1 \), \( 2 \).
Then, \( F_{i} \to \bar{F}_{i} \) is the normalization map, and
\( \tau^{*}(\bar{F}_{i}) = F_{i} + 2J_{i}\) for \( i = 1 \), \( 2 \).
In particular,
\begin{equation}\label{eq:-KZ}
K_{Z} \sim -F_{1} - J_{1} + J_{2}, \quad \text{and} \quad
-2K_{Z} \sim F_{1} + F_{2}.
\end{equation}
Let \( G_{i} \) and \( S_{j} \) denote the proper transforms of
\( \bar{G}_{i} \) and \( \bar{S}_{j} \) in \( Z \)
for \( 1 \leq i \leq k \) and \( 1 \leq j \leq l \).
Here, \( G_{i} \) is the total transform of \( \bar{G}_{i} \),
since \( \bar{G}_{i} \cap \bar{F} = \emptyset \).
We set \( F := F_{1} + F_{2} \), \( G := \sum\nolimits_{i = 1}^{k} G_{i} \), and
\( S := \sum\nolimits_{j = 1}^{l} S_{j} \).
We require the following conditions:

\begin{enumerate}
    \addtocounter{enumi}{4}
    \renewcommand{\labelenumi}{($\SC$\theenumi)}
    \makeatletter
    \renewcommand{\p@enumi}{$\SC$}
    \makeatother
\item \label{condC:NC2}
\( S + F + J_{1} + J_{2}\) is a simple normal crossing divisor.
\item  \label{condC:-1}
\( S_{j} \) are \( (-1) \)-curves for any \( 1 \leq j \leq l \).
\end{enumerate}
Let \( B^{+} \) be the total transform of \( \bar{B}^{+} \) in \( Z \).
Then, \( B^{+} \) is also
a simple normal crossing divisor by \eqref{condC:NC1} and \eqref{condC:NC2}.
We define \( B \) to be a reduced divisor such that
\( S + G + F \leq B \leq S + G + F + J_{1} + J_{2} = B^{+}\).
\end{step}

\begin{lem}\label{lem:step:Z}
In the situation of Step~\emph{\ref{step:Z}},
\[ \OH^{2}(Z, \Theta_{Z/\Bbbk}(-\log (G + F))) =
\OH^{2}(Z, \Theta_{Z/\Bbbk}(-\log B)) = 0.\]
\end{lem}

\begin{proof}
The second vanishing is derived from the first, since
we have an exact sequence
\[ 0 \to \Theta_{Z/\Bbbk}(-\log B) \to
\Theta_{Z/\Bbbk}(-\log (G + F)) \to
\bigoplus\nolimits_{j = 1}^{l}\SO_{S_{j}}(S_{j})
\oplus \bigoplus\nolimits_{J_{i} \subset B} \SO_{J_{i}}(J_{i}) \to 0\]
for the simple normal crossing divisors \( B \) and \( G + F \),
in which \( \OH^{1}(\SO_{S_{j}}(S_{j})) = \OH^{1}(\SO_{J_{i}}(J_{i}))
\linebreak 
= \OH^{1}(\BPP^{1}, \SO(-1)) = 0 \) by \eqref{condC:-1}.
The first vanishing is equivalent to
\[ \OH^{0}(Z, \Omega^{1}_{Z/\Bbbk}(\log (G + F)) \otimes \SO_{Z}(K_{Z})) = 0 \]
by Serre duality.
By \eqref{eq:-KZ}, we have an inclusion
\begin{align*}
\tau_{*}\left( \Omega^{1}_{Z/\Bbbk}(\log (G + F)) \otimes \SO_{Z}(K_{Z}) \right)
&\injmap
\tau_{*}\left( \Omega^{1}_{Z/\Bbbk}(\log (G + F_{2})) \otimes \SO_{Z}(J_{2})
\right) \\
&\injmap \Omega^{1}_{Y/\Bbbk}(\log (\bar{G} + \bar{F}_{2})),
\end{align*}
where the last map is the injection to the double-dual.
Thus, we are done by Lemma~\ref{lem:step:Y}.
\end{proof}

\begin{step}\label{step:M}
We take a successive blowings up \( \varphi \colon M \to Z \)
whose centers are certain nodes of the total transform of \( B^{+} \).
On the choice of nodes,
we assume that the total transform \( B_{M} = \varphi^{-1}(B)\) of \( B \)
in \( M \) contains a disjoint union \( D = \bigcup\nolimits_{i = 1}^{m} D^{(i)} \) of
linear chains \( D^{(1)} \), \ldots, \( D^{(m)} \)
of smooth rational curves satisfying the following conditions
\eqref{condC:DB}--\eqref{condC:end}:
\begin{enumerate}
    \addtocounter{enumi}{6}
    \renewcommand{\labelenumi}{($\SC$\theenumi)}
    \makeatletter
    \renewcommand{\p@enumi}{$\SC$}
    \makeatother

\item \label{condC:DB}
\( \varphi(D) = B \).
\item \label{condC:classT}
Each \( D^{(i)} \) contains no \( (-1) \)-curve, and
is contractible to a toric singularity of class T.
\end{enumerate}

Let \( \mu \colon M \to X \) be the contraction morphism of \( D \).
Then, \( X \) has only toric singularities of class T and
\( \mu \) is the minimal resolution of singularities.
Let \( \Delta \) be the \( \BQQ \)-divisor supported on \( D \) define by
\( K_{M} + \Delta = \mu^{*}(K_{X}) \).

\begin{enumerate}
    \addtocounter{enumi}{8}
    \renewcommand{\labelenumi}{($\SC$\theenumi)}
    \makeatletter
    \renewcommand{\p@enumi}{$\SC$}
    \makeatother

\item \label{condC:K2}
\( K_{X}^{2} = K^{2} \).

\item \label{condC:1+}
\( \Delta \Gamma > 1 \) for any \( \varphi \)-exceptional
\( (-1) \)-curve \( \Gamma \) on \( M \).

\item  \label{condC:end}
For any \( D^{(i)} \), there exist smooth rational
curves \( \Gamma \) and \( \Gamma' \) not contained in \( D^{(i)} \)
such that
\begin{itemize}
\item  \( \Gamma + D^{(i)} + \Gamma' \) is a linear chain of smooth
rational curves with the end components \( \Gamma \) and \( \Gamma' \),

\item \( \varphi( \Gamma \cup \Gamma') \)
is contained in \( B^{+} \).
\end{itemize}

\end{enumerate}

We do not require the following condition \eqref{condA:3}
which is however useful for checking the ampleness of \( K_{X} \)
(cf.\ Proposition~\ref{prop:step:M}.\eqref{prop:step:M:3} below):

\begin{enumerate}
        \addtocounter{enumi}{2}
    \renewcommand{\labelenumi}{($\SA$\theenumi)}
    \makeatletter
    \renewcommand{\p@enumi}{$\SA$}
    \makeatother

\item  \label{condA:3}
There is no \( \varphi \)-exceptional \( (-2) \)-curve
contained in \( M \setminus D \).
\end{enumerate}

\end{step}

\begin{prop}\label{prop:step:M}
In the situation of Step~\emph{\ref{step:M}}, the following hold\emph{:}
\begin{enumerate}

\item \label{prop:step:M:0}
\( \OH^{1}(X, \SO_{X}) = \OH^{2}(X, \SO_{X}) = 0 \).

\item  \label{prop:step:M:1}
\( \OH^{2}(M, \Theta_{M/\Bbbk}(-\log D)) = \OH^{2}(X, \Theta_{X}) = 0 \).

\item \label{prop:step:M:2} \( K_{X}\) is nef and big,
and it is \( \BQQ \)-linearly equivalent to an effective \( \BQQ \)-divisor
whose support is the union of \( D \) and
the \( \varphi \)-exceptional locus.

\item \label{prop:step:M:3}
\( K_{X} \) is ample if and only if
there is no \( (-2) \)-curve contained in \( M \setminus D \).
If \eqref{condA:1} and \eqref{condA:3} are satisfied, then \( K_{X} \) is ample.
If \( K_{X} \) is ample, then \eqref{condA:2}
and \eqref{condA:3} are satisfied.
In case \( B \) contains \( J_{1} \) or \( J_{2} \),
\( K_{X} \) is ample if and only if \eqref{condA:2}
and \eqref{condA:3} are satisfied.

\end{enumerate}
\end{prop}

\begin{proof}
\eqref{prop:step:M:0} follows from that \( X \) has only rational singularities
and \( M \) is rational.

\eqref{prop:step:M:1}:
By Corollary~\ref{cor:compareTangent}, it suffices to prove
\( \OH^{2}(M, \Theta_{M/\Bbbk}(-\log D)) = 0 \).
Now \( B_{M} = \varphi^{-1}(B)\) is a simple normal crossing divisor
containing \( D \).
Since the cokernel of the natural injection
\( \Theta_{M/\Bbbk}(-\log B_{M}) \injmap \Theta_{M/\Bbbk}(-\log D) \)
is supported on the one-dimensional subscheme \( B_{M} - D \),
it is enough to prove:
\( \OH^{2}(M, \Theta_{M/\Bbbk}(-\log B_{M})) = 0 \).
We have an isomorphism
\( \Theta_{M/\Bbbk}(-\log B_{M}) \isom \varphi^{*}\Theta_{Z/\Bbbk}(-\log B) \),
since the center of the each step of the successive blowings up \( \varphi \)
is a node of the total transform of \( B \).
Hence, by Lemma~\ref{lem:step:Z}, we have
\[ \OH^{2}(M, \Theta_{M/\Bbbk}(-\log B_{M})) \isom
\OH^{2}(Z, \Theta_{Z/\Bbbk}(-\log B)) = 0. \]

\eqref{prop:step:M:2}: By \eqref{eq:-KZ}, \( K_{Z} + (1/2)(F_{1} +
F_{2}) \sim_{\BQQ} 0 \). Since \( F_{1} + F_{2} \) is smooth, the
pair \( (Z, (1/2)(F_{1} + F_{2})) \) is terminal (in the sense of
\cite[Definition~1.16]{Kollar}). Hence, \( K_{M} + (1/2)(F_{1}' +
F_{2}') \sim_{\BQQ} A \) for an effective \( \BQQ \)-divisor \( A \)
whose support is just the exceptional locus for
\( \varphi \colon M \to Z \). Thus,
\begin{equation}\label{eq:pullbackKX}
K_{M} + \Delta \sim_{\BQQ} A + \Delta - (1/2)(F_{1}' + F'_{2}).
\end{equation}
Here, the right hand side is an effective \( \BQQ \)-divisor
whose support is the union of \( D \) and
the \( \varphi \)-exceptional locus, since
\[ \mult_{F_{i}'}(\Delta) > 1/2 \quad \text{ for } \quad  i = 1, 2, \]
by \( F_{i}^{\prime 2} \leq F_{i}^{2} = -4 \) and
by Corollary~\ref{corsub:going-up2} and Lemma~\ref{lem:n=2}.
If \( K_{M} + \Delta \) is nef, then this is also big by \eqref{condC:K2}.
Thus, it suffices to derive a contradiction
assuming that \( K_{M} + \Delta \) is not nef.
Then, there is a curve \( \Gamma \) with \( (K_{M} + \Delta)\Gamma < 0 \).
By \eqref{eq:pullbackKX},
\( \Gamma \) is \( \varphi \)-exceptional or \( \mu \)-exceptional.
But \( \Gamma \) is not \( \mu \)-exceptional; for otherwise,
\( (K_{M} + \Delta)\Gamma = \mu^{*}(K_{X})\Gamma = 0 \).
Thus, \( \Gamma \) is \( \varphi \)-exceptional.
Hence, \( \Delta \Gamma \geq 0 \) and
\( K_{M}\Gamma < 0 \); consequently, \( \Gamma \) is a \( (-1) \)-curve with
\( \Delta \Gamma < 1 \). This is a contradiction to \eqref{condC:1+}.

\eqref{prop:step:M:3}:
By Nakai--Moishezon's criterion,
\( K_{X}\) is not ample if and only if
there is a curve \( \Gamma \) such that it is not \( \mu \)-exceptional and
\( (K_{M} + \Delta)\Gamma = 0 \).
Let \( \Gamma \) be such a curve.
Then, \( \Gamma^{2} < 0 \) by the Hodge index theorem, and moreover,
\( K_{M}\Gamma = \Delta\Gamma = 0 \) by \eqref{condC:1+}.
Thus, \( \Gamma \) is a \( (-2) \)-curve contained in \( M \setminus D\).
Conversely, if \( \Gamma \) is a \( (-2) \)-curve in \( M \setminus D \),
then \( K_{X} \) is not ample by
\( K_{X}\varphi_{*}(\Gamma) = (K_{M} + \Delta)\Gamma = 0 \).

Assume that the \( (-2) \)-curve \( \Gamma \) above is not
\( \varphi \)-exceptional.
Then, \( \Gamma \) does not intersect
the \( \varphi \)-exceptional locus,
since \( A\Gamma = 0 \) by \eqref{eq:pullbackKX}.
Hence, \( \varphi(\Gamma) \cap \varphi(D) = \emptyset \).
In particular, \( \tau\varphi(\Gamma) \subset Y \setminus (\bar{S} \cup \bar{G}) \).
Therefore,  \eqref{condA:1} and \eqref{condA:3}
imply the ampleness of \( K_{X} \).
If \( B \) contains \( J_{1} \),
then \( \varphi(\Gamma) \) does not intersect the fiber \( F_{1} \cup J_{1} \);
hence, \( \tau\varphi(\Gamma) \) is a \( (-2) \)-curve contained in a fiber of
\( Y \to \BPP^{1} \) and is contained in \( Y \setminus (\bar{S} \cup \bar{G}) \).
Thus, \eqref{condA:2} and \eqref{condA:3}
imply the ampleness of \( K_{X} \) in this case.
If \eqref{condA:2} or \eqref{condA:3} fails,
then we can find an irreducible curve
\( \Gamma \) contained in \( M \setminus D \); thus \( K_{X} \) is not ample.
\end{proof}

\begin{remn}
Our proof of Proposition~\ref{prop:step:M}.\eqref{prop:step:M:1} is
different from the proof of the corresponding result in
\cite[Section~4]{LeePark}, but has many common ideas.
\end{remn}

\begin{lem}\label{lem:condC:1+}
In Step~\emph{\ref{step:M}}, the condition
\eqref{condC:1+} is derived from the following three conditions
for all the \( \varphi \)-exceptional \( (-1) \)-curves \( \Gamma \)\emph{:}
\begin{enumerate}

\item \label{condC:1+:1}
\( \Gamma \) intersects at least two irreducible components of \( D \).

\item \label{condC:1+:2}
If \( \Gamma \) intersects a connected component of \( D \) which
defines a toric singularity of type \( T(l, 2, 1) \) for some \( l \geq 1 \),
then \( \Gamma \) intersects another component of \( D \)
defining a toric singularity of type \( T(d, n, a)\)
with \( n > 2 \).

\item \label{condC:1+:3} \( \Delta \Gamma > 1 \)
if \( \Gamma \) intersects a \( (-2) \)-curve
belonging to a sequence of \( (-2) \)-curves
at an end of a connected component of \( D \).
\end{enumerate}
\end{lem}

\begin{proof}
Let \( \Gamma \) be a \( \varphi \)-exceptional \( (-1) \)-curve.
Assume that \( \Gamma \) intersects two irreducible components
\( \Xi_{1} \), \( \Xi_{2} \) of \( D \) such that
\begin{itemize}
\item  \( \Xi_{i} \) does not belong to
any sequence of \( (-2) \)-curves at an end of a connected component of \( D \)
for \( i = 1 \), \( 2 \),

\item  the linear chain containing \( \Xi_{1} \) defines a singularity
of type \( T(d, n, a) \) with \( n > 1 \).
\end{itemize}
It suffices to prove \( \Delta\Gamma > 1 \).
By Lemma~\ref{corsub:going-up2},
the multiplicity of \( \Delta \) along \( \Xi_{1} \) is greater than \( 1/2 \)
and the multiplicity along \( \Xi_{2} \) is greater than or equal to \( 1/2 \).
Thus, we have \( \Delta \Gamma > 1 \).
\end{proof}

Applying Theorem~\ref{thm:globalsmoothing} to \( X \) above,
we have:

\begin{prop}\label{prop:LPmethod1}
Let \( \Phi \) be a cubic pencil on \( \BPP^{2}_{\Bbbk} \)
satisfying \eqref{condC:Phi} and \eqref{condC:nodal},
and let \( Y \to \BPP^{2} \) be the elimination of the base locus of \( \Phi \)
by a succession of blowings up at points.
Let \( \bar{F} \), \( \bar{G} \), \( \bar{S} \), \( \bar{B} \),
\( \bar{B}^{+} \) be divisors on \( Y \) satisfying
conditions \eqref{condC:NC1} and \eqref{condC:LinIndep}, and also
conditions \eqref{condC:NC2} and \eqref{condC:-1} on the blown up
surface \( Z \) at the nodes of \( \bar{F} \).
Let \( \varphi \colon M \to Z \) be a succession of blowings
up at certain nodes of the total transform of \( \bar{B}^{+} \), and
let \( D = \sum D^{(i)} \) be a disjoint union of linear chains of
smooth rational curves contained in the total transform of \( \bar{B} \)
in which the conditions \eqref{condC:DB}--\eqref{condC:end} are satisfied.
Then, there is a minimal projective surface \( \BSS \) of general type
defined over \( \Bbbk \) such that \( p_{g}(\BSS) = q(\BSS) = 0 \),
\( K_{\BSS}^{2} \) equals the given integer \( 1 \leq K^{2} \leq 4\) in the Main Theorem,
and \( \OH^{2}(\BSS, \Theta_{\BSS/\Bbbk}) = 0 \).
If \( M \setminus D \) contains no \( (-2) \)-curves, we can even assume
the canonical divisor \( K_{\BSS} \) to be ample.
\end{prop}

\begin{step}\label{step:XLambda}
Let \( B_{M}^{+} \) be the total transform of \( B^{+} \) in \( M \).
This is a simple normal crossing divisor consisting of rational curves.
Since \( \tau \circ \varphi \colon M \to Y \) is
a succession of blowings up whose centers are nodes of
the total transform of \( \bar{B} \),
by Lemma~\ref{lem:piLambda} and by the condition on \( \bar{S}_{j} \)
in \emph{Step}~\ref{step:Y},
we can realize \( M \) as the closed
fiber of a smooth projective family \( M_{\Lambda} \to \Spec \Lambda \)
of rational surfaces and
\( B_{M}^{+} \) as the closed fiber of a divisor \( B_{M_{\Lambda}}^{+} \) on
\( M_{\Lambda}\) flat over \( \Lambda \) such that
every irreducible component of \( B_{M_{\Lambda}}^{+} \) is
a \( \BPP^{1} \)-bundle over \( \Lambda \) and
that any non-empty intersection of two
prime divisors of \( B_{M_{\Lambda}}^{+} \) is
a section over \( \Spec \Lambda \).
Moreover, we have a relative Cartier divisor on \( M_{\Lambda} \)
whose restriction to \( M \) is linearly equivalent to the pullback
of an ample divisor on \( X \). For,
the exceptional divisors for the birational
morphism \( M_{\Lambda} \to Y_{\Lambda} \to \BPP^{2}_{\Lambda}\)
and the pullback of \( \SO_{\BPP^{2}}(1) \)
generate \( \Pic(M_{\Lambda}) \).
Then, by Proposition~\ref{prop:contraction} below,
we have a projective family \( X_{\Lambda} \to \Spec \Lambda \)
of normal projective surfaces and a birational morphism
\( \mu_{\Lambda} \colon M_{\Lambda} \to X_{\Lambda}  \) over \( \Spec \Lambda \)
such that the closed fiber of \( X_{\Lambda} \) is isomorphic to \( X \)
and the morphism \( \mu_{\Lambda} \) restricted to the closed fibers
is just the minimal resolution \( \mu \colon M \to X \).
\end{step}

\begin{prop}\label{prop:contraction}
Let \( \Lambda \) be a discrete valuation ring with the residue field \( \Bbbk \)
which is algebraically closed. Let \( V \) be a normal projective surface
with only rational singularities defined over \( \Bbbk \).
Let \( \widetilde{W} \to \Spec \Lambda\) be a smooth projective morphism
such that \( W := \widetilde{W} \times_{\Spec \Lambda} \Spec \Bbbk\)
is the minimal resolution of singularities of \( V \).
Let \( \{\widetilde{E}_{i}\}_{i \in I} \) be a set of
prime divisors on \( \widetilde{W} \) such that\emph{:}
\begin{itemize}
\item \( \widetilde{E}_{i} \) is a \( \BPP^{1} \)-bundle over
\( \Spec \Lambda \) for any \( i \in I \).
\item
\( E_{i} = \widetilde{E}_{i} \times_{\Spec \Lambda} \Spec \Bbbk \)
is an irreducible component
of the exceptional locus \( E \) of the minimal resolution
\( \nu \colon W \to V \).
\item  Conversely, any irreducible component of \( E\)
is equal to \( E_{i} \) for a unique \( i \in I\).

\item  \( \widetilde{E}_{i} \cap \widetilde{E}_{j} \) is a section
of \( \widetilde{W} \to \Spec \Lambda\)
if \( E_{i} \cap E_{j} \ne \emptyset \).
\end{itemize}
Assume that there is a divisor \( \widetilde{L} \) on \( \widetilde{W} \)
such that \( L := \widetilde{L}|_{W} \) is linearly equivalent to
the pullback of an ample divisor on \( V \).
Then, there exist a normal projective \( \Lambda \)-scheme \( \widetilde{V} \)
and a proper birational morphism
\( \tilde{\nu} \colon \widetilde{W} \to \widetilde{V} \)
satisfying the following conditions\emph{:}
\begin{enumerate}
\item  \( V \isom \widetilde{V} \times_{\Spec \Lambda} \Spec \Bbbk  \), and
\( \nu \) is obtained by the base change of \( \widetilde{\nu} \)
by \( \Spec \Bbbk \to \Spec \Lambda\).

\item  The \( \tilde{\nu} \)-exceptional locus
is \( \bigcup \widetilde{E}_{i} \), and \( \tilde{\nu}(\widetilde{E}_{i}) \)
is a section of \( \widetilde{V} \to \Spec \Lambda \) for all \( i \in I\),

\item \label{prop:contraction:index}
If \( rK_{V} \) is Cartier for a positive integer \( r \), then
so is \( rK_{\widetilde{V}} \) and \( rK_{\widetilde{V}}|_{V} \sim rK_{V} \).
\end{enumerate}
\end{prop}

\begin{proof}
We may assume that \( L = \nu^{*}(L_{V}) \) for a very ample divisor \( L_{V} \)
on \( V \) such that \( \OH^{i}(V, \SO_{V}(L_{V})) = 0 \) for all \( i > 0 \),
by replacing \( L \) with \( mL \) for \( m \gg 0 \). Since
\( V \) has only rational singularities,
\( \OH^{i}(W, \SO_{W}(L)) = 0 \) for any \( i > 0 \).
The morphism defined by the base point free linear system \( |L| \) is
just \( \nu \colon W \to V \) followed by
a closed immersion \( V \injmap \BPP^{N} \), where \( N = \dim |L| \).
We can show:

\begin{clsub}\label{clsub:prop:contraction}
\begin{enumerate}
    \renewcommand{\theenumi}{\roman{enumi}}
    \renewcommand{\labelenumi}{(\theenumi)}
\item \label{clsub:prop:contraction:1}
\( \OH^{i}(\widetilde{W}, \SO_{\widetilde{W}}(\widetilde{L})) = 0 \)
for any \( i > 0 \).

\item \label{clsub:prop:contraction:2}
The natural homomorphism
\( \OH^{0}(\widetilde{W}, \SO_{\widetilde{W}}(\widetilde{L}))
\otimes_{\Lambda} \Bbbk \to
\OH^{0}(W, \SO_{W}(L))\)
is an isomorphism.

\item \label{clsub:prop:contraction:3}
\( \SO_{\widetilde{W}}(\widetilde{L}) \) is generated by global sections,
i.e.,
\( \OH^{0}(\widetilde{W}, \SO_{\widetilde{W}}(\widetilde{L}))
\otimes_{\Lambda} \SO_{\widetilde{W}}
\to \SO_{\widetilde{W}}(\widetilde{L})\)
is surjective.
\end{enumerate}
\end{clsub}

\begin{proof}
In fact, \eqref{clsub:prop:contraction:1} is a consequence of
the upper semi-continuity theorem
for \( \widetilde{W} \to \Spec \Lambda \) and the vanishing
\( \OH^{i}(W, \SO_{W}(L)) = 0 \)
for any \( i > 0 \).
Then, \eqref{clsub:prop:contraction:2} is obtained as the
base change isomorphism.
The homomorphism in \eqref{clsub:prop:contraction:3} is surjective after
tensoring \( \Bbbk \) over \( \Lambda \),
by the freeness of \( |L| \) and
by \eqref{clsub:prop:contraction:2}. Thus, the homomorphism is surjective
along the closed fiber of \( \widetilde{W} \to \Spec \Lambda \), and
is surjective everywhere on \( \widetilde{W} \),
since \( \Lambda \) is a local ring.
\end{proof}

We continue the proof of Proposition~\ref{prop:contraction}.
Now, by Claim~\ref{clsub:prop:contraction}.\eqref{clsub:prop:contraction:2}
above, we see that
\( \OH^{0}(\widetilde{W}, \SO_{\widetilde{W}}(\widetilde{L}))  \)
is a free \( \Lambda \)-module of rank \( N \), and
the surjection in
Claim~\ref{clsub:prop:contraction}.\eqref{clsub:prop:contraction:3}
defines a morphism
\( \widetilde{W} \to \BPP^{N}_{\Lambda} \).
Let \( \widetilde{W} \to \widetilde{V} \to \BPP^{N}_{\Lambda} \) be
the Stein factorization of \( \widetilde{W} \to \BPP^{N}_{\Lambda}  \).
Then, \( \widetilde{V} \) is normal and
the fiber over \( \Spec \Bbbk \) of
the proper morphism \( \tilde{\nu} \colon \widetilde{W} \to \widetilde{V} \)
is just \( \nu \colon W \to V \).

By construction, \( \tilde{\nu} \) is an isomorphism on
\( W \setminus \bigcup E_{i} \);
hence \( \tilde{\nu} \) is a birational morphism.
Every \( \widetilde{E}_{i} \) for \( i \in I \)
is \( \tilde{\nu} \)-exceptional, since
\( \SO_{\widetilde{W}}(\widetilde{L})|_{\widetilde{E}_{i}}
\isom \SO_{\widetilde{E}_{i}}\).
We shall show that there is no other \( \tilde{\nu} \)-exceptional divisor.
Let \( \widetilde{\Gamma} \) be a \( \tilde{\nu} \)-exceptional
prime divisor on \( \widetilde{W} \).
If \( \widetilde{\Gamma} \) is not flat over \( \Lambda \), then
\( \widetilde{\Gamma} \)
is contained in the closed fiber \( W \) and also contained in some \( E_{i} \);
this is a contradiction, since  \( \Codim (E_{i}, \widetilde{W}) = 2 \).
Thus, \( \widetilde{\Gamma} \) is flat over \( \Lambda \).
Then, the closed fiber \( \Gamma \) of
\( \widetilde{\Gamma} \to \Spec \Lambda \)
is a union of \( \nu \)-exceptional curves.
Here, the intersection number \( b := \Gamma E_{i} \) is negative
for some \( i \in I \).
Hence, \( \widetilde{\Gamma} = \widetilde{E}_{i} \), since
\( \SO_{\widetilde{W}}(\widetilde{\Gamma})|_{\widetilde{E}_{i}} \)
is isomorphic to \( \SO_{\BPP^{1}_{\Lambda}}(b) \).
Therefore, the \( \tilde{\nu} \)-exceptional locus is
\( \bigcup \widetilde{E}_{i} \).
Since \( \widetilde{E}_{i} \) is a \( \BPP^{1} \)-bundle over \( \Lambda \),
the image \( \tilde{\nu}(\widetilde{E}_{i}) \)
is a section of \( \widetilde{V} \to \Spec \Lambda \).
Thus, all the assertions except \eqref{prop:contraction:index} are satisfied.

For \eqref{prop:contraction:index},
it suffices to show that \( rK_{\widetilde{V}} \) is Cartier.
In fact, if \( rK_{\widetilde{V}} \) is Cartier, then
\( rK_{\widetilde{V}}|_{V} \sim rK_{V} \) by Remark~\ref{remsub:lem:GorIndex}.
We consider the rational numbers
\( b_{i} \geq 0\) such that \( \nu^{*}(K_{X}) = K_{W} + \sum b_{i}E_{i}\).
Then, \( rb_{i} \in \BZZ \) for any \( i \).
Let \( \widetilde{D} \) be the Cartier divisor
\( rK_{\widetilde{W}} + \sum rb_{i} \widetilde{E}_{i} \) on \( \widetilde{W} \).
Then, \( \widetilde{D}|_{W} \sim \nu^{*}(rK_{V}) \).
We apply the argument above to
\( \widetilde{L'} = a\widetilde{L} + \widetilde{D} \) for \( a \gg 0 \)
instead of \( \widetilde{L} \), where \( aL_{V} + rK_{V} \) is ample on \( V \).
Then, \( \SO_{\widetilde{W}}(\widetilde{L'}) \) is generated by global sections and
defines a morphism \( \widetilde{W} \to \BPP^{N'}_{\Lambda} \) for
some \( N' > 0\). Let \( \tilde{\nu}' \colon \widetilde{W} \to \widetilde{V'} \) be the birational morphism
obtained as the Stein factorization. Then, the exceptional locus of \( \widetilde{W} \to \widetilde{V'}  \)
is just the union of \( \bigcup \widetilde{E}_{i} \).
By Zariski's main theorem, we have an isomorphism
\( \widetilde{V} \isom \widetilde{V'} \) compatible with \( \tilde{\nu} \)
and \( \tilde{\nu}' \). Consequently, \( \widetilde{D} \) is the pullback
of a Cartier divisor on \( \widetilde{V} \).
Hence \( rK_{\widetilde{V}} \) is Cartier and
\( \widetilde{D} = \tilde{\nu}^{*}(rK_{\widetilde{V}}) \) by
\( \tilde{\nu}_{*}(\widetilde{D}) = \tilde{\nu}_{*}(rK_{\widetilde{W}}) =
rK_{\widetilde{V}}\).
Thus, we have finished the proof of Proposition~\ref{prop:contraction}.
\end{proof}

\begin{step}\label{step:redC}
We can apply Theorem~\ref{thm:DVRsmoothing} to
\( X_{\Lambda} \to \Spec \Lambda \)
constructed in \emph{Step}~\ref{step:XLambda}. In fact,
the assumptions \eqref{thm:DVRsmoothing:ass1} and \eqref{thm:DVRsmoothing:ass2}
of Theorem~\ref{thm:DVRsmoothing}
are confirmed by Proposition~\ref{prop:step:M}
and by an argument in \emph{Step}~\ref{step:M} and
the condition \eqref{condC:end}, respectively.
Let \( X_{\BKK}\) be the geometric generic fiber
of \( X_{\Lambda} \to \Spec \Lambda \)
for an algebraically closed field \( \BKK \) containing \( \Lambda \).
Thus, by Theorem~\ref{thm:DVRsmoothing} and by Corollary~\ref{cor:Pi1},
we have deformations
\( \SX \to T \) of \( X \) and \( \SX_{\BKK} \to T_{\BKK} \) of
\( \overline{X} \) such that:
\begin{itemize}
\item  \( T \) is a non-singular algebraic curve over \( \Bbbk \), and
\( \SX \to T \) satisfies the conditions
\eqref{thm:globalsmoothing:1} and \eqref{thm:globalsmoothing:2}
of Theorem~\ref{thm:globalsmoothing}.

\item  \( T_{\BKK} \) is a non-singular algebraic curve over
\( \BKK \), and \( \SX_{\BKK} \to T_{\BKK} \) satisfies the conditions
\eqref{thm:globalsmoothing:1} and \eqref{thm:globalsmoothing:2}
of Theorem~\ref{thm:globalsmoothing}.

\item  \( \pi_{1}^{\alg}(X_{\BKK, t'}) \to \pi_{1}^{\alg}(X_{t}) \)
is surjective for any closed smooth fibers \( X_{t} = \SX \times_{T} t\) and
\( X_{\BKK, t'} = \SX_{\BKK} \times_{T_{\BKK}} t'\).
\end{itemize}
In order to get the algebraic simply connectedness of \( X_{\BKK, t'} \),
we shall compare \( X_{\BKK} \) and a normal projective surface
\( \overline{X} \) defined over \( \BCC \) which is constructed by
the same method as in \emph{Steps}~\ref{step:pencil}--\ref{step:M}.

Let \( Y_{\BCC} \to \BPP^{1}_{\BCC}\) be the elliptic fibration obtained from
the pencil \( \Phi \) on \( \BPP^{2}_{\BCC} \) by
replacing \( \Bbbk \) with \( \BCC \).
Then, by the same process as in \emph{Steps}~\ref{step:Z} and \ref{step:M},
we have birational morphisms \( \tau_{\BCC} \colon Z_{\BCC} \to Y_{\BCC} \),
\( \varphi_{\BCC} \colon M_{\BCC} \to Z_{\BCC} \), and a disjoint union
\( D_{\BCC} \) of linear chains of smooth rational curves such that
the contraction of \( D_{\BCC} \) is the minimal resolution
\( \mu_{\BCC} \colon M_{\BCC} \to X_{\BCC} \) of
a normal projective surface \( X_{\BCC} \)
with only toric singularities of class T.
Let \( \BKK_{0} \) be a subfield of \( \BKK \) finitely generated over
\( \BQQ \) such that \( X_{\BKK} \), \( \Sing X_{\BKK} \),
and \( \overline{X} \to \overline{T} \) are defined over \( \BKK_{0} \).
We take an injection \( \BKK_{0} \injmap \BCC \) and set
\( \overline{X} := X_{\BKK_{0}} \times_{\Spec \BKK_{0}} \Spec \BCC \).
Then, \( \overline{X} \isom X_{\BCC} \)
by our construction of \( X_{\Lambda} \to \Spec \Lambda \).

\begin{lem}\label{lem:simplyconn}
If \( X_{\BCC} \setminus \Sing X_{\BCC} \isom M_{\BCC} \setminus D_{\BCC}\)
is simply connected \emph{(}with respect to the Euclidean topology\emph{)},
then a general smooth closed fiber \( X_{t} \) of \( \SX \to T \)
is algebraically simply connected.
\end{lem}

\begin{proof}
By the assumption, any \( \BQQ \)-Gorenstein smoothing of \( X_{\BCC} \) is
simply connected by an argument in the proof of \cite[Theorem~3.1]{LeePark}
using results on rational blowdown \( 4 \)-manifolds, Milnor fibers, and
van-Kampen's theorem.
Thus, by Remark~\ref{remsub:cor:Pi1}, a smooth closed fiber \( X_{\BKK, t'} \)
of \( \SX_{\BKK} \to T_{\BKK} \) is algebraically simply connected.
Therefore, so is \( X_{t}\) by Corollary~\ref{cor:Pi1}.
\end{proof}

The following is useful for proving the simply connectedness of
\( M_{\BCC} \setminus D_{\BCC} \) and used in the proof of
\cite[Theorem~3.1]{LeePark}.

\begin{lemsub}\label{lemsub:simplyconn}
In the construction of \( M \) and \( D \) above,
assume that \( \Bbbk = \BCC \).
Then, \( M \setminus D \) is simply connected
\emph{(}with respect to the euclidean topology\emph{)}
provided that, for any connected component \( D_{i} \) of \( D \),
there exists a smooth rational curve \( E \) on \( M \)
and another connected component \( D_{j} \) of \( D \) such that\emph{:}
\begin{enumerate}
\item  \label{lemsub:simplyconn:1}
\( E \) intersects an end component of \( D_{i} \) and
an end component of \( D_{j} \).

\item  \label{lemsub:simplyconn:2} \( E D_{i} = ED_{j} = 1 \), and
\( E \) does not intersect the other connected components of \( D \).

\item  \label{lemsub:simplyconn:3}
\( \gcd(d_{i}n_{i}, d_{j}n_{j}) = 1 \) for the types
\( T(d_{i}, n_{i}, a_{i}) \) and \( T(d_{j}, n_{j}, a_{j}) \)
of the toric singularities defined by \( D_{i} \) and \( D_{j} \),
respectively.
\end{enumerate}
\end{lemsub}

\begin{proof}
Let \( Q_{i} \) be the toric singular point \( \mu(D_{i}) \) of class T
and let \( T(d_{i}, n_{i}, a_{i}) \) be the type.
We have an open neighborhood \( \overline{\SU}_{i} \)
(with respect to the euclidean topology) of \( Q_{i} \) and
a finite surjective morphism
\( \lambda_{i} \colon \SU_{i}\sptilde \to \overline{\SU}_{i} \)
such that
\begin{itemize}
\item  \( \overline{\SU}_{i} \cap \Sing X = \{Q_{i}\} \),

\item  \( \overline{\SU}_{i} \) is topologically contractible,

\item  \( \SU_{i}\sptilde \) is a non-singular surface topologically
contractible to a point,

\item \( \lambda_{i} \) is \'etale over \( \overline{\SU}_{i} \setminus \{Q_{i}\} \),

\item \( \lambda_{i} \) is a cyclic covering of degree \( d_{i}n_{i}^{2} \).
\end{itemize}
In particular,
\( \pi_{1}(\overline{\SU}_{i} \setminus \{Q_{i}\}) \) is a cyclic group of
order \( d_{i}n_{i}^{2} \).
We set \( \SU_{i} := \mu^{-1}(\overline{\SU}_{i})\), which is
homotopic to \( D_{i} \). Hence, \( \SU_{i} \) is simply connected.
Moreover,
\( \SU_{i} \setminus D = \SU_{i} \setminus D_{i} \isom
\overline{\SU}_{i} \setminus Q_{i} \).
We may assume that \( \SU_{i} \cap \SU_{j} = \emptyset\)
for any \( i \), \( j \).
For the open immersion \( \SU_{i} \setminus D \injmap M \setminus D \),
we have a homomorphism
\( \pi_{1}(\SU_{i} \setminus D) \to \pi_{1}(M \setminus D) \)
defined up to conjugate by considering a path connecting the
reference points of \( \SU_{i} \setminus D \) and \( M \setminus D \).
Let \( \gamma_{i} \) be the image of a generator
of \( \pi_{1}(\SU_{i} \setminus D) \) in  \( \pi_{1}(M \setminus D) \).
Then, \( \pi_{1}(M \setminus D) \) is generated by the images
\( \gamma_{i} \) for any \( i \)
by van-Kampen's theorem applied to the open covering
\( M = (M \setminus D) \cup \bigcup \SU_{i} \),
since \( M \) and \( \SU_{i} \) are simply connected.

For a connected component \( D_{i} \), let \( D_{j} \) be
another connected component and \( E \) be a smooth rational curve satisfying
\eqref{lemsub:simplyconn:1}--\eqref{lemsub:simplyconn:3}.
Then, \( E \setminus D \isom \BCC \setminus \{0\}\).
Let \( \gamma_{E} \) be the image of a generator of
the cyclic group
\( \pi_{1}(E \setminus D) \) in \( \pi_{1}(M \setminus D) \)
by a homomorphism \( \pi_{1}(E \setminus D) \to \pi_{1}(M \setminus D) \)
defined up to conjugate similarly to the above
from the inclusion \( E \setminus D \injmap M \setminus D \).
Then, \( \gamma_{E} \) is conjugate to
\( \gamma_{i} \) or \( \gamma_{i}^{-1} \),
and conjugate to \( \gamma_{j} \) or \( \gamma_{j}^{-1} \), since
\( E \) intersects \( D_{i} \) (resp. \( D_{j} \)) transversely
only at one point which is in an end component.
Therefore, \( \gamma_{i} \) is conjugate to \( \gamma_{j} \) or \( \gamma_{j}^{-1} \).
Hence, \( \gamma_{i} = \gamma_{j} = 1 \) by \eqref{lemsub:simplyconn:3}.
This proves that \( M \setminus D \) is simply connected.
\end{proof}

\end{step}

By the discussion in \emph{Steps}~\ref{step:XLambda} and \ref{step:redC},
we have:

\begin{prop}\label{prop:LPmethod2}
In the situation of Proposition~\emph{\ref{prop:LPmethod1}},
let \( M_{\BCC} \) be the same surface obtained as above by replacing \( \Bbbk \)
with \( \BCC \), and let \( D_{\BCC} \) be the same linear chain as \( D \)
on \( M_{\BCC} \).
If \( M_{\BCC} \setminus D_{\BCC} \) is simply connected in addition,
then one can require the surface \( \BSS \) in
Proposition~\emph{\ref{prop:LPmethod1}} to be algebraically simply connected.
\end{prop}


\section{Proof of Main Theorem}
\label{sect:proof}

We shall prove the Main Theorem by giving explicit examples using
the method in Section~\ref{sect:Global}.
In Examples~\ref{exam:K2=2Main}--\ref{exam:char2,K2=2} below,
the necessary tasks are:
\begin{itemize}
\item  Giving assumptions on \( \chara(\Bbbk) \) and \( K^{2} \).

\item  Defining two cubic homogeneous polynomials \( \phi_{0} \) and \( \phi_{\infty} \)
in \( \BZZ[\xtt, \ytt, \ztt] \) and checking the conditions
\eqref{condC:Phi} and \eqref{condC:nodal}
(cf.\ \emph{Step}~\ref{step:pencil} in Section~\ref{sect:Global}).

\item  Defining the divisors \( \bar{F} \), \( \bar{S} \), and
\( \bar{G} \) on \( Y \), and checking \eqref{condC:NC1}
and \eqref{condC:LinIndep}.
For the ampleness of \( K_{X_{t}} \),
we check \eqref{condA:1} or \eqref{condA:2}
(cf.\ \emph{Step}~\ref{step:Y} in Section~\ref{sect:Global}).

\item Defining the divisor \( B \) on \( Z \) after
checking \eqref{condC:NC2} and \eqref{condC:-1}
(cf.\ \emph{Step}~\ref{step:Z} in Section~\ref{sect:Global}).

\item Defining the birational morphism \( \varphi \colon M \to Y \) and
the union \( D \) of linear chains of rational curves on \( M \)
satisfying \eqref{condC:DB}--\eqref{condC:end}.
For the ampleness of \( K_{X_{t}} \),
we check \eqref{condA:3} (cf.\ \emph{Step}~\ref{step:M} in Section~\ref{sect:Global}).

\item If possible, proving the ampleness of \( K_{X} \)
(equivalent to the absence of \( (-2) \)-curves on \( M \setminus D \))
using \eqref{condA:1}--\eqref{condA:3}, or another argument
(cf.\ Proposition~\ref{prop:step:M}.\eqref{prop:step:M:3}).

\item Proving that \( M \setminus D \)
is simply connected (with respect to the euclidean topology)
when \( \Bbbk = \BCC \) using Lemma~\ref{lemsub:simplyconn} or
referring to papers \cite{LeePark}, \cite{PPS1} and \cite{PPS2}.
\end{itemize}

Having done these tasks, we obtain the desired surfaces by
Propositions~\ref{prop:LPmethod1} and \ref{prop:LPmethod2}.
The proof of the Main Theorem is written at the end.

\begin{notan}
\begin{itemize}
\item Let \( (\xtt:\ytt:\ztt) \) be a homogeneous coordinate of \( \BPP^{2} \).
\item A \( (-k) \)-curve means a non-singular rational curve
with self-intersection number \( -k \).

\item  The symbol \( \LC(b_{1}, \ldots, b_{l}) \) expresses
a linear chain \( E = \sum\nolimits_{i = 1}^{l} E_{i} \)
of smooth rational curves with the dual graph
\begin{center}
\begin{picture}(160, 30)(30, 0)
    \put(35, 5){\circle{10}}\put(30, 20){$E_{1}$}
    \put(40, 5){\line(1, 0){25}}
    \put(70, 5){\circle{10}}\put(65, 20){$E_{2}$}
    \put(75, 5){\line(1, 0){10}}
    \put(90, 5){\line(1, 0){10}}
    \put(105, 5){\line(1, 0){10}}
    \put(120, 5){\line(1, 0){10}}
    \put(135, 5){\line(1, 0){10}}
    \put(150, 5){\circle{10}}\put(145, 20){$E_{l}$}
\end{picture}
\end{center}
such that
\( E_{i}^{2} = -b_{i} \) for \( 1 \leq i \leq l \).
Here, \( E_{i} \) is called the \( i \)-th component.
\item For a non-zero regular section \( \phi \) of an invertible sheaf,
\( (\phi)_{0} \) denotes the divisor of zeros of \( \phi \).
\item The configuration type of singular fibers for
a minimal elliptic fibration is the list of
singular fibers written by Kodaira's symbol
(cf.\ \cite[Theorem~6.2]{Kodaira}).
\end{itemize}
\end{notan}

\begin{exam}\label{exam:K2=2Main}
Assume that \( \chara(\Bbbk) \ne 2, 3\), and
we set \( K^{2} = 2 \). Here, we use essentially the same
construction as in \cite[Section~3]{LeePark}. We set
\[ \phi_{0} := (3\xtt\ytt - \ztt^{2})\ztt \quad \text{and} \quad
\phi_{\infty} := (\xtt + \ytt)^{3}.\]
Then, \( \Phi_{0} = (\phi_{0})_{0} = Q + L_{1}\) and
\( \Phi_{\infty} = (\phi_{\infty})_{0} = 3L_{2} \)
for the conic \( Q := (3\xtt\ytt - \ztt^{2})_{0} \) and for the lines
\( L_{1} := (\ztt)_{0} \) and \( L_{2} := (\xtt + \ytt)_{0} \).
Thus, \eqref{condC:Phi} holds.
Moreover, \eqref{condC:nodal} also holds; In fact, for \( c \ne 0 \),
the divisor \( \Phi_{c} = (\phi_{0} + c\phi_{\infty})_{0}\)
is singular if and only if
\( c = \pm 1/2 \), where \( \Phi_{1/2} \) and \( \Phi_{-1/2} \)
are nodal rational curves with the nodes at \( (1:-1:1) \) and \( (1:1:1) \),
respectively.
On the minimal elliptic fibration \( \pi \colon Y \to \BPP^{1}_{\Bbbk} \)
defined by \( \Phi \), the configuration type of singular fibers is
\( (\text{IV}^{*}, \text{I}_{2}, \text{I}_{1}, \text{I}_{1}) \), and
the \( (-1) \)-curves exceptional for \( Y \to \BPP^{2}_{\Bbbk} \)
are mutually disjoint sections \( \bar{S}_{1} \), \( \bar{S}_{2} \), and
\( \bar{S}_{3} \) as in Figure~\ref{fig-Y1}.
\begin{figure}[hbtb]
\begin{center}
\includegraphics[scale=1]{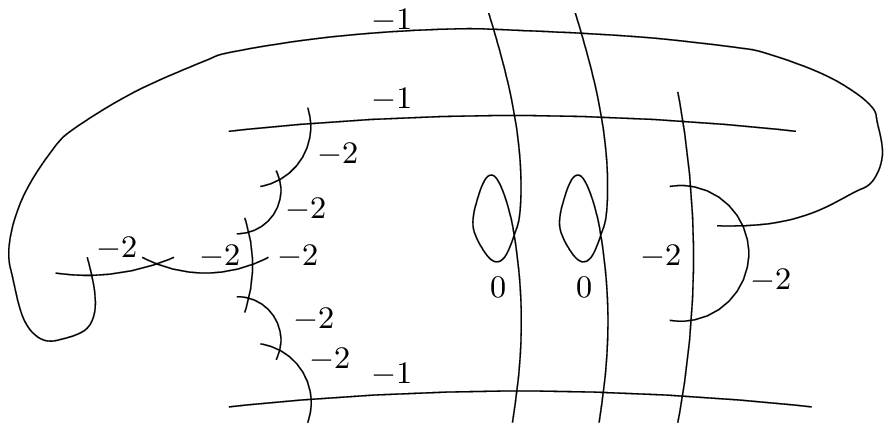}
\end{center}
\caption{The rational elliptic surface $Y$ in Example~\ref{exam:K2=2Main}}
\label{fig-Y1}
\end{figure}
Let \( \bar{F}_{1} \) and \( \bar{F}_{2} \) be the proper transforms of
\( \Phi_{1/2} \) and \( \Phi_{-1/2} \) in \( Y \), respectively.
These are the singular fibers of type \( \text{I}_{1} \), and
we define \( \bar{F} := \bar{F}_{1} + \bar{F}_{2} \).
The singular fiber of type \( \text{I}_{2} \) is the union of
the proper transforms \( Q\sptilde \) and \( L_{1}\sptilde \)
of \( Q \) and \( L_{1} \) in \( Y \), respectively.
The central irreducible component of the singular fiber of
type \( \text{IV}^{*} \), which meets three other components, is just
the proper transform \( L_{2}\sptilde \)
of \( L_{2} \) in \( Y \).
We may assume that \( \bar{S}_{1} \), \( \bar{S}_{2} \), and \( \bar{S}_{3} \)
are contracted to the points
\( (1:-1:\sqrt{-3}) \), \( (1:-1:-\sqrt{-3}) \), and \( (1:-1: 0) \),
respectively,
where \( Q \cap L_{2} = \{(1:-1:\sqrt{-3}), (1:-1:-\sqrt{-3})\}\),
and \( L_{1} \cap L_{2} = (1:-1:0) \).
Hence, \( \bar{S}_{1} \) and \( \bar{S}_{2} \)
intersect \( Q\sptilde \),
and \( \bar{S}_{3} \) intersects \( L_{1}\sptilde \).
We define \( \bar{S} \) to be \( \bar{S}_{1} + \bar{S}_{2} + \bar{S}_{3} \).
The union \( \bar{G}^{+} \) of all the \( (-2) \)-curves on \( Y \)
is just the union of the singular fibers of type \( \text{IV}^{*} \) and
\( \text{I}_{2} \).
Hence, \( \bar{B}^{+} = \bar{F} + \bar{G}^{+} + \bar{S}\)
satisfies the condition \eqref{condC:NC1}.
We define \( \bar{G} \) to be the union of the \( (-2) \)-curves
except for two \( (-2) \)-curves:
One is \( L_{1}\sptilde \) and the other is
the irreducible component of the singular fiber
of type \( \text{IV}^{*} \)
next to the end component meeting \( \bar{S}_{1} \).
Then, \( \bar{G} \) satisfies \eqref{condC:LinIndep}.
This is shown as follows:
Now, \( \bar{G} \) has three connected components whose
dual graphs are of type \( \SAA_{5} \), \( \SAA_{1} \), and
\( \SAA_{1} \). Hence,
\( \det (\bar{G}_{i}\bar{G}_{j}) = -24 \not\equiv 0 \mod \chara(\Bbbk) \)
by \( \det \SAA_{n} = (-1)^{n}(n+1) \).
Thus, we have \eqref{condC:LinIndep} by Lemma~\ref{lem:determinant}.

We shall prove \eqref{condA:1}.
Let \( V \to \BPP^{2}_{\Bbbk} \) be the blowing up at the point
\( (1:-1:\sqrt{-3}) \in Q \cap L_{2}\).
Then, \( V \) is the Hirzebruch surface of degree one and
the proper transform \( Q' \) of \( Q \) in \( V \) is ample.
For the total transform \( L_{1}' \) of \( L_{1} \) in \( V \),
we have an isomorphism from
\( Y \setminus (\bar{S} \cup \bar{G})\) to the affine open subset
\( V \setminus (Q' \cup L_{1}') \).
Thus, \eqref{condA:1} holds.

By Figure~\ref{fig-Y1}, we see that
\eqref{condC:NC2} is true with the exceptional divisors
\( J_{1} \) and \( J_{2} \) over the nodes, and
that \eqref{condC:-1} is also true.
We define \( B : = S + F + G\), i.e.,
\( B \) does not contain \( J_{1} \) and \( J_{2} \).

We take a birational morphism \( \varphi \colon M \to Z \)
so that the total transform \( B_{M}^{+} \)
of \( \bar{B}^{+} \) in \( M \) as in Figure~\ref{fig-tZ1}.
\begin{figure}[hbtb]
\begin{center}
\includegraphics[scale=1]{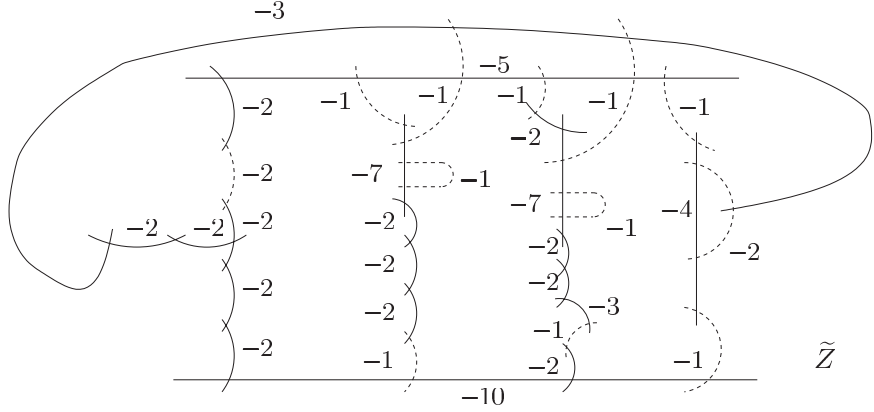}
\end{center}
\caption{The rational surface \( M \) in Example~\ref{exam:K2=2Main}}
\label{fig-tZ1}
\end{figure}
Note that \( \varphi \) is determined uniquely by this figure.
A detailed construction of \( \varphi \) is written in
\cite[Section~3]{LeePark}.
In particular, \( \rho(M) = \rho(Z) + 16 = 28 \) and
\( K_{M}^{2} = -18 \).
Here, we can find a disjoint union \( D \) of
the following five linear chains
of smooth rational curves in \( \varphi^{-1}(B) \):
\begin{gather*}
D_{1} = \LC(2, 10, 2, 2, 2, 2, 2, 3), \quad
D_{2} = \LC(2, 7, 2, 2, 3), \\
D_{3} = \LC(7, 2, 2, 2), \quad
D_{4} = \LC(5, 2), \quad
D_{5} = \LC(4)
\end{gather*}
Then, \eqref{condC:DB} is checked by
comparing Figures~\ref{fig-Y1} and \ref{fig-tZ1}.
Moreover, \eqref{condC:classT} is also true:
In fact, \( D_{1} \), \( D_{2} \), \ldots, \( D_{5} \)
define toric singularities of types
\( T(1, 15, 8) \), \( T(1, 9, 5) \), \( T(1, 5, 1) \), \( T(1, 3, 1) \),
\( T(1, 2, 1) \), respectively (cf.\ Table~\ref{table:M}).
Moreover, using Table~\ref{table:M}, we can check \eqref{condC:K2} by:
\begin{align*}
K_{X}^{2} &= K_{M}^{2} + \delta(1, 15, 8) + \delta(1, 9, 5) +
\delta(1, 5, 1) + \delta(1, 3, 1) + \delta(1, 2, 1) \\
&= -18 + 8 + 5 + 4 + 2 + 1 = 2.
\end{align*}
The condition \eqref{condC:end} is also checked by Figure~\ref{fig-tZ1}.
We shall show \eqref{condC:1+}.
By Lemma~\ref{lem:condC:1+} and by Figure~\ref{fig-tZ1},
it suffices to check \( \Delta E_{i} > 1 \) for
the \( (-1) \)-curves \( E_{1} \), \( E_{2} \), and
\( E_{3} \) characterized by:
\begin{itemize}
\item  \( E_{1} \) joins the \( (-10) \)-curve of \( D_{1} \) and
the end \( (-2) \)-component of \( D_{3} \).

\item \( E_{2} \) joins the end \( (-2) \)-component
of \( D_{2} \) and \( (-5) \)-curve of \( D_{4} \).

\item  \( E_{3} \) joins the end \( (-2) \)-component
of \( D_{1} \) and the \( (-3)\)-curve of \( D_{2}  \).
\end{itemize}
We can calculate \( \Delta E_{i} \) for \( i = 1 \), \( 2 \), \( 3 \)
using Table~\ref{table:M}, where
the multiplicity \( c_{i} \) of the \( i \)-th irreducible component
of the linear chain of type \( T(d, n, a) \) equals \( 1 - r_{i}/n \).
Thus,
\begin{gather*}
\Delta E_{1} = (1 - 1/15) + (1 - 4/5) > 1, \quad
\Delta E_{2} = (1 - 5/9) + (1 - 1/3) > 1, \text{ and } \\
\Delta E_{3} = (1 - 8/15) + (1 - 4/9) > 1.
\end{gather*}
Hence, \eqref{condC:1+} is satisfied.
Looking at
Figure~\ref{fig-tZ1}, we have \eqref{condA:3} immediately.
Thus, \( K_{X} \) is ample
by \eqref{condA:1}, \eqref{condA:3}, and
by Proposition~\ref{prop:step:M}.\eqref{prop:step:M:3}.
The simply connectedness of \( X \setminus \Sing X \isom M \setminus D \)
has been shown in the proof of \cite[Theorem~3.1]{LeePark}.
Thus, we have done all the tasks.
\end{exam}

\begin{exam}
    \label{exam:K2=2char3}
Assume that \( \chara(\Bbbk) \ne 2 \) and we set \( K^{2} = 2 \).
We follow the construction in \cite[Section~6, Construction]{LeePark}. We set
\[ \phi_{0} := \ytt^{2}\ztt - \xtt^{2}(\xtt - \ztt) \quad \text{and} \quad
\phi_{\infty} := \begin{cases}
(\xtt + \ztt)\ztt^{2}, &\text{ if } \chara(\Bbbk) \ne 5, \\
(\xtt + 2\ztt)\ztt^{2}, &\text{ if } \chara(\Bbbk) = 5.
\end{cases}\]
Then, \( \Phi_{\infty} = (\phi_{\infty})_{0} = L_{1} + 2L_{2} \) for the lines
\( L_{1} = (\xtt + \ztt)_{0} \) and \( L_{2} = (\ztt)_{0} \), and
\( \Phi_{0} = (\phi_{0})_{0}\)
is a nodal rational cubic curve with node at \( (0:0:1) \).
Furthermore, \( L_{2} \) is the tangent line of \( \Phi_{0} \)
at an inflection point \( (0:1:0) = L_{1} \cap L_{2} \).
In particular, \eqref{condC:Phi} holds.
Moreover, \eqref{condC:nodal} also holds: In fact,
we can check the following:
\begin{itemize}
\item For \( c \ne 0 \), the divisor
\( \Phi_{c} = (\phi_{0} + c\phi_{\infty})_{0}\) is singular if and only if
\( c = (11 \pm 5\sqrt{5})/2 \) when \( \chara(\Bbbk) \ne 5 \) and
\( c = 1 \pm \sqrt{3} \) when \( \chara(\Bbbk)=5\).

\item Let \( c_{\pm} \) be the constants \( (11 \pm 5\sqrt{5})/2 \)
when \( \chara(\Bbbk) \ne 5 \) and \( 1 \pm \sqrt{3}\) when \( \chara(\Bbbk) = 5 \).
Then, the node of \( \Phi_{c_{\pm}} \) is \( (-(1 \pm \sqrt{5})/2: 0 : 1) \)
when \( \chara(\Bbbk) \ne 5  \) and
\( (\pm 2\sqrt{3}: 0 : 1) \) when \( \chara(\Bbbk)=5\).
\end{itemize}
On the minimal elliptic fibration \( \pi \colon Y \to \BPP^{1}_{\Bbbk} \)
defined by \( \Phi \),
the configuration type of singular fibers is
\( (\text{III}^{*}, \text{I}_{1}, \text{I}_{1}, \text{I}_{1}) \), and
the \( (-1) \)-curves exceptional for \( Y \to \BPP^{2}_{\Bbbk} \)
are mutually disjoint sections \( \bar{S}_{1} \), \( \bar{S}_{2} \), and
\( \bar{S}_{3} \) as in Figure~\ref{figure3}.
\begin{figure}[hbtb]
\begin{center}
\includegraphics[scale=1]{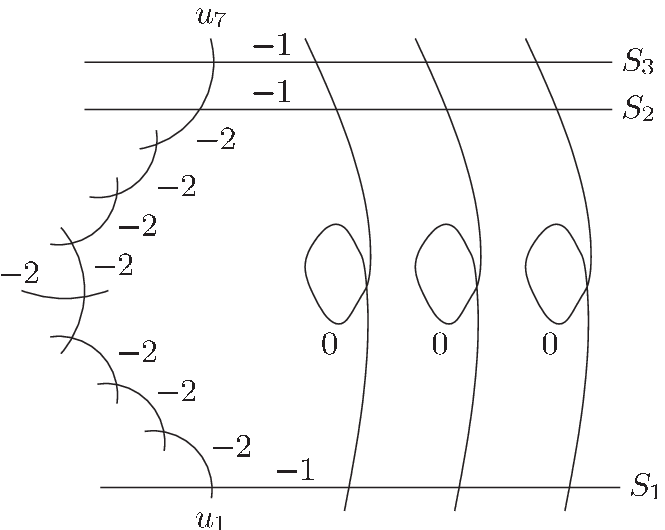}
\end{center}
\caption{The rational elliptic surface $Y$ in
Examples~\ref{exam:K2=2char3}} \label{figure3}
\end{figure}
Let \( \bar{F}_{1} \) and \( \bar{F}_{2} \) be the proper transforms of
\( \Phi_{c_{+}} \) and \( \Phi_{c_{-}} \) in \( Y \), respectively.
These are singular fibers of type \( \text{I}_{1} \).
We define \( \bar{F} := \bar{F}_{1} + \bar{F}_{2} \).
We may assume that \( \bar{S}_{1} \), \( \bar{S}_{2} \), and \( \bar{S}_{3} \)
are contracted to the points \( (0:1:0) \), \( (1: \sqrt{-2} : -1) \), and
\( (1 : -\sqrt{-2} : -1) \), respectively,
where \( \Phi_{0} \cap L_{2} = L_{1} \cap L_{2} = \{(0:1:0)\}  \)
and \( \Phi_{0} \cap L_{1} = \{(0:1:0), (1: \pm\sqrt{-2} : -1)\} \).
An end component of the singular fiber of type \( \text{III}^{*} \) is
contracted to \( (0:1:0) \) and meets \( \bar{S}_{1} \).
Another end component is the proper transform \( L_{1}\sptilde \)
of \( L_{1} \) in \( Y \)
and meets \( \bar{S}_{2} \) and \( \bar{S}_{3} \).
The other end component is the proper transform \( L_{2}\sptilde \)
of \( L_{2} \) in \( Y \).
We define \( \bar{S} := \bar{S}_{1} + \bar{S}_{2} + \bar{S}_{3} \).
The union \( \bar{G}^{+} \) of \( (-2) \)-curves on \( Y \) is just
the support of the singular fiber of type \( \text{III}^{*} \).
Thus, \eqref{condC:NC1} holds
for \( \bar{B}^{+} = \bar{F} + \bar{G}^{+} + \bar{S} \).
We define \( \bar{G} := \bar{G}^{+} - L_{2}\sptilde \).
Then, \( \bar{G} \) is connected with the dual graph \( \SAA_{7} \).
Thus, \( \det (\bar{G}_{i}\bar{G}_{j}) = -8 \not\equiv 0 \mod \chara(\Bbbk) \),
and hence, we have  \eqref{condC:LinIndep} by
Lemma~\ref{lem:determinant}.
Furthermore, \eqref{condA:1} is true, since
\( Y \setminus (\bar{S} \cup \bar{G})  \) is isomorphic to
\( \BPP^{2} \setminus L_{1} \).

We have \eqref{condC:NC2} and \eqref{condC:-1} on \( Z\)
by looking at Figure~\ref{figure3}.
We define \( B := S + F + G \), i.e., \( J_{1} \) and \( J_{2} \)
are not contained in \( B \).

We take the birational morphism \( \varphi \colon M \to Z \)
so that the total transform \( B_{M}^{+} \)
of \( \bar{B}^{+} \) in \( M \) as in
Figure~\ref{fig:K2=2char3} (cf.\ \cite[Section~6, Fig.~4]{LeePark}).
\begin{figure}[hbtb]
\begin{center}
\includegraphics[scale=1.2]{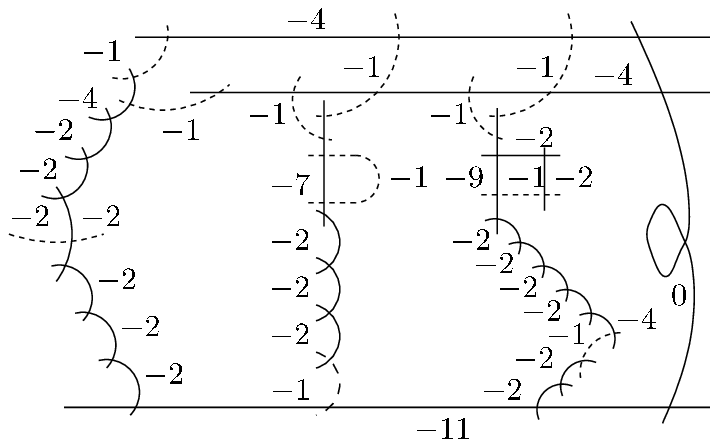}
\end{center}
\caption{The rational surface $M$ in Example~\ref{exam:K2=2char3}}
\label{fig:K2=2char3}
\end{figure}
In particular, \( \rho(M) = \rho(Z) + 20 = 32 \) and \( K_{M}^{2} = -22 \).
Here, we can find a disjoint union \( D \) of the following five linear chains
of smooth rational curves satisfying \eqref{condC:DB}:
\begin{gather*}
D_{1} = \LC(2, 2, 11, 2, 2, 2, 2, 2, 2, 4), \quad
D_{2} = \LC(2, 2, 9, 2, 2, 2, 2, 4), \\
D_{3} = \LC(7, 2, 2, 2), \quad
D_{4} = \LC(4), \quad D_{5} = \LC(4).
\end{gather*}
By Tables~\ref{table:M} and \ref{table:M2}, we see that
the linear chains \( D_{1} \), \ldots, \( D_{5} \)
define toric singularities of types
\( T(1, 25, 17) \), \( T(1, 19, 13) \), \( T(1, 5, 1) \), \( T(1, 2, 1) \),
\( T(1, 2, 1) \), respectively.
Thus, \eqref{condC:classT} holds.
We have \eqref{condC:K2} from the calculation
\begin{align*}
K_{X}^{2} &= K_{M}^{2} + \delta(1, 25, 17) + \delta(1, 19, 13) +
\delta(1, 5, 1) + 2\delta(1, 4, 1) \\
&= -22 + 10 + 8 + 4 + 2 = 2
\end{align*}
using Tables~\ref{table:M} and \ref{table:M2}.
The condition \eqref{condC:end} is checked by looking at
Figure~\ref{fig:K2=2char3}.
We shall prove \eqref{condC:1+}. By Lemma~\ref{lem:condC:1+},
it is enough to check \( \Delta E_{i} > 1 \) for the
\( \varphi \)-exceptional \( (-1) \)-curves
\( E_{1} \), \( E_{2} \), \( E_{3} \) characterized by:
\begin{itemize}
\item  \( E_{1} \) joins the \( (-11) \)-curve in \( D_{1} \)
and the end \( (-2) \)-curve in \( D_{3} \).

\item  \( E_{2} \) joins the \( (-9) \)-curve in \( D_{2} \) and
the end \( (-2) \)-curve of \( D_{2} \).

\item  \( E_{3} \) joins \( (-4) \)-curve in \( D_{2} \) and
the end \( (-2) \)-curve in \( D_{1} \).
\end{itemize}
Then, by Tables~\ref{table:M} and \ref{table:M2}, we have:
\begin{gather*}
\Delta E_{1} = (1 - 1/25) + (1 - 4/5) > 1, \quad
\Delta E_{2} = (1 - 1/19) + (1 - 13/19) > 1, \\
\Delta E_{3} = (1 - 6/19) + (1 - 17/25) > 1.
\end{gather*}
Thus, \eqref{condC:1+} is satisfied.
The condition \eqref{condA:3} is checked by looking at
Figure~\ref{fig:K2=2char3}.
Thus, \( K_{X} \) is ample by \eqref{condA:1}, \eqref{condA:3},
and Proposition~\ref{prop:step:M}.\eqref{prop:step:M:3}.

We shall prove that \( M \setminus D\) is simply connected when
\( \Bbbk = \BCC \) (the proof is omitted in \cite[Section~6]{LeePark}).
We apply Lemma~\ref{lemsub:simplyconn}.
Let \( T(d_{i}, n_{i}, a_{i}) \) be the type of the singular point
defined by \( D_{i} \).
Then, \( d_{1} = \cdots = d_{5} = 1 \) and \( (n_{1}, \ldots, n_{5}) =
(25, 19, 5, 2, 2)\). Looking at
Figure~\ref{fig:K2=2char3}, we have:
\begin{itemize}
\item  the \( (-1) \)-curve \( E_{3} \) which meets
end components of \( D_{1} \) and \( D_{2} \).

\item a \( (-1) \)-curve meeting the end \( (-4) \)-component of \( D_{1} \) and
\( D_{4} \) (resp.\ \( D_{5} \)).

\item a \( (-1) \)-curve meeting the end \( (-7) \)-component of
\( D_{3} \) and \( D_{4} \).
\end{itemize}
Since \( \gcd(n_{1}, n_{2}) = \gcd(25, 19) = 1 \),
\( \gcd(n_{1}, n_{4}) = \gcd(n_{1}, n_{5}) = \gcd(25, 2) = 1\), and
\( \gcd(n_{3}, n_{4}) = \gcd(5, 2) = 1 \),
the conditions of Lemma~\ref{lemsub:simplyconn} are all satisfied, and
hence \( M \setminus D \) is simply connected.
Therefore, we have done all the tasks.
\end{exam}

\begin{remsub}\label{remsub:exam:K2=1,2char3}
It is known by \cite{JLRRS} that there exist rational elliptic surfaces whose
singular fibers are of configuration type
\( (\text{III}^{*}, \text{I}_{1}, \text{I}_{1}, \text{I}_{1}) \)
in characteristic 3.
\end{remsub}

\begin{exam}
    \label{exam:K2=1Main}
Assume that \( \chara(\Bbbk) \ne 2 \) and set \( K^{2} = 1 \).
We consider the same cubic pencil \( \Phi \) generated by \( \Phi_{0} \) and
\( \Phi_{\infty} \) in Example~\ref{exam:K2=2char3}. Thus,
\eqref{condC:Phi} and \eqref{condC:nodal} hold.
On the elliptic fibration \( Y \to \BPP^{1}_{\Bbbk} \),
let \( \bar{F} = \bar{F}_{1} + \bar{F}_{2}\) and
\( \bar{S} = \bar{S}_{1}+ \bar{S}_{2} + \bar{S}_{3} \)
be the same as in Example~\ref{exam:K2=2char3}, but
here we define \( \bar{G} \) to be \( \bar{G}^{+} \) minus
two irreducible components of the singular fiber of type
\( \text{III}^{*} \); one is \( L_{2}\sptilde \)
and the other is the component next to \( L_{1}\sptilde \).
Then, \eqref{condC:NC1} and \eqref{condC:LinIndep} are also satisfied.
For the blowing up \( V \to \BPP^{2} \) at
\( (0:1:0) = L_{1} \cap L_{2} \cap \Phi_{0}\) and
for the proper transform \( L_{1}' \) of \( L_{1} \) in \( V \),
we see that \( Y \setminus (\bar{S} \cup \bar{G}) \)
is isomorphic to \( V \setminus L_{1}' \) minus a point, which
is the image of \( \bar{S}_{1} \).
In particular, the condition \eqref{condA:1} is not satisfied.
However, any projective curve contained in
\( Y \setminus (\bar{S} \cup \bar{G}) \) is a non-singular rational curve having
the self-intersection number zero and is mapped to a line of \( \BPP^{2}_{\Bbbk} \)
passing through \( (0:1:0) \).

We have \eqref{condC:NC2} and \eqref{condC:-1} by Figure~\ref{figure3}.
We define \( B := S + G + F \).

We take the birational morphism \( \varphi \colon M \to Z \) by
the following steps:
\eqref{exam:K2=1Main:step1}--\eqref{exam:K2=1Main:step3}:
\begin{enumerate}
    \renewcommand{\theenumi}{\roman{enumi}}
    \renewcommand{\labelenumi}{(\theenumi)}
\item  \label{exam:K2=1Main:step1}
First, we blow up at the two intersection points
\( F_{1} \cap S_2\) and \( F_{1} \cap S_3\).

\item \label{exam:K2=1Main:step2}
We blow up \( 4 \) times successively at the
intersection point of \( F_{1} \) and the proper transform of \( S_{1} \).
Then, we have a linear chain \( \LC(7, 2, 2, 2) \) of rational curves.

\item \label{exam:K2=1Main:step3}
Applying the same process in \eqref{exam:K2=1Main:step1} and
\eqref{exam:K2=1Main:step2} to \( F_{2} \),
we obtain a rational surface \( M \) with \( \rho(M) = \rho(Z) + 12 = 24\)
and \( K_{M}^{2} = -14 \).
\end{enumerate}
Then, the total transform \( B_{M}^{+} \) of \( B^{+} \) is as in
Figure~\ref{fig:K2=1Main}.

\begin{figure}[htb]
\begin{center}
\includegraphics[scale=0.95]{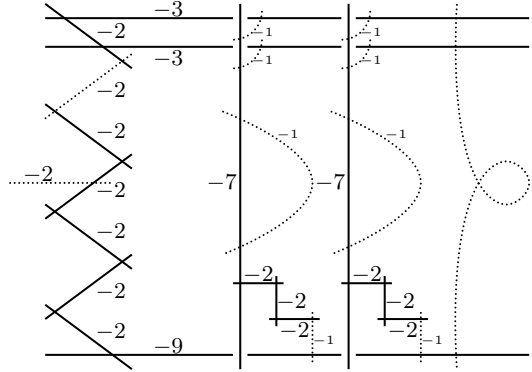}
\end{center}
\caption{The rational surface \( M \) in Example~\ref{exam:K2=1Main}}
\label{fig:K2=1Main}
\end{figure}
Here, we see that \( M \) has a disjoint union \( D \) of the following
four linear chains of smooth rational curves:
\begin{gather*}
D_{1} = \LC(7, 2, 2, 2) \text{ over } F_{1}, \quad
D_{2} = \LC(7, 2, 2, 2) \text{ over } F_{2}, \\
D_{3} = \LC(3, 2, 3), \quad D_{4} = \LC(9, 2, 2, 2, 2, 2).
\end{gather*}
Then, \eqref{condC:DB} and \eqref{condC:classT} hold: In fact,
\( D_{1} \), \ldots, \( D_{4} \) define toric singularities of types
\( T(1, 5, 1) \), \( T(1, 5, 1) \), \( T(3, 2, 1) \), \( T(1, 7, 1) \),
respectively (cf.\ Table~\ref{table:M}).
We have \eqref{condC:K2} by the calculation
\[ K_{X}^{2} = K_{M}^{2} +
2\delta(1, 5, 1) + \delta(3, 2, 1) + \delta(1, 7, 1) =
-14 + 8 + 1 + 6 = 1\]
using Table~\ref{table:M}.
We can check \eqref{condC:end} and \eqref{condA:3} by Figure~\ref{fig:K2=1Main}.
We shall prove \eqref{condC:1+}. By Lemma~\ref{lem:condC:1+},
it is enough to show \( \Delta E_{i} > 1 \) for \( i = 1 \), \( 2 \)
for the \( (-1) \)-curves \( E_{1} \), \( E_{2} \) such that
\( E_{i} \) joins the \( (-9) \)-curve of \( D_{4} \) and
the end \( (-2) \)-curve of \( D_{i} \) for \( i = 1 \), \( 2 \).
By using Table~\ref{table:M},
we have
\[ \Delta E_{1} = \Delta E_{2} = (1 - 1/7) + (1 - 4/5) > 1. \]
Thus, \eqref{condC:1+} holds.
The ampleness of \( K_{X} \) is proved as follows.
Assuming the contrary, we have a \( (-2) \)-curve \( \Gamma \) contained in
\( M \setminus D \) by Proposition~\ref{prop:step:M}.\eqref{prop:step:M:3}.
Then, \( \Gamma \) is not \( \varphi \)-exceptional by \eqref{condA:3}, and
the projective curve \( \bar{\Gamma} := \tau(\varphi(\Gamma)) \) is contained in
\( Y \setminus (\bar{G} \cup \bar{S}) \).
By the property of this open subset of \( Y \) discussed above,
\( \bar{\Gamma} \) is a non-singular rational curve with self-intersection number zero
and the image of \( \bar{\Gamma} \) in \( \BPP^{2}_{\Bbbk} \)
is a line \( L \) passing through \( (0:1:0) \).
Since \( \varphi(\Gamma) \) is a \( (-2) \)-curve,
\( \bar{\Gamma} \) is a bisection of \( Y \to \BPP^{1} \) passing through
the nodes of \( \bar{F}_{1} \) and \( \bar{F}_{2} \).
Hence, \( L \) is a line passing through the nodes of \( \Phi_{c_{\pm}} \);
thus, \( L = (\ytt)_{0} \). This contradicts \( L \ni (0:1:0) \).
Therefore, \( K_{X} \) is ample.

We shall show that \( M \setminus D \) is simply connected
when \( \Bbbk = \BCC \).
We apply Lemma~\ref{lemsub:simplyconn}.
Let \( T(d_{i}, n_{i}, a_{i}) \) be the type of \( D_{i} \)
for \( 1 \leq i \leq 4 \).
Then, \( (d_{1}, \ldots, d_{4}) = (1, 1, 3, 1) \) and
\( (n_{1}, \ldots, n_{4}) = (5, 5, 2, 7) \).
Looking at Figure~\ref{fig:K2=1Main},
we have:
\begin{itemize}
\item  the \( (-1) \)-curve \( E_{4} \) which meets end components of
\( D_{1} \) and \( D_{4} \).

\item  the \( (-1) \)-curve \( E_{2} \) which meets end components of
\( D_{2} \) and \( D_{4} \).

\item  a \( (-1) \)-curve meeting the end \( (-3) \)-component of \( D_{3} \)
and the end \( (-7) \)-component of \( D_{1} \).
\end{itemize}
Since \( \gcd(d_{1}n_{1}, d_{4}n_{4}) = \gcd(d_{2}n_{2}, d_{4}n_{4}) =
\gcd(5, 7) = 1 \) and
\( \gcd(d_{1}n_{1}, d_{3}n_{3}) = \gcd(5, 6)
\linebreak 
= 1 \),
the conditions of Lemma~\ref{lemsub:simplyconn} are all satisfied, and
hence \( M \setminus D \) is simply connected.
Therefore, we have done all the tasks.
\end{exam}

\begin{exam}
    \label{exam:K2=3}
Assume that \( \chara(\Bbbk) \ne 2 \), and we set \( K^{2} = 3 \).
We follow the construction in \cite[Section 3]{PPS1}.
We set
\[ \phi_{0} := (\xtt\ytt + \ztt^{2})(\xtt + \ytt) \quad \text{and} \quad
\phi_{\infty} := \xtt\ytt\ztt.\]
Then, \( \Phi_{0} = Q + L_{4} \) and
\( \Phi_{\infty} = L_{1} + L_{2} + L_{3} \) for
the conic \( Q = (\xtt\ytt + \ztt^{2})_{0} \), and
the lines
\( L_{1} = (\xtt)_{0} \), \( L_{2} = (\ytt)_{0} \), \( L_{3} = (\ztt)_{0} \), and
\( L_{4} = (\xtt + \ytt)_{0} \).
In particular, \eqref{condC:Phi} holds, and
moreover, \eqref{condC:nodal} holds: In fact, for \( c \ne 0 \),
the divisor \( \Phi_{c} = (\phi_{0} + c\phi_{\infty})_{0}\) is singular if and only if
\( c = \pm 4 \), where \( \Phi_{4} \) and \( \Phi_{-4} \) are nodal rational
curves with the nodes at \( (1:1:-1) \) and \( (1:1:1) \), respectively.
On the minimal elliptic fibration \( \pi \colon Y \to \BPP^{1}_{\Bbbk} \) defined
by \( \Phi \), the configuration type of singular fibers is
\( (\text{I}_{8}, \text{I}_{2}, \text{I}_{1}, \text{I}_{1}) \), and
the \( (-1) \)-curves exceptional for \( Y \to \BPP^{2}_{\Bbbk} \)
are mutually disjoint four sections \( \bar{S}_{1} \), \ldots,
\( \bar{S}_{4} \) as in Figure~\ref{fig:K2=3}
(cf.\ \cite[Section 3, Figure~2]{PPS1} for \( E(1) \)
(where the symbols are different from ours)).
\begin{figure}[hbtb]
\begin{center}
\includegraphics[scale=1]{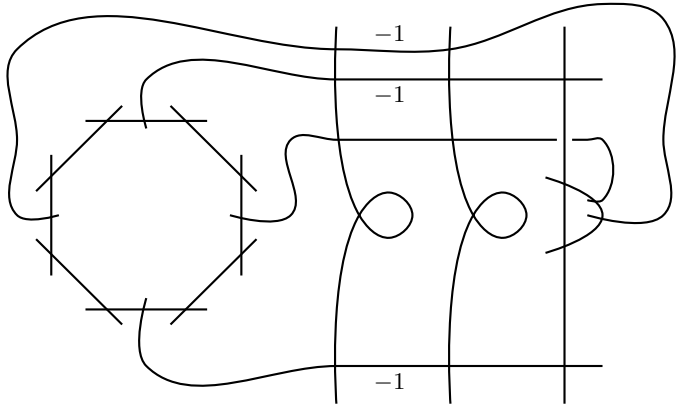}
\end{center}
\caption{The rational elliptic surface $Y$ in
Example~\ref{exam:K2=3}} \label{fig:K2=3}
\end{figure}
Let \( \bar{F}_{1} \) and \( \bar{F}_{2} \) be the proper transforms of
\( \Phi_{4} \) and \( \Phi_{-4} \), respectively.
These are the two singular fibers of type \( \text{I}_{1} \).
We define \( \bar{F} := \bar{F}_{1} + \bar{F}_{2} \).
The singular fiber of type \( \text{I}_{2} \) is the union of
the proper transforms \( Q\sptilde \) and \( L_{4}\sptilde \)
of \( Q \) and \( L_{4} \) in \( Y \).
The singular fiber of type \( \text{I}_{8} \) contains the proper transforms
\( L_{1}\sptilde \), \( L_{2}\sptilde \), and \( L_{3}\sptilde \)
of the lines \( L_{1}  \), \( L_{2} \),
and \( L_{3} \) in \( Y \).
The irreducible components of the singular fiber of type \( \text{I}_{8} \) is
labelled as
\( \Gamma_{0} + \Gamma_{1} + \cdots + \Gamma_{7}\)
as a cyclic chain, where we set \( L_{1}\sptilde = \Gamma_{1} \),
\( L_{2}\sptilde = \Gamma_{7}\), and
\( L_{3}\sptilde = \Gamma_{4}\).
We may assume that \( \bar{S}_{1} \), \( \bar{S}_{2} \), \( \bar{S}_{3} \), and
\( \bar{S}_{4} \) are contracted to the points
\( (0:0:1) \), \( (0:1:0) \), \( (1:0:0) \), and \( (1:-1:0) \), respectively,
where
\( L_{1} \cap L_{2} \cap L_{4} = \{(0:0:1)\} \),
\( L_{1} \cap Q = L_{1} \cap  L_{3} = \{(0:1:0)\} \),
\( L_{2} \cap Q = L_{2} \cap L_{3} = \{(1:0:0)\} \), and
\( L_{3} \cap L_{4} = \{(1 : -1 :0)\} \).
Hence, \( \bar{S}_{1} \) intersects \( L_{4}\sptilde \) and
\( \Gamma_{0} \);
\( \bar{S}_{2} \) intersects \( Q\sptilde \) and \( \Gamma_{2} \);
\( \bar{S}_{3} \) intersects \( Q\sptilde \) and \( \Gamma_{6} \);
\( \bar{S}_{4} \) intersects
\( L_{4}\sptilde \) and \( L_{3}\sptilde = \Gamma_{4}\).
We define \( \bar{S} := \bar{S}_{1} + \bar{S}_{2} + \bar{S}_{4} \)
removing \( \bar{S}_{3} \).
Then, the condition \eqref{condC:NC1} holds.
We define \( \bar{G} \) to be
\[ \bar{G}^{+} - \Gamma_{3} - Q\sptilde =
\sum\nolimits_{0 \leq i \leq 7, \, i \ne 3} \Gamma_{i} + L_{4}\sptilde. \]
Then, \( \bar{G} \) has two connected components whose dual graphs are
\( \SAA_{7} \) and \( \SAA_{1} \).
Thus, \( \det (\bar{G}_{i}\bar{G}_{i}) = 16 \not\equiv 0 \mod \chara(\Bbbk) \),
and hence \eqref{condC:LinIndep} is satisfied by Lemma~\ref{lem:determinant}.
For \( \eqref{condA:1} \), since \( \rho(Y) = 10 \),
it is enough to prove that \( \bar{S}_{1} \),
\( \bar{S}_{2} \), and the irreducible components of \( \bar{G} \) are
linearly independent in \( \Pic(Y) \).
Assume that
\[ a_{1}\bar{S}_{1} + a_{2}\bar{S}_{2} + a_{3}L_{4}\sptilde +
\sum\nolimits_{0 \leq i \leq 7, \, i \ne 3} m_{i}\Gamma_{i} \sim 0 \]
for integers \( a_{j} \) and \( m_{i} \).
Then, calculating the intersection numbers with
\( Q\sptilde \), \( L_{4}\sptilde \), \( \bar{S}_{2} \), \( \bar{S}_{4} \), and
\( \Gamma_{3} \), we have:
\[
a_{2} + 2a_{3} = a_{1} - 2a_{3} = -a_{2} + m_{2} = a_{3} + m_{4}
= m_{2} + m_{4} = 0.
\]
In particular, \( a_{1} = a_{2} = a_{3} = m_{2} = m_{4} = 0 \).
The other \( m_{j} \) are all zero, since \( \det (\bar{G}_{i}\bar{G}_{j}) \ne 0 \).
Thus, we have the linear independence, and hence \eqref{condA:1}.

We have \eqref{condC:NC2} and \eqref{condC:-1} on \( Z \) by
looking at Figure~\ref{fig:K2=3}.
We define \( B := S + G + F \).
\begin{figure}[hbtb]
\begin{center}
\includegraphics[scale=1]{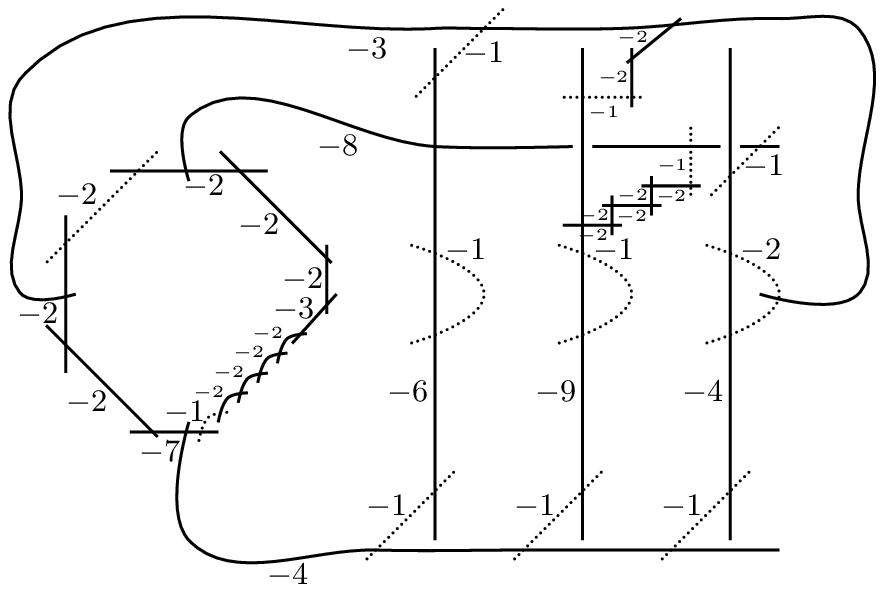}
\end{center}
\caption{The rational surface $M$ in Example~\ref{exam:K2=3}}
\label{fig:MforK2=3}
\end{figure}
We take the birational morphism \( \varphi \colon M \to Y \) so that
the total transform \( B_{M}^{+} \) of \( \bar{B}^{+} \) in \( M \)
as in
Figure~\ref{fig:MforK2=3} (cf.\ \cite[Section 3, Figure~5]{PPS1}).
In particular, \( \rho(M) = \rho(Z) + 19 \) and \( K_{M}^{2} = -21 \).
We have a disjoint union
\( D \) of the following linear chains of smooth rational curves in
\( \varphi^{-1}(B) \):
\begin{gather*}
D_{1} = \LC(6, 8, 2, 2, 2, 3, 2, 2, 2, 2), \quad D_{2} = \LC(4), \\
D_{3} = \LC(4, 7, 2, 2, 3, 2, 2), \quad D_{4} = \LC(9, 2, 2, 2, 2,
2).
\end{gather*}
Then, \eqref{condC:DB} holds. Moreover, \eqref{condC:classT} and \eqref{condC:K2} hold:
In fact, by Tables~\ref{table:M} and \ref{table:M2}, we see that
\( D_{1} \), \ldots \( D_{4} \) define toric singularities
of type \( T(1, 35, 6) \), \( T(1, 2, 1) \), \( T(1, 15, 9) \),
\( T(1, 7, 1) \), respectively, and we calculate
\[ K_{X}^{2} = K_{M}^{2} + \delta(1, 35, 6) + \delta(1, 2, 1) +
\delta(1, 15, 9) + \delta(1, 7, 1) =
-21 + 10 + 1 + 7 + 6 = 3.\]
The conditions \eqref{condC:end} and \eqref{condA:3}
follow immediately from Figure~\ref{fig:MforK2=3}.
We shall show \eqref{condC:1+} using Lemma~\ref{lem:condC:1+}.
Then, it suffices to prove \( \Delta E_{i} > 1 \) for the \( (-1) \)-curves
\( E_{1} \), \( E_{2} \), and \( E_{3} \) determined by:
\begin{itemize}
\item  \( E_{1} \) joins the \( (-7) \)-curve in \( D_{3} \) and
the end \( (-2) \)-component of \( D_{1} \).

\item  \( E_{2} \) joins the \( (-9) \)-curve in \( D_{4} \) and
the end \( (-2) \)-component of \( D_{3} \).

\item  \( E_{3} \) joins the \( (-8) \)-curve in \( D_{1} \) and
the end \( (-2) \)-component of \( D_{4} \).
\end{itemize}
Then, by Tables~\ref{table:M} and \ref{table:M2}, we have:
\begin{gather*}
\Delta E_{1} = (1 - 1/19) + (1 - 29/35) > 1, \quad
\Delta E_{2} = (1 - 1/7) + (1 - 14/19) > 1, \quad \\
\Delta E_{3} = (1 -1/35 ) + (1 - 6/7) > 1.
\end{gather*}
This establishes \eqref{condC:1+}.
The ampleness of \( K_{X} \) follows from \eqref{condA:1} and \eqref{condA:3}
by Proposition~\ref{prop:step:M}.\eqref{prop:step:M:3}.
The simply connectedness of \( M \setminus D \) has been shown in
the proof of \cite[Theorem~3.1]{PPS1}.
Thus, we have done all the tasks.
\end{exam}

\begin{remsub}
It is known by \cite{JLRRS} that there exist rational elliptic surfaces whose singular fibers are of
configuration type \( (\text{I}_{8}, \text{I}_{1}, \text{I}_{1}, \text{I}_{2}) \)
in characteristic 3.
In characteristic 2, this configuration does not exist by \cite{Lang}.
\end{remsub}

\begin{exam}\label{exam:K2=4}
Assume that \( \chara(\Bbbk) \ne 2 \), and we set \( K^{2} = 4 \).
We follow the construction in \cite[Section~2]{PPS2}.
Here, we define \( \phi_{0} \) and \( \phi_{\infty} \) to be the same as
in Example~\ref{exam:K2=3}.
Thus, \eqref{condC:Phi} and \eqref{condC:nodal} hold.
On the elliptic fibration \( \pi \colon Y \to \BPP^{1}_{\Bbbk} \),
we have four sections \( \bar{S}_{1} \), \ldots, \( \bar{S}_{4} \) in
Example~\ref{exam:K2=3}, but here, consider another horizontal curve
\( N\sptilde \) which is the proper transform of the line
\( N := (\xtt + \ytt + 2\ztt)_{0} \) in \( Y \).
Note that \( N \) contains \( (1:-1:0) = L_{3} \cap L_{4} \)
and the node \( (1:1:-1) \) of \( \Phi_{4} \).
We define \( \bar{S} \) to be \( \bar{S}_{3} + \bar{S}_{4} + N\sptilde \).
We also set \( \bar{F} = \bar{F}_{1} + \bar{F}_{2} \) as in
Example~\ref{exam:K2=3}.
Then, we can check \eqref{condC:NC1} by
Figure~\ref{fig:K2=4} (cf.\ \cite[Figure~2]{PPS2}).
\begin{figure}[hbtb]
\begin{center}
\includegraphics[scale=1]{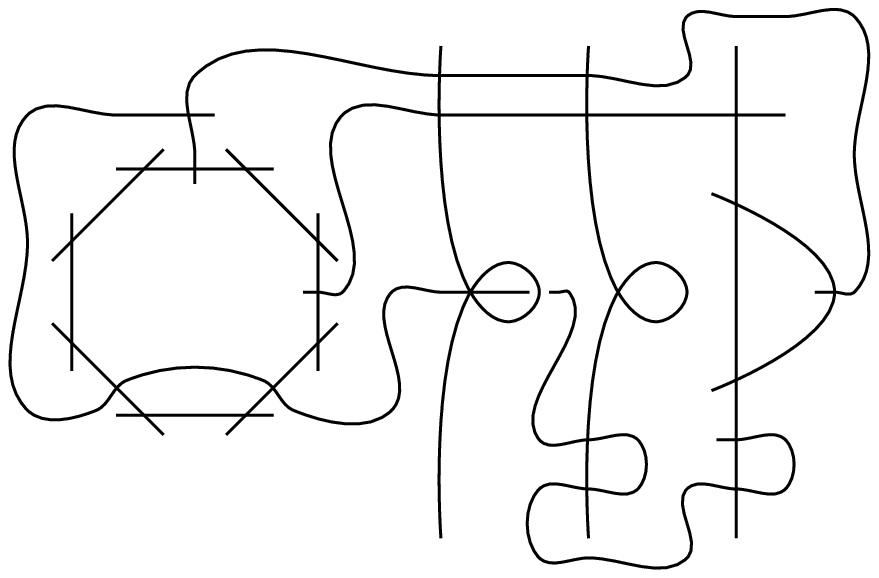}
\end{center}
\caption{The rational elliptic surface $Y$ in
Example~\ref{exam:K2=4}} \label{fig:K2=4}
\end{figure}
We define the divisor \( \bar{G} \) to be
\[ \bar{G}^{+} - \Gamma_{0} - L_{2}\sptilde - Q\sptilde =
\sum\nolimits_{i = 1}^{6}\Gamma_{i} + L_{4}\sptilde. \]
Note that \( \Gamma_{1} = L_{1}\sptilde \), \( \Gamma_{4} = L_{3}\sptilde \),
and \( \Gamma_{7} = L_{2}\sptilde \).
We shall prove \eqref{condC:LinIndep} not using Lemma~\ref{lem:determinant}.
Assume that we have a linear equivalence relation
\[ \sum\nolimits_{i = 1}^{6} m_{i} \bar{G}_{i} + m_{7}L_{4}\sptilde + m_{8}\bar{F}_{1}
\sim pH\]
for integers \( m_{1} \), \ldots, \( m_{8} \), \( p = \chara(\Bbbk) \), and
a Cartier divisor \( H \) on \( Y \).
Considering the intersection numbers with \( Q\sptilde \) and
\( \bar{S}_{j} \), we have
\[
pH Q\sptilde = m_{7}, \quad  pH \bar{S}_{1} = m_{7} + m_{8}, \quad
pH\bar{S}_{2} = m_{2} + m_{8}, \quad \text{and} \quad
pH\bar{S}_{3} = m_{6} + m_{8}, \]
which imply \( m_{2} \equiv m_{6} \equiv m_{7} \equiv m_{8} \equiv 0 \mod p\).
Moreover, we have
\( m_{i-1} + m_{i+1} \equiv 2m_{i} \mod p\) for \( 2 \leq i \leq 6 \)
by calculating \( pH \Gamma_{i} \).
Thus, \( m_{i} \equiv 0 \mod p\) for any \( i \), since \( p \ne 2 \).
Hence, \eqref{condC:LinIndep} holds.
The conditions \eqref{condA:2}, \eqref{condC:NC2}, and \eqref{condC:-1}
follow immediately from
Figure~\ref{fig:K2=4}.
We define \( B := S + G + F + J_{1}\), where we set \( \bar{F}_{1}
\) to be the proper transform of \( \Phi_{4} \) in \( Y \). Note
that \( N\sptilde \) passes through the node of \( \bar{F}_{1} \).

We take the birational morphism \( \varphi \colon M \to Y \)
by the same process as in \cite[Section~2]{PPS2}
for constructing \( Z = Y\sharp 9\bar{\BPP}^{2}\)
from \( Y \), where the symbol \( Y\sharp 9\bar{\BPP}^{2} \) in \cite{PPS2}
stands for a blown up surface of \( Y \) at nine (infinitely near) points.
Thus, \( \rho(M) = \rho(Z) + 7 = 19 \), \( K_{M}^{2} = -9 \), and
the total transform \( B_{M}^{+} \)
of \( \bar{B}^{+} \) in \( M \) has a configuration
in Figure~\ref{fig:MforK2=4} (cf.\ \cite[Figure~3]{PPS2}).
\begin{figure}[hbtb]
\begin{center}
\includegraphics[scale=1]{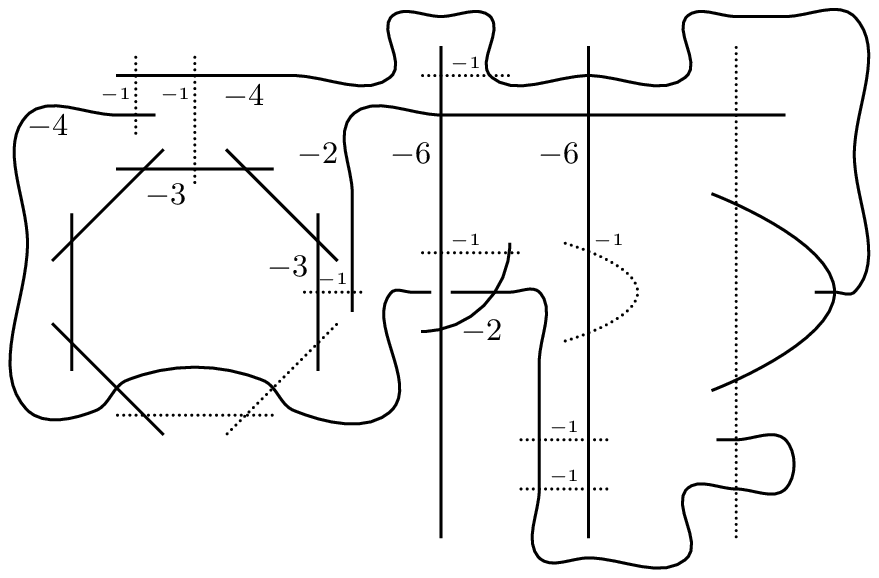}
\end{center}
\caption{The rational surface $M$ in Example~\ref{exam:K2=4}}
\label{fig:MforK2=4}
\end{figure}
Here, we can find a linear chain
\[ D = \LC(2, 4, 6, 2, 6, 2, 4, 2, 2, 2, 3, 2, 3)\]
of smooth rational curves satisfying \eqref{condC:DB},
which corresponds to the solid lines of
Figure~\ref{fig:MforK2=4}.
By Table~\ref{table:M2}, we see that \( D \) defines a toric singularity of
type \( T(1, 252, 145) \), and
\[ K_{M}^{2} = K_{X}^{2} + \delta(1, 252, 145) = -9 + 13 = 4. \]
Thus, \eqref{condC:classT} and \eqref{condC:K2} hold.
The conditions \eqref{condC:end} and
\eqref{condA:3} follow immediately from
Figure~\ref{fig:MforK2=4}.
We shall prove \eqref{condC:1+} by using Lemma~\ref{lem:condC:1+}.
Then, it suffices to show \( \Delta E > 1 \) for the \( (-1) \)-curve
\( E \) which
joins the end \( (-2) \)-component and the end \( (-3) \)-component
of \( D \). We have
\[\Delta E = (1 - 145/252) + (1 - 107/252) > 1 \]
by Table~\ref{table:M2}. Thus, \eqref{condC:1+} holds.
The ampleness of \( K_{X} \) follows from \eqref{condA:2} and
\eqref{condA:3} by Proposition~\ref{prop:step:M}.\eqref{prop:step:M:3}, since
\( B \supset J_{1} \).
The simply connectedness of \( M \setminus D \) has been shown in
the proof of \cite[Proposition~2.1]{PPS2}.
Thus, we have done all the tasks.
\end{exam}

\begin{exam}
    \label{exam:char2,K2=1}
Assume that \( \chara(\Bbbk) \ne 3 \) and we set \( K^{2} = 1 \).
We define
\[ \phi_{0} = \xtt^{2}\ytt + \ytt^{2}\ztt + \ztt^{2}\xtt
\quad \text{and}
\quad \phi_{\infty} = 3\xtt\ytt\ztt.\]
Then, \( \Phi_{\infty} = (\phi_{\infty})_{0} = L_{1} + L_{2} + L_{3}\)
for the lines \( L_{1} = (\xtt)_{0} \), \( L_{2} = (\ytt)_{0} \),
\( L_{3} = (\ztt)_{0} \), and
\( \Phi_{0} = (\phi_{0})_{0} \) is a smooth cubic curve such that
\[ \Phi_{0}|_{L_{1}} = P_{2} + 2P_{3}, \quad
\Phi_{0}|_{L_{2}} = 2P_{1} + P_{3}, \quad
\Phi_{0}|_{L_{3}} = P_{1} + 2P_{2} \]
where \( P_{1} = (1:0:0) \), \( P_{2} = (0:1:0) \), and \( P_{3} = (0:0:1) \).
In particular, \eqref{condC:Phi} holds, and
moreover, \eqref{condC:nodal} holds: In fact, since \( \chara(\Bbbk) \ne 3 \),
for \( c \ne 0 \),
\( \Phi_{c} = (\phi_{0} + c\phi_{\infty})_{0} \) is singular if and only if
\( c = -\omega^{i} \) for \( i = 0 \), \( 1 \), \( 2 \), where
\( \omega \) is a primitive cubic root of \( 1 \), and
\( \Phi_{-\omega^{i}} \) is a nodal rational curve with the node at
\( (1:\omega^{i}:\omega^{-i}) \).
We also define \( N \) to be the line passing through
\( (0:1:0) = L_{1} \cap L_{3} \)
and the node \( (1:1:1) \) of \( \Phi_{-1} \). Hence,
\( N = (\xtt - \ztt)_{0} \). Thus, we have Figure~\ref{figure6}.
\begin{figure}[hbtb]
\begin{center}
\includegraphics[scale=0.9]{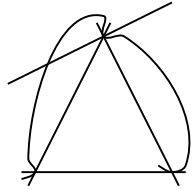}\end{center}
\caption{A pencil of cubics in Examples~\ref{exam:char2,K2=1} and
\ref{exam:char2,K2=3}}
\label{figure6}
\end{figure}
On the minimal elliptic fibration \( \pi \colon Y \to \BPP^{1}_{\Bbbk} \)
defined by \( \Phi \), the configuration type of singular fibers is
\( (\text{I}_{9}, \text{I}_{1}, \text{I}_{1}, \text{I}_{1}) \),
and the \( (-1) \)-curves exceptional for \( Y \to \BPP^{2}_{\Bbbk} \)
are mutually disjoint three sections \( \bar{S}_{1} \), \( \bar{S}_{2} \), and
\( \bar{S}_{3} \) as in Figure~\ref{figure7}.
\begin{figure}[hbtb]
\begin{center}
\includegraphics[scale=0.85]{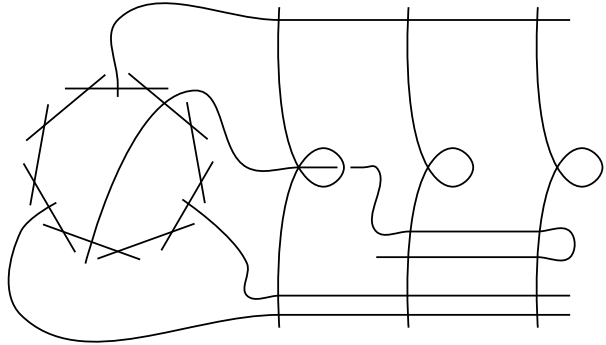}
\end{center}
\caption{The rational elliptic surface $Y$ in %
Examples~\ref{exam:char2,K2=1} and \ref{exam:char2,K2=3}}
\label{figure7}
\end{figure}
Here, the singular fibers of type \( \text{I}_{1} \) are the proper transforms of
\( \Phi_{-\omega^{i}} \) for \( i = 0 \), \( 1 \), \( 2 \).
We set \( \bar{F}_{1} \) and \( \bar{F}_{2} \) to be the proper transforms of
\( \Phi_{-1} \) and \( \Phi_{-\omega} \), respectively, and define
\( \bar{F} := \bar{F}_{1} + \bar{F}_{2} \).
The irreducible components of the singular fiber of type \( \text{I}_{9} \) are
labeled as \( \Gamma_{1} + \cdots + \Gamma_{9} \) as a cyclic chain of
smooth rational curves,
where \( \Gamma_{3i} \) is the proper transform of \( L_{i}\sptilde \)
for \( i = 1 \), \( 2 \), \( 3 \).
We may assume that \( \bar{S}_{j} \) is contracted to \( P_{j} \) by
\( Y \to \BPP^{2}_{\Bbbk} \) for \( j = 1 \), \( 2 \), \( 3 \).
Then,
\[ \bar{S_{j}} \Gamma_{i} =
\begin{cases}
1, & \text{ if } i \equiv 3j + 4 \mod 9, \\
0, & \text{otherwise}.
\end{cases}
\]

The proper transform \( N\sptilde \) of \( N \) in \( Y \) is
a bisection of \( \pi \colon Y \to \BPP^{1}_{\Bbbk} \)
with self-intersection number zero
and passing through the node of \( \bar{F}_{1} \) but no nodes of other
singular fibers of type \( \text{I}_{1} \).
We have
\[ N\sptilde \bar{S}_{j} = N\sptilde \Gamma_{2} -1
= N\sptilde \Gamma_{6} -1 = N\sptilde \Gamma_{i} = 0\]
for any \( 1 \leq j \leq 3\) and \( 1 \leq i \ne 2, 6 \leq 9 \).
We define \( \bar{S} := \bar{S}_{2} + N\sptilde\).
Since the union \( \bar{G}^{+} \) of all the \( (-2) \)-curves on
\( Y \) are the singular fiber of type \( \text{I}_{9} \),
we have \eqref{condC:NC1}.
We define
\[ \bar{G} := \bar{G}^{+} - \Gamma_{3} - \Gamma_{4} - \Gamma_{5}
- \Gamma_{8} - \Gamma_{9}
= \bar{\Gamma}_{1} + \bar{\Gamma}_{2} + \bar{\Gamma}_{6} + \bar{\Gamma}_{7}. \]
Then, \( \bar{G} \) has two connected components with the dual graph
\( \SAA_{2} \). Hence,
\( \det (\bar{G}_{i}\bar{G}_{j}) \not\equiv 0 \mod \chara(\Bbbk) \),
and \eqref{condC:LinIndep} holds by Lemma~\ref{lem:determinant}.
The condition \eqref{condA:2} fails:
In fact, \( \Gamma_{4} \subset Y \setminus (\bar{S} \cup \bar{G}) \).
Thus, it is impossible to require \( K_{X} \) to be ample by
Proposition~\ref{prop:step:M}.\eqref{prop:step:M:3}.
The conditions \eqref{condC:NC2} and \eqref{condC:-1} on \( Z \)
follow immediately from Figure~\ref{figure7}.
We define \( B := S + G + F + J_{1} \). Note that \( N\sptilde \) passes through
the node of \( \bar{F}_{1} \).

We take the birational morphism \( \varphi \colon M \to Z \) so that
the total transforms of \( B + J_{2} \) and \( \bar{G}^{+} \) in \( M \)
form a configuration of curves as in Figure~\ref{fig:K^2=1char2}.
\begin{figure}[hbtb]
\begin{center}
\includegraphics[scale=0.95]{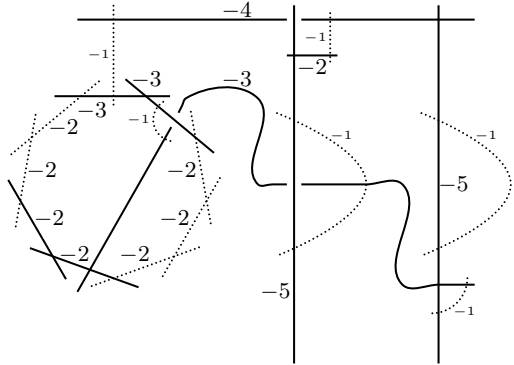}
\end{center}
\caption{The rational surface \( M \) in Example~\ref{exam:char2,K2=1}}
\label{fig:K^2=1char2}
\end{figure}
Then, \( \rho(M) = \rho(Z) + 5 = 17\) and \( K_{M}^{2} = -7 \).
Here, we have a disjoint union \( D \) of the following three
linear chains of rational curves satisfying \eqref{condC:DB}:
\[ D_{1} = \LC(4, 5, 3, 2, 2), \quad
D_{2} = \LC(5, 2), \quad
D_{3} = \LC(3, 3).\]
By Tables~\ref{table:M} and \ref{table:M2},
\( D_{1} \), \( D_{2} \), and \( D_{3} \) define
toric singularities of type
\( T(1, 11, 3) \), \( T(1, 3, 1) \), and \( T(2, 2, 1) \), respectively, and
\[ K_{X}^{2} = K_{M}^{2} + \delta(1, 11, 3) + \delta(1, 3, 1) + \delta(2, 2, 1)
= -7 + 5 + 2 + 1 = 1. \]
Thus, \eqref{condC:classT} and \eqref{condC:K2} hold.
The condition \eqref{condC:end} follows immediately from
Figure~\ref{fig:K^2=1char2}.
We shall prove \eqref{condC:1+} by using Lemma~\ref{lem:condC:1+}.
Then, it suffices to show \( \Delta E > 1 \)
for the \( (-1) \)-curve \( E \)
which joins the end \( (-4) \)-component of \( D_{1} \) and
the end \( (-2) \)-component of \( D_{2} \).
By Tables~\ref{table:M} and \ref{table:M2}, we have
\[ \Delta E = (1 - 3/11) + (1 - 2/3) > 1.\]
Hence, \eqref{condC:1+} follows.
We shall prove
the simply connectedness of \( M \setminus D \) in case \( \Bbbk = \BCC \)
by applying Lemma~\ref{lemsub:simplyconn}.
We have \( (d_{1}, d_{2}, d_{3}) = (1, 1, 2) \) and
\( (n_{1}, n_{2}, n_{3}) = (11, 3, 2) \) for the type
\( T(d_{i}, n_{i}, a_{i}) \)
of the singularity defined by \( D_{i} \) for \( 1 \leq i \leq 3 \).
Looking at Figure~\ref{fig:K^2=1char2}, we have the following \( (-1) \)-curves:
\begin{itemize}
\item  the \( (-1) \)-curve \( E \) which meets end components of \( D_{1} \) and \( D_{2} \).

\item  a \( (-1) \)-curve meeting the end \( (-4) \)-component of \( D_{1} \)
and an end \( (-3) \)-component of \( D_{3} \).
\end{itemize}
Since \( \gcd(d_{1}n_{1}, d_{2}n_{2}) = \gcd(11, 3) = 1 \)
and \( \gcd(d_{1}n_{1}, d_{3}n_{3}) = \gcd(11, 4) = 1 \),
the conditions of Lemma~\ref{lemsub:simplyconn} are all satisfied, and
hence \( M \setminus D \) is simply connected.
Therefore, we have done all the tasks.
\end{exam}

\begin{exam}\label{exam:char2,K2=3}
Assume that \( \chara(\Bbbk) \ne 3 \) and we set \( K^{2} = 3 \).
We consider the same cubic pencil \( \Phi \) as in
Example~\ref{exam:char2,K2=1}. We also consider the same \( \bar{F} =
\bar{F}_{1} + \bar{F}_{2}\), but define
\[ \bar{S} :=\bar{S}_{1} + N\sptilde \quad \text{and} \quad
\bar{G} := \bar{G}^{+} - L_{1}\sptilde
= \sum\nolimits_{1 \leq i \ne 3 \leq 9} \Gamma_{i}.\]
Then, \eqref{condC:Phi}, \eqref{condC:nodal}, and \eqref{condC:NC1}
hold automatically. Moreover \eqref{condC:LinIndep} by Lemma~\ref{lem:determinant},
since the dual graph of \( \bar{G} \) is \( \SAA_{8} \) and
\( \det(\bar{G}_{i}\bar{G}_{j}) \not\equiv 0 \mod \chara(\Bbbk) \).
We have \eqref{condA:2} immediately from Figure~\ref{figure7}.
As in Example~\ref{exam:char2,K2=1},
\eqref{condC:NC2} and \eqref{condC:-1} hold on \( Z \).
But, we define here \( G \) to be \( S + F + J_{1} + J_{2} \).
Let \( \varphi \colon M \to Z \) be the birational morphism
such that the total transform \( B^{+}_{M} \)
of \( \bar{B}^{+}\)  is as in Figure~\ref{fig:K2=3char2}.
\begin{figure}[hbtb]
\begin{center}
\includegraphics[scale=0.95]{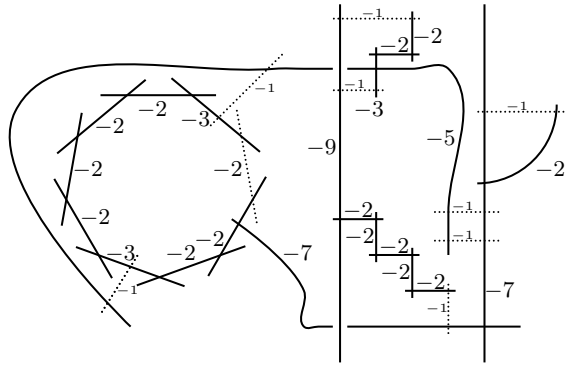}
\end{center}
\caption{The rational surface \( M \) in Example~\ref{exam:char2,K2=3}}
\label{fig:K2=3char2}
\end{figure}
Then, \( \rho(M) = \rho(Z) + 15 = 27 \) and \( K_{M}^{2} = -17 \).
Here, we have a disjoint union \( D \) of the following three
linear chains of smooth rational curves satisfying \eqref{condC:DB}:
\[ D_{1} = (2, 7, 7, 2, 2, 3, 2, 2, 2, 2, 3) \quad
D_{2} = (5, 3, 2, 2), \quad
D_{3} = (9, 2, 2, 2, 2, 2).\]
By Tables~\ref{table:M} and \ref{table:M2},
\( D_{1} \), \( D_{2} \), and \( D_{3} \) define the toric singularities of type
\( T(1, 63, 34) \), \( T(2, 4, 1) \), and \( T(1, 7, 1) \), respectively, and
\[ K_{X}^{2} = K_{M}^{2} + \delta(1, 63, 34) + \delta(2, 4, 1) + \delta(1, 7, 1) =
-17 + 11 + 3 + 6 = 3.\]
Thus, \eqref{condC:classT} and \eqref{condC:K2} hold.
The conditions \eqref{condC:end} and \eqref{condA:3} follow immediately
from Figure~\ref{fig:K2=3char2}.
We shall prove \eqref{condC:1+} by using Lemma~\ref{lem:condC:1+}.
Then, it suffices to check \( \Delta E_{i} > 0 \) for \( i = 1 \), \( 2 \), \( 3 \)
for the \( (-1) \)-curves \( E_{1} \), \( E_{2} \), and \( E_{3} \) on \( M \)
characterized by:
\begin{itemize}
\item  \( E_{1} \) joins the end \( (-2) \)-component of \( D_{2} \) and
the \( (-9) \)-curve in \( D_{3} \).

\item  \( E_{2} \) joins the end \( (-2) \)-component of \( D_{3} \) and
the third component of \( D_{1} \) which is a \( (-7) \)-curve.

\item  \( E_{3} \) joins the end \( (-2) \)-component of \( D_{1} \) and
the second component of \( D_{1} \) which is a \( (-7) \)-curve.
\end{itemize}
By Tables~\ref{table:M} and \ref{table:M2}, we can calculate
\begin{gather*}
\Delta E_{1} = (1 - 3/4) + (1 - 1/7) > 1, \quad
\Delta E_{2} = (1 - 6/7) + (1 - 1/63) > 1, \\
\Delta E_{3} = (1 - 34/63) + (1 - 5/63) > 1.
\end{gather*}
Hence, \eqref{condC:1+} holds.
The ampleness of \( K_{X} \) follows from
\eqref{condA:2} and \eqref{condA:3} by
Proposition~\ref{prop:step:M}.\eqref{prop:step:M:3}.
We shall show that \( M \setminus D \) is simply connected
when \( \Bbbk = \BCC \) by applying Lemma~\ref{lemsub:simplyconn}.
We have \( (d_{1}, d_{2}, d_{3}) = (1, 2, 1) \) and
\( (n_{1}, n_{2}, n_{3}) = (63, 4, 7) \) for the type
\( T(d_{i}, n_{i}, a_{i}) \) of the singularity defined by \( D_{i} \)
for \( 1 \leq i \leq 3 \). Looking at Figure~\ref{fig:K2=3char2},
we have the following two \( (-1) \)-curves:
\begin{itemize}
\item  the \( (-1) \)-curve \( E_{1} \) which meets end components of
\( D_{2} \) and \( D_{3} \).

\item  a \( (-1) \)-curve meeting the end \( (-3) \)-component of \( D_{1} \)
and the end \( (-5) \)-component of \( D_{2} \).
\end{itemize}
Since \( \gcd(d_{2}n_{2}, d_{3}n_{3}) = \gcd(8, 7) = 1 \) and
\( \gcd(d_{1}n_{1}, d_{2}n_{2}) = \gcd(63, 8) = 1\), the conditions of
Lemma~\ref{lemsub:simplyconn} are all satisfied,
and hence \( M \setminus D \) is simply connected.
Therefore, we have done all the tasks.
\end{exam}

\begin{exam}
    \label{exam:char2,K2=2}
Assume that \( \chara(\Bbbk) \ne 3 \) and we set \( K^{2} = 2 \).
We define
\[ \phi_{0} = \xtt^{3} + \ytt\ztt(\ytt + \ztt) \quad \text{and} \quad
\phi_{\infty} = 3\xtt\ytt\ztt.\]
Then, \( \Phi_{0} \) is a smooth cubic curve, and
\( \Phi_{\infty} = L_{1} + L_{2} + L_{3} \)
for the lines \( L_{1} = (\xtt)_{0} \), \( L_{2} = (\ytt)_{0} \), and
\( L_{3} = (\ztt)_{0} \).
In particular, \eqref{condC:Phi} holds.
Here, note that \( P_{3} = (0:0:1) = L_{1} \cap L_{2}  \) and
\( P_{2} = (0:1:0) = L_{1} \cap L_{3} \) are inflection points of
\( \Phi_{0} \), and
that \( L_{2} \) and \( L_{3} \) are the tangent lines of
\( \Phi_{0} \) at \( P_{3} \) and \( P_{2} \),
respectively (cf.\ Figure~\ref{figure8}).
\begin{figure}[hbtb]
\begin{center}
\includegraphics[scale=0.9]{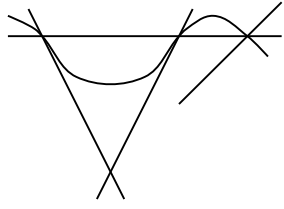}
\end{center}
\caption{A pencil of cubics} \label{figure8}
\end{figure}
The condition \eqref{condC:nodal} also holds. In fact,
for \( c \ne 0 \), the divisor \( \Phi_{c} = (\phi_{0} + c\phi_{\infty})_{0} \)
is singular if and only if
\( c = -\omega^{i} \) for \( i = 0 \), \( 1 \), \( 2 \),
where \( \omega \) is a primitive cubic root of \( 1 \), and
\( \Phi_{-\omega^{i}} \) is a nodal rational curve with the node at
\( (1 : \omega^{i} : \omega^{i}) \).
We define \( N \) to be the line passing through
\( (0:1:-1) \in \Phi_{0} \cap L_{1} \) and
the node \( (1:1:1) \) of \( \Phi_{-1} \). Thus,
\( N = (2\xtt - \ytt - \ztt)_{0} \) (cf.\ Figure~\ref{figure8}).
On the minimal elliptic fibration \( \pi \colon Y \to \BPP^{1}_{\Bbbk} \)
defined by \( \Phi \), the configuration type of singular fibers is
\( (\text{I}_{9}, \text{I}_{1}, \text{I}_{1}, \text{I}_{1}) \),
and the \( (-1) \)-curves exceptional for \( Y \to \BPP^{2}_{\Bbbk} \)
are mutually disjoint sections
\( \bar{S}_{1} \), \( \bar{S}_{2} \), and \( \bar{S}_{3} \) as in
Figure~\ref{figure9}.
\begin{figure}[hbtb]
\begin{center}
\includegraphics[scale=0.85]{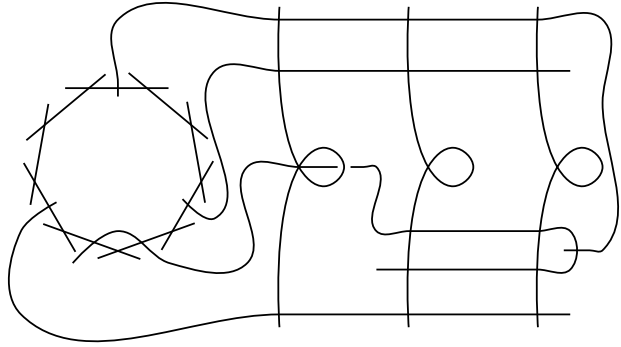}
\end{center}
\caption{The rational elliptic surface $Y$ in Example~\ref{exam:char2,K2=2}}
\label{figure9}
\end{figure}
We may assume that \( \bar{S}_{1} \), \( \bar{S}_{2} \), \( \bar{S}_{3} \)
are contracted to the points \( (0:1:-1) \in \Phi_{0} \cap L_{1} \),
\( P_{3} = (0:0:1) \), and \( P_{2} = (0:1:0) \),
respectively.
The three singular fibers of type \( \text{I}_{1} \) of \( \pi \)
are the proper transforms of \( \Phi_{-\omega^{i}} \)
for \( i = 0 \), \( 1 \), \( 2 \).
We set \( \bar{F}_{1} \) and \( \bar{F}_{2} \) to be the proper transforms
of \( \Phi_{-1} \) and \( \Phi_{-\omega} \), respectively, and define
\( \bar{F} = \bar{F}_{1} + \bar{F}_{2} \).
The irreducible components of the singular fiber of type \( \text{I}_{9} \)
are labeled as \( \Gamma_{0} + \Gamma_{1} + \cdots + \Gamma_{8} \)
as a cyclic chain of rational curves in such a way that
\( \Gamma_{0} = L_{1}\sptilde\), \( \Gamma_{4} = L_{2}\sptilde \), and
\( \Gamma_{5} = L_{3}\sptilde \), where \( L_{i}\sptilde \)
is the proper transform of \( L_{i} \) in \( Y \)
for \( i = 1 \), \( 2 \), \( 3 \).
Then,
\[ \bar{S}_{1}\Gamma_{0} = \bar{S}_{2}\Gamma_{3} = \bar{S}_{2}\Gamma_{6} = 1,
\quad \text{and} \quad \bar{S}_{j}\Gamma_{i} = 0\]
for other \( (i, j) \) with \( 0 \leq i \leq 8 \) and \( 1 \leq j \leq 3 \).
The proper transform \( N\sptilde \) of \( N \) in \( Y \) is a bisection
of \( \pi \) with self-intersection number zero
passing through the node of \( \bar{F}_{1} \) but
no nodes of the other singular fibers of type \( \text{I}_{1} \).
We have
\[ N\sptilde \bar{S}_{1} = N\sptilde \Gamma_{4} = N\sptilde \Gamma_{5} = 1, \quad
\text{and} \quad
N\sptilde \bar{S}_{j} = N\sptilde\Gamma_{i} = 0 \]
for \( 0 \leq i \ne 4, 5 \leq 8\) and \( j = 2 \), \( 3 \)
(cf.\ Figure~\ref{figure9}).
We define \( \bar{S} := \bar{S}_{3} + N\sptilde\).
Since the union \( \bar{G}^{+} \) of all the \( (-2) \)-curves on \( Y \)
is just the singular fiber of type \( \text{I}_{9} \),
we have \eqref{condC:NC1}.
We define
\[ \bar{G} := \bar{G}^{+} - \Gamma_{2} - \Gamma_{3} - \Gamma_{4}
= \sum\nolimits_{0 \leq i \ne 2, 3, 4 \leq 8} \Gamma_{i}. \]
We shall show \eqref{condC:LinIndep} without using Lemma~\ref{lem:determinant}.
Assume that
\[ a_{1}\bar{S}_{3} + a_{2}N\sptilde +
\sum\nolimits_{i = 0}^{8}m_{i}\Gamma_{i} \sim pH \]
for a Cartier divisor \( H \), integers \( a_{1} \), \( a_{2} \),
\( m_{0} \), \ldots, \( m_{8} \) with \( m_{2} = m_{3} = m_{4} = 0 \), where
\( p = \chara(\Bbbk) \).
Considering the intersection numbers with
\( \bar{S}_{1} \), \( \bar{S}_{3} \), \( N\sptilde \),
\( \Gamma_{1} \), \( \Gamma_{2} \), \( \Gamma_{5} \),
we have
\[
a_{2} + m_{0} \equiv -a_{1} + m_{6} \equiv m_{5}
\equiv m_{0} - 2m_{1} \equiv m_{1} \equiv a_{2} - 2m_{5} + m_{6} \equiv 0
\mod p.\]
Thus, \( a_{1} \equiv a_{2} \equiv m_{i} \equiv 0  \mod p\) for \( 0 \leq i \leq 6 \).
Moreover, we have \( m_{7} \equiv m_{8} \equiv 0 \mod p\)
by considering the intersection numbers
with \( \Gamma_{0} \) and \( \Gamma_{8} \).
Hence, \eqref{condC:LinIndep} holds.
The condition \eqref{condA:2} fails. In fact,
\( \Gamma_{4} \subset Y \setminus (\bar{S} \cup \bar{G})  \).
Thus, it is impossible to require \( K_{X} \) to be ample by
Proposition~\ref{prop:step:M}.\eqref{prop:step:M:3}.
The conditions \eqref{condC:NC2} and \eqref{condC:-1} on \( Z \)
follow from Figure~\ref{figure9} immediately.
We define \( B := S + G + F + J_{1} \).

We take the birational morphism \( \varphi \colon M \to Z \)
so that the total transforms of \( B + J_{2} \)  and \( \bar{G}^{+} \)
in \( M \) form a configuration of curves as in Figure~\ref{fig:K2=2char2}.
\begin{figure}[hbtb]
\begin{center}
\includegraphics[scale=0.95]{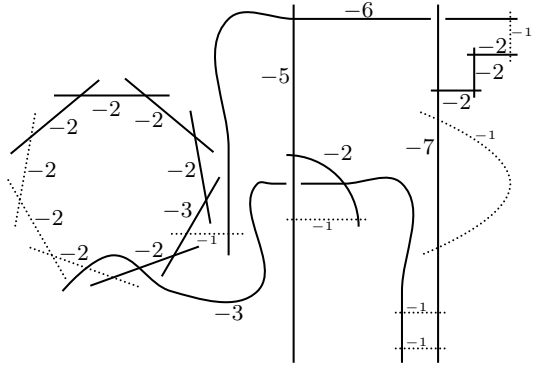}
\end{center}
\caption{The rational surface \( M \) in Example~\ref{exam:char2,K2=2}}
\label{fig:K2=2char2}
\end{figure}
Then, \( \rho(M) = \rho(Z) + 8 = 20 \) and \( K_{M}^{2} = -10 \).
Here, we have a disjoint union \( D \) of the following two
linear chains of smooth rational curves satisfying \eqref{condC:DB}:
\[ D_{1} = \LC(6, 5, 2, 3, 2, 3, 2, 2, 2, 2), \quad
D_{2} = \LC(7, 2, 2, 2).\]
By Tables~\ref{table:M} and \ref{table:M2},
\( D_{1} \) and \( D_{2} \) define toric singularities of type
\( T(3, 23, 4) \) and \( T(1, 5, 1) \), respectively, and
\[ K_{X}^{2} = K_{M}^{2} + \delta(3, 23, 4) + \delta(1, 5, 1) = -10 + 8 + 4 = 2.\]
Thus, \eqref{condC:classT} and \eqref{condC:K2} hold.
The condition \eqref{condC:end} follows immediately
from Figure~\ref{fig:K2=2char2}. We shall prove \eqref{condC:1+} using
Lemma~\ref{lem:condC:1+}. Then, it suffices to show \( \Delta E  > 1 \)
for the \( (-1) \)-curve \( E \) which joins the \( (-6) \)-curve of \( D_{1} \)
and the end \( (-2) \)-component
of \( D_{2} \).
By Tables~\ref{table:M} and \ref{table:M2}, we have
\[ \Delta E = (1 - 4/23) + (1 - 4/5) > 1. \]
Hence, \eqref{condC:1+} holds.
We shall prove the simply connectedness of \( M \setminus D \)
in case \( \Bbbk = \BCC \) by applying Lemma~\ref{lemsub:simplyconn}.
We have \( (d_{1}, d_{2}) = (3, 1) \) and \( (n_{1}, n_{2}) = (23, 5) \)
for the type \( T(d_{i}, n_{i}, a_{i}) \) of the singularity defined by
\( D_{i} \) for \( i = 1 \), \( 2 \).
Since the \( (-1) \)-curve \( E \)
meets end components of \( D_{1} \) and \( D_{2} \) and since
\( \gcd(d_{1}n_{1}, d_{2}n_{2}) = \gcd(3 \times 23, 5) = 1  \),
the conditions of Lemma~\ref{lemsub:simplyconn} are satisfied,
and hence \( M \setminus D \) is simply connected.
Therefore, we have done all the tasks.
\end{exam}

Finally, we shall prove our main result. We restate the result.

\begin{mainthm}
For any algebraically closed field \( \Bbbk \) and for any integer
\( 1 \leq K^{2} \leq 4 \), there exists an algebraically simply
connected minimal surface \( \BSS \)
of general type over \( \Bbbk \) with
\(p_{g}(\BSS) = q(\BSS) = \dim \OH^{2}(\BSS, \Theta_{\BSS/\Bbbk}) = 0 \)
and \( K_{\BSS}^{2} = K^{2} \) except \( (\chara(\Bbbk), K^{2}) = (2, 4) \),
where \( \Theta_{\BSS/\Bbbk} \) denotes the tangent sheaf.
Moreover, one can find such a surface with ample
canonical divisor when \( 1 \leq K^{2} \leq 4 \), except
\((\chara(\Bbbk), K^{2}) = (2, 1)\), \((2, 2)\), and \( (2, 4)\).
\end{mainthm}

\begin{proof}
Assume that \( (p, K^{2}) \ne (2, 4) \). Then,
we have an algebraically simply connected minimal
surface \( \BSS \) of general type defined over \( \Bbbk \) such that
\( p_{g}(\BSS) = q(\BSS) = \dim \OH^{2}(\BSS, \Theta_{\BSS/\Bbbk})
\linebreak 
= 0\)
and \( K_{\BSS}^{2} = K^{2} \) by Propositions~\ref{prop:LPmethod1}
and \ref{prop:LPmethod2} applied to
Examples~\ref{exam:K2=2Main}--\ref{exam:char2,K2=2} above.
In fact, the case \( K^{2} = 1 \) is treated in
Examples~\ref{exam:K2=1Main} (when \( p \ne 2  \)) and
\ref{exam:char2,K2=1} (when \( p \ne 3 \));
the case \( K^{2} = 2 \) in
Examples~\ref{exam:K2=2Main} (when \( p \ne 2\), \( 3 \)),
\ref{exam:K2=2char3} (when \( p \ne 2 \)), and
\ref{exam:char2,K2=2} (when \( p \ne 3 \));
the case \( K^{2} = 3 \) in Examples~\ref{exam:K2=3} (when \( p \ne 2 \))
and \ref{exam:char2,K2=3} (when \( p \ne 3 \));
and the case \( K^{2} = 4 \) in Example~\ref{exam:K2=4} (when
\( p \ne 2 \)).
Here, \( K_{X} \) is not ample only in Examples~\ref{exam:char2,K2=1} and
\ref{exam:char2,K2=2}. Thus, we can require \( K_{\BSS} \) to be ample
if \( (p, K^{2}) \ne (2, 1) \), \( (2, 2) \).
Hence, the proof
has been completed.
\end{proof}


\end{document}